\documentclass[reqno,12pt]{amsart}

\usepackage{amsfonts,amsthm,amsmath,amssymb,amscd,mathrsfs}
\usepackage{graphicx}
\usepackage{caption}
\usepackage{subcaption}
\usepackage{indentfirst}
\usepackage{cite}
\usepackage{enumerate}
\usepackage{latexsym}
\usepackage[dvips]{epsfig}
\usepackage{color}
\usepackage{bookmark}
\usepackage{lmodern}
\usepackage{soul}

\usepackage{lmodern}

\usepackage{extarrows} 
\usepackage{physics}
\usepackage{amsmath}
\usepackage{tikz}
\usepackage{mathdots}
\usepackage{yhmath}
\usepackage{cancel}
\usepackage{color}
\usepackage{siunitx}
\usepackage{array}
\usepackage{multirow}
\usepackage{amssymb}
\usepackage{gensymb}
\usepackage{tabularx}
\usepackage{extarrows}
\usepackage{booktabs}
\usetikzlibrary{fadings}
\usetikzlibrary{patterns}
\usetikzlibrary{shadows.blur}
\usetikzlibrary{shapes}

\newtheorem{theorem}{Theorem}[section]

\newtheorem{remark}{Remark}[section]
\newtheorem{definition}{Definition}[section]
\newtheorem{lemma}[theorem]{Lemma}
\newtheorem{pro}[theorem]{Proposition}

\renewcommand{\div}{{\rm div \thinspace }}

\newcommand{\bt}{\begin{theorem}}
	\newcommand{\bl}{\begin{lemma}}
		\newcommand{\el}{\end{lemma}}
	\newcommand{\et}{\end{theorem}}

\newcommand{\bn}{\begin{eqnarray}}
	\newcommand{\en}{\end{eqnarray}}
\newcommand{\bnn}{\begin{eqnarray*}}
	\newcommand{\enn}{\end{eqnarray*}}
\newcommand{\ba}{\begin{aligned}}
	\newcommand{\ea}{\end{aligned}}
\newcommand{\be}{\begin{equation}}
	\newcommand{\ee}{\end{equation}}

\newcommand{\Bu}{{\boldsymbol{u}}}
\newcommand{\Be}{{\boldsymbol{e}}}

\newcommand{\Bx}{{\boldsymbol{x}}}

\begin{document}

\title[Self-Similar Solutions with Physical Boundaries]{Self-Similar Solutions to the steady Navier-Stokes Equations in a two-dimensional sector}

	\author{Jeaheang Bang}
	\address{Institute for Theoretical Sciences, Westlake University, China}
	\email{jhbang@westlake.edu.cn}

	\author{Changfeng Gui}
	\address{Department of Mathematics, Faculty of Science and Technology, University of Macau, Taipa, Macao}
	\email{Changfenggui@um.edu.mo}
	
	\author{Hao Liu}
	\address{School of Mathematical Sciences and Institute of Natural Sciences, Shanghai Jiao Tong University, 800 Dongchuan Road, Shanghai, China}
	\email{mathhao.liu@sjtu.edu.cn}
	
	\author{Yun Wang}
	\address{School of Mathematical Sciences, Center for dynamical systems and differential equations, Soochow University, Suzhou, China}
	\email{ywang3@suda.edu.cn}

	\author{Chunjing Xie}
	\address{School of Mathematical Sciences, Institute of Natural Sciences,
		Ministry of Education Key Laboratory of Scientific and Engineering Computing,
		and CMA-Shanghai, Shanghai Jiao Tong University, 800 Dongchuan Road, Shanghai, China}
	\email{cjxie@sjtu.edu.cn}

\begin{abstract}
This paper is concerned with self-similar solutions of the steady Navier-Stokes system in a two-dimensional sector with the no-slip boundary condition. 
We give necessary and sufficient conditions in terms of the angle of the sector and the flux of the flow to guarantee the existence of self-similar solutions of a given type.
We also investigate the uniqueness and non-uniqueness of flows with a given type.
The non-uniqueness result is a new phenomenon for these flows.
 Our results not only give rigorous justifications for some statements in \cite{Rosenhead40} but also show that some numerical computations in \cite{Rosenhead40} may not be precise.  
  As a consequence of the classification of self-similar solutions in the half-space, we characterize the leading order term of the steady Navier-Stokes system in an aperture domain when the flux is small. 
   The main approach is to study the ODE system governing self-similar solutions, where the detailed properties of both complete and incomplete elliptic functions have been investigated.
\end{abstract}

\maketitle

\section{Introduction and Main Results}\label{secintroduction}
The main goal of this paper is to study the so-called self-similar solutions to the steady Navier-Stokes equations
    \begin{equation}
     \label{SNS}
     \left\{
     \begin{aligned}
          &     -\Delta\Bu + \Bu \cdot \nabla \Bu
         + \nabla p =0,
         \\
         &  \div \Bu=0  
        \end{aligned}
    \right.  
    \end{equation}
in a two-dimensional sector $K$ with angle $2\alpha$ ($0<\alpha<\pi$) centered at the origin,   supplemented with the no-slip boundary conditions
  \begin{equation}\label{NoslipBC}
  \Bu=0\quad \text{on}\,\, \partial K\setminus \{0\}.
  \end{equation} 
  
   In $\mathbb{R}^2\setminus \{0\}$, one can introduce the standard polar coordinates, 
\[
x_1=r \cos \theta, \quad x_2=r \sin\theta,
\]
and the orthogonal unit vectors
\[
\Be_r =(\cos \theta, \sin\theta), \quad \Be_\theta=(-\sin \theta, \cos\theta).
\]
Then in terms of polar coordinates, we can write the sector $K$ as
\begin{equation}\label{eq:sec}
     K =\{(r, \theta): r>0, ~ \theta \in (-\alpha, \alpha)\} \textrm{ where } 0<\alpha<\pi,
 \end{equation}
 and the velocity field could be written as $\Bu =u^r \Be_r +u^\theta \Be_\theta$. 

To determine a solution, we still need the so-called flux condition. It follows from the  no-slip boundary conditions \eqref{NoslipBC} and the divergence-free property ${\eqref{SNS}}_2$ of the flows that one can define the flux
\begin{equation}\label{eq:flux}
            \Phi:= \int _{-\alpha}^{\alpha} u^r(1, \theta) \, d\theta =  \int _{-\alpha}^{\alpha} u^r(r,\theta) \, r\, d\theta \quad \text{for all}\,\, r>0.
 \end{equation}

%     \begin{align}
%     \label{SNS}
%         -\Delta\Bu + \Bu \cdot \nabla \Bu
%         + \nabla p =0,
%         \quad
%         \div \Bu=0,
%     \end{align}
%    in a sector $K$, 
% with the no-slip boundary conditions, i.e., 
% \begin{equation}\label{BC}
%     \Bu=0\quad \text{on}\,\, \partial K\setminus \{0\},
% \end{equation}
% where in the polar coordinates
% \begin{equation}\label{sector}
%     K: =\{(r, \theta): r>0, ~ \theta \in (-\alpha, \alpha)\}
% \end{equation}
% when we are given a flux
%     \begin{align}
%     \label{eq:fluxcon}
%         \int _{-\alpha}^{\alpha} u^r \, d\theta= \Phi\in \mathbb{R}, \qquad u^r= \Bu\cdot \Be_r.
%     \end{align}

 It is well-known that for Navier-Stokes equations in any $n$-dimensional cone $\Omega$ which is invariant under the dilation $x\to\lambda x$, the following scaling property holds: if $(\Bu(x),p(x))$ is a solution in $\Omega$, then for every $\lambda>0$, the scaled one $(\lambda \Bu(\lambda x), \lambda^2 p(\lambda x))$ is also a solution in $\Omega$. Note that not only $\mathbb{R}^n\setminus\{0\}$, but also $\mathbb{R}^n_+\setminus\{0\}$ and more generally, an $n$-dimensional cone are invariant under the dilation. 
 Motivated by this scaling property, one may study the so-called self-similar (SS) solutions, i.e., the solutions satisfying
	$(\Bu, p)(x)=(\lambda \Bu(\lambda x), \lambda^2 p(\lambda x))$ in $\Omega$ for all $\lambda>0$.
For the study of self-similar solutions for steady Navier-Stokes equations and their applications, one may refer to \cite{Tian98, LiLiYan18, Sverak11, KorolevSverak11} and references therein.

 In $\mathbb{R}^2\setminus \{0\}$, the study of the self-similar solutions can be reduced to the
  study of solutions on the unit circle, i.e., the 
  $2\pi$-periodic solutions (as in \cite{Sverak11, GuillodWittwer15SIAM}). This leads to a classical equation whose solutions are  complete elliptic functions. 
All the self-similar solutions, also called Jeffery-Hamel solutions,  are found explicitly and classified, see \cite{Hamel17, Jeffrey15, Sverak11}. The precise results in these references can be summarized as follows.
\begin{theorem}(\cite{Hamel17, Jeffrey15, Sverak11})  Consider the Navier-Stokes equations \eqref{SNS} in $\mathbb{R}^2\setminus \{0\}$  together with \eqref{eq:flux} where $\alpha=\pi$.   \begin{enumerate}
  \item   For any $\Phi\in \mathbb{R}$, there is a class 
 of self-similar solutions given by
    \begin{align*}
        \Bu_{\Phi,0}
        =\frac{\Phi}{2\pi r} \Be_r + \frac{\mu}{r} \Be_\theta,
        \quad \text{for any}\,\, \mu  \in \mathbb{R},
    \end{align*}
   
\item     For $m=1, 2, \cdots$, if $\Phi\leq \Phi_{max}(m):=\pi(m^2-4)$, there is a self-similar solution of the form
    \begin{align*}
        \Bu_{\Phi,m}
        =\frac{1}{r} f(\theta+\theta_0) \Be_r, 
    \end{align*}
where $f$ is an elliptic function with the minimal period of $\frac{2\pi}{m}$,   and  $\theta_0$ is an angle that can be chosen arbitrarily. Furthermore, the function $f$ is uniquely determined by $\Phi $ and $m$.  If $\Phi>\Phi_{max}(m)$, then there is no self-similar solution with minimal period $\frac{2\pi}{m}$. 

\end{enumerate}
\end{theorem}

For self-similar solutions in a two-dimensional sector, a major new feature is that one has to impose a boundary condition on the physical boundary. In particular, we are interested in the no-slip boundary condition \eqref{NoslipBC}. 
Self-similar solutions in a sector with the no-slip boundary condition, such as jet flows and flows between inclined walls, have various physical applications and have attracted much attention; see \cite{Gusarov20, BanksDrazinZaturska88} and references therein.
% Moreover, self-similar solutions in a two-dimensional sector exhibit rich phenomena near the boundary, as shown below.  

However, the study of self-similar solutions in a sector is more involved and poses a significant difficulty due to the presence of physical boundaries.
%The presence of boundaries, however, poses a significant difficulty.
% The self-similar solutions in a sector do not necessarily t, periodic solutions.
% % which result in equations involving an incomplete elliptic function in general. 
% % In other words, there could be periodic solutions that are not $2\alpha$-periodic and solutions that have singularities outside the sector are also possible. 
% In addition, given a flux condition and a type of solution, one can obtain the uniqueness of a solution when the domain is $\mathbb{R}^2\setminus \{0\}$. However, as we found in this paper, non-uniqueness occurs for some types of solutions when the domain is a sector.
% These additional possibilities and new phenomena require a more refined analysis of solutions. 
In \cite{Rosenhead40}, Rosenhead studied the self-similar solutions in a sector satisfying the no-slip boundary condition via expressing the solutions in terms of the elliptic functions. 
% and then studied numerically when the flux was in a certain range. 
As an application, Fraenkel studied laminar flows in a channel with slightly curved walls in \cite{Fraenkel62}, by using the self-similar solutions studied in \cite{Rosenhead40} as leading order terms.
 Unfortunately, for the given angle of the sector and flux,  in \cite{Rosenhead40},  the existence of solutions was established by numerical computations rather than rigorous proof. A rigorous analysis of the properties of the flows mentioned in \cite{Rosenhead40, Fraenkel62} is the first objective of this paper; in particular, we give necessary and sufficient conditions for the existence of a non-trivial self-similar solution of a particular type.

    When the domain is the half-space, in \cite[Lemma 5.1] {GaldiPadulaSolonnikov96}, the authors proved the existence and uniqueness of a symmetric self-similar solution provided the given flux is small. For more studies on the self-similar solutions in a sector, one may refer to \cite{Kobayashi14, Kobayashi19, RivkindSolonnikov00} and references therein. 
%The existence and uniqueness were obtained under the smallness condition.
It is an interesting problem to get rid of the  restriction of the flux imposed in  \cite{GaldiPadulaSolonnikov96, Kobayashi14} because it is not expected from the numerical results of \cite{Rosenhead40}.

Now to state our results, 
 % let us first briefly describe the behavior of a solution that we will deal with. Solutions $\Bu$ to our problem will only have a radial component, that is, $\Bu=u^r\, \Be_r$. The function $f(\theta):=r\,u^r$ will satisfy the boundary condition $f(-\alpha)=f(\alpha)=0$ and oscillate between its maximum and minimum as like a standard sinusoidal function. (More rigorous explanations can be found in Section \ref{sec:2DC}.)
 % In general, given the angle $\alpha$ of the sector $K$ and flux condition $ \Phi$, there are infinitely many self-similar solutions $\Bu$ to \eqref{mainP}, which are distinguished by the number of times that $f(\theta)$ changes its sign.
 % With this behavior in mind, 
we first define the type of solutions in terms of the numbers of the outflow and inflow regions as follows.
    \begin{definition}\label{def:type}
         Let $(m_+,m_-)\neq (0,0)$ be a pair of non-negative integers,  and
         let $\Bu=u^r \Be_r$\footnote{It can be proved that a self-similar (SS) solution $\Bu$ to  \eqref{SNS} with the no-slip boundary condition \eqref{NoslipBC} must satisfy $u^{\theta}\equiv 0$; see Part (1) of Theorem \ref{2Dresult_2} and its proof in Section \ref{sec:2DC}.} be a non-trivial self-similar solution to \eqref{SNS}-\eqref{NoslipBC} in a two-dimensional sector ${K}$.
         The solution $\Bu$ is said to be \emph{of type $(m_+,m_-)$} if for the quantity $f(\theta):=r \,u^{r}$,
          the number of intervals of $\theta$ where $f(\theta)>0$ (outflow region) is $m_+$ and the number of intervals of $\theta$ where $f(\theta)<0$ (inflow region) is $m_-$.
In particular, a solution of type $(1,0)$ is called \emph{a pure outflow} whereas a solution of type $(0,1)$ is called \emph{a pure inflow.}
         %the number of interior local maximum points is $m_+$ and the number of interior local minimum points $m_-$. We say the trivial solution $\Bu \equiv 0$ is of any possible type $(m_+,m_-)$.
    \end{definition}

For example, Figures 1-3 give the illustration for $f(\theta)=ru^r$ with different types. 

\begin{figure}[h]
			\begin{center}
				\includegraphics[width=0.8\textwidth]{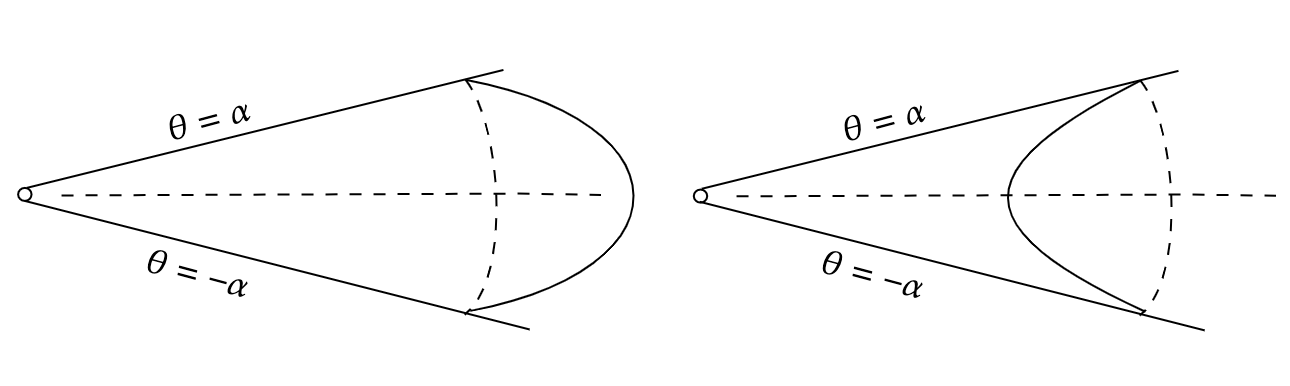}
			\caption{Pure outflow and pure inflow}
		\end{center}
  \end{figure}

 \begin{figure}[h]
			\begin{center}
				\includegraphics[width=0.8\textwidth]{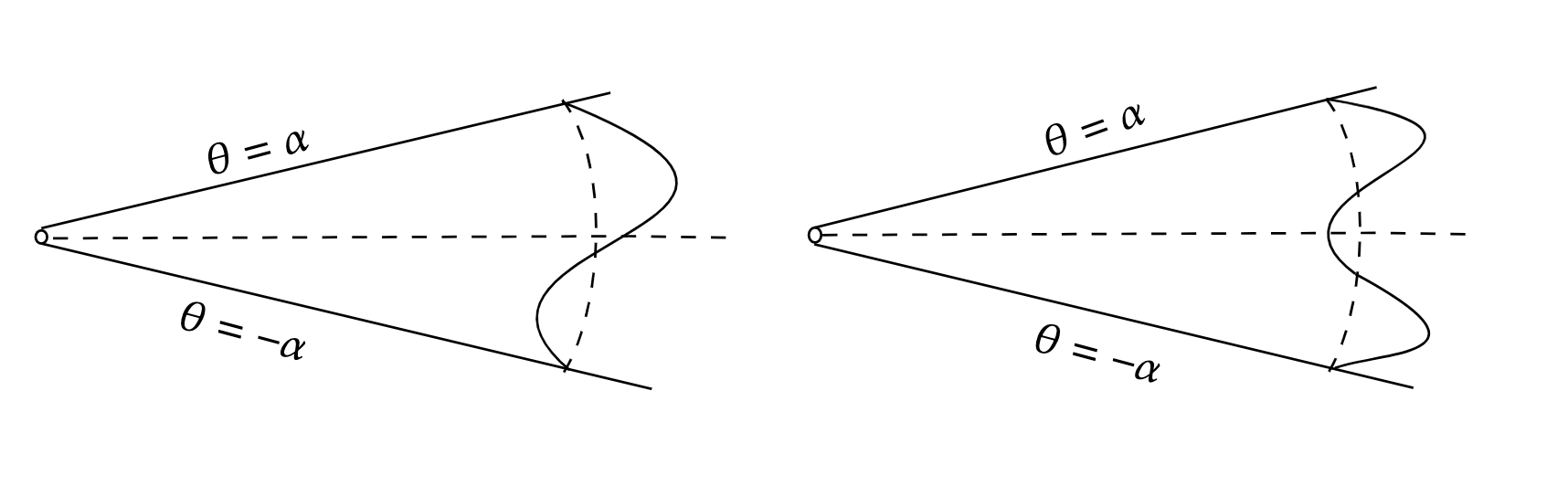}
			\caption{Type (1,1) and type (2,1) flow}
		\end{center}
  \end{figure}

  \begin{figure}[h]
			\begin{center}
				\includegraphics[width=0.8\textwidth]{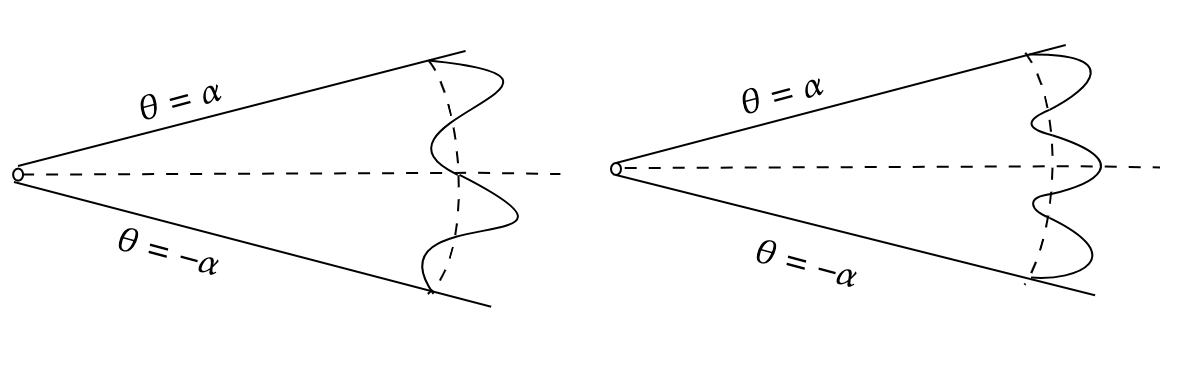}
			\caption{Type (2,2) and type (3,2) flow}
		\end{center}
  \end{figure}

Normally, the integers $m_{+}$ and $m_{-}$ satisfy that 
\begin{equation*}
    |m_{+}-m_{-}|\leq 1.
\end{equation*}
However, there are some special solutions of type $(m, 0)$, $m\geq 2$, which we will discuss later and can be viewed as the limiting cases of type $(m, m)$ flows.
In addition, note that we do not consider the trivial solution $\Bu \equiv 0$ in the definition above. 

Now, let us state our first main result, which is mainly about the existence and non-existence of a self-similar solution of each type. % when $\Phi\neq \Phi^{(m_+,m_-)}_{max}(\alpha)}$
    \begin{theorem}\label{2Dresult_2}
  Consider a self-similar (SS) solution $\Bu$ to the Navier-Stokes equations \eqref{SNS} with the no-slip boundary condition \eqref{NoslipBC} and the flux condition \eqref{eq:flux} in a sector $K$, which  is defined in \eqref{eq:sec} with angle $2\alpha$ ($\alpha\in (0, \pi)$). The following statements hold.
    \begin{enumerate}
    \item 
    $u^\theta \equiv 0$ for every SS solution $\Bu$.
    %every $\alpha\in (0, \pi), \Phi\in \mathbb{R}$,
    \item (Type $(1,0)$ flow)
    If $\frac{\pi}{2}\leq \alpha < \pi$,  there is no SS solution of type $(1,0)$. If $0<\alpha<\frac{\pi}{2}$, there exists a maximum flux $\Phi_{\textrm{max}}^{(1,0)}(\alpha)>0$ such that 
        \begin{enumerate}[(a)]
            \item 
    if $0<\Phi< \Phi^{(1,0)}_{max}(\alpha)$, then there exists an SS solution of type $(1,0)$; 
        \item 
        if $\Phi> \Phi^{(1,0)}_{max}(\alpha)$, then there exists no SS solution of type $(1,0)$.
         \end{enumerate}
    \item 
    (The other types except for $(1,0)$ and $(2,1)$)
 {Let  $(m_+,m_-)$ be a pair of non-negative integers such that $|m_+-m_-|\leq 1$ and $(m_+,m_-)\not =(0,0),(1,0), (2,1)$. For any $\alpha \in (0, \pi)$,
there exists a maximum flux $\Phi_{max}^{(m_+,m_-)}(\alpha)\in \mathbb{R} $ such that }
    \begin{enumerate}[(a)]
        \item 
    if $\Phi < \Phi_{max}^{(m_+,m_-)} (\alpha) $, then there exists an SS solution of type $(m_+,m_-)$;
        \item 
        if $\Phi> \Phi^{(m_+,m_-)}_{max}(\alpha)$, then there exists no SS solution of type $(m_+,m_-)$.
   \end{enumerate}
   \item  For $(m_+,m_-)=(2,1)$, if $0<\alpha\leq \frac{\pi}{2}$, then there exists a maximum flux $\Phi^{(2,1)}_{max}(\alpha)$ such that the same conclusions as Part (3) hold.
%     \item
% If $\Phi= \Phi^{(0,1)}_{max}(\alpha), \frac{\pi}{2}<\alpha<\pi$, then there exists an SS solution of type (0,1). 
%           If $\Phi=\Phi^{(m,m)}_{max}(\alpha),m\geq 2, 0<\alpha<\pi$, then there exists an SS solution of type $(m,0)$.
%     \item The SS solution of type $(m, m), m\geq 1$ is unique for every $\Phi < \Phi_{max}^{(m, m)} (\alpha)$. Moreover, the SS solution of type $(1, 0)$  is also unique for every $0< \Phi < \Phi_{max}^{(1, 0)}(\alpha)$. 
    \end{enumerate}
    \end{theorem}

\begin{remark}
In Theorem \ref{2Dresult_2}, we did not impose any smallness restriction on flux, and this theorem pertains to non-symmetric flows as well as symmetric ones (symmetric with respect to the middle line $\theta=0$). In fact, the flows of type $(m_+, m_-)$ obtained in Theorem \ref{2Dresult_2}  are non-symmetric if and only if $m_+ =m_-$.
\end{remark}
\begin{remark}
In addition, we have established the existence of type $(m_{+}, m_{-})$ flows for the special case $\Phi=\Phi^{(m_+,m_-)}_{max}(\alpha)$, which  is studied in Section \ref{sum} (cf. Theorem \ref{limitc}).
\end{remark}
 {
\begin{remark}
 Indeed, we can show that there is an
 $\alpha^*\geq\frac{\pi}{2}$ which is implicitly defined in \eqref{eq:alpha^*}, such that
for $0<\alpha\leq \alpha^*$, Part (4) in Theorem \ref{2Dresult_2} still holds; see Proposition \ref{Protype21} for the detailed explanations.
%for every flux below this maximum flux, there is a corresponding $(2,1)$ flow. 
 % We can prove the existence of the maximum flux that can be attained by  $(2,1)$ flows, but 
 %We can not prove that for every flux below this maximum flux, there is a corresponding $(2,1)$ flow when $\alpha^*<\alpha<\pi$.
  Numerical computations show that $\alpha^*\geq 2.232$. However, its exact value is unknown.
\end{remark}
}
We also prove various properties of the maximum fluxes and uniqueness or non-uniqueness of some types of flows.  
\begin{theorem}
    \label{main2_2}
    First, the following properties of $\Phi^{(m_+,m_-)}_{max}$ hold.
            \begin{enumerate}
            \item $\Phi_{max}^{(1,0)}(\alpha)=  \frac{3\pi}{\alpha} + o(\frac{1}{\alpha})$ as $\alpha\to 0$ and 
            $\Phi^{(1,0)}_{max} (\alpha) = 8\left(\frac{\pi}{2}-\alpha\right)  + o \left(\frac{\pi}{2}-\alpha \right) $ as $\alpha \to \frac{\pi}{2}$.
            
            \item $\Phi^{(1,1)}_{max} (\alpha)$ is decreasing with respect to $\alpha$;
            $\Phi^{(1,1)}_{max} (\alpha)=\Phi_{max}^{(1,0)}(\alpha)$ for $0<\alpha<\frac{\pi}{2}$ and 
            $\Phi^{(1,1)}_{max} (\alpha)=\Phi_{max}^{(0,1)}(\alpha)$ for $\frac{\pi}{2}\leq\alpha<\pi$.

            \item $\Phi_{max}^{(m,m)}(\alpha)= m \Phi_{max}^{(1,1)} ( \frac{\alpha}{m})
            = m \Phi_{max}^{(1,0)} ( \frac{\alpha}{m}), m\geq 2$.
            \item 
           
            $\Phi^{(m,m)}_{max}(\alpha) 
            < \Phi^{(m,m+1)}_{max} (\alpha)
            \leq \frac{m}{m+1} \Phi^{(m+1,m+1)}_{max} (\alpha) < \Phi_{max}^{(m+1, m+1)}(\alpha) 
            \leq \Phi_{max}^{(m+1, m)}(\alpha)$  for any $m\geq 1$.

            \end{enumerate}     
   Second, the following uniqueness results hold.
            \begin{enumerate}
            \item The SS solution of type $(1, 0)$ (pure outflow)  is unique for every $0< \Phi < \Phi_{max}^{(1, 0)}(\alpha)$. 

            \item The SS solution of type $(0, 1)$ (pure inflow) is unique for every $\Phi < \Phi_{max}^{(0, 1)}(\alpha).$
            
            \item The SS solution of type $(m, m), m\geq 1$ is unique for every $\Phi < \Phi_{max}^{(m, m)} (\alpha)$.
            \end{enumerate}
  Finally, we have the following non-uniqueness of solutions.
            \begin{enumerate}
            
            \item For every $ \alpha \in (0, \frac{\pi}{2}]$ and  every $\Phi\in (\Phi_{max}^{(1, 1)}(\alpha), \Phi_{max}^{(1, 2)}(\alpha))$, there exist at least two SS solutions of type $(1, 2)$.

            \item For every $ \alpha \in (0, \pi)$ and every $\Phi\in (\Phi_{max}^{(m, m)}(\alpha),  \Phi_{max}^{(m, m+1)} (\alpha))$ with $m\geq 2$, there exist at least two SS solutions of type $(m, m+1)$. 
            \end{enumerate}
    \end{theorem}

\begin{remark}
Additional properties of $\Phi^{(m_+,m_-)}_{max}(\alpha)$, such as the estimate for the magnitude of $\Phi^{(0,1)}_{max}(\alpha)$, etc., can also be found in Section \ref{sum} (cf. Theorem \ref{limitc}).    
\end{remark}

\begin{remark}
Most of the facts in Theorem \ref{main2_2} are consistent with the numerical results in \cite{Rosenhead40}. However, Theorem \ref{main2_2} together with Theorem \ref{limitc} reveals that some numerical computations in \cite{Rosenhead40} may not be precise. For example, $\Phi_{max}^{(1, 1)}(\frac{\pi}{2})$ should be $0$, not ``0.5" given in \cite{Rosenhead40}; $\Phi_{max}^{(m, m+1)}(\alpha)$ should be greater than $\Phi_{max}^{(m, m)}(\alpha)$, which contradicts with the numerical results in \cite{Rosenhead40}. 
\end{remark}

\begin{remark}
 There is a more general class of scaling invariant solutions, rotated self similar solutions, 
which satisfy
\[
\Bu(x)= \lambda R(-2\beta \log \lambda) \Bu(\lambda R(2\beta \log \lambda) x)\quad \text{for all }\lambda>0
 \text{  and some  }\beta \in \mathbb{R},
\]
where $R$
% \[
% R(\beta)=\begin{pmatrix}
%     \cos\beta & -\sin\beta\\
%     \sin\beta &\cos\beta
% \end{pmatrix}
%\]
is a rotation matrix on $\mathbb{R}^2$.  It is worth mentioning that on $\mathbb{R}^2\setminus \{0\}$, Guillod and Wittwer \cite{GuillodWittwer15SIAM} constructed and characterized all the rotated self-similar solutions which are not necessarily self-similar. However, in a two-dimensional sector $K$, there is no definition of such a rotated self-similar solution.
%due to the no-slip boundary condition \eqref{NoslipBC}, %it is not hard to see all the rotated self-similar %solutions must be self-similar.  
\end{remark}

Self-similar solutions not only are a very important class of solutions by themselves, but also play a very important role in understanding the asymptotic behavior of general solutions to the stationary Navier-Stokes equations at far fields or near singularities, see \cite{KorolevSverak11, BangGuiLiuWangXie23, JiaSverak17, MiuraTsai12} and references therein. Here we give an application of Theorems \ref{2Dresult_2} and \ref{main2_2}.
It was proved in \cite{GaldiPadulaSolonnikov96}  that when the flux is small, there exists  a unique solution of the Navier-Stokes system in the aperture domain
\begin{equation}\label{defapture}
    \Omega :=\{ \Bx=(x_1,x_2) \in \mathbb{R}^2:\ \mbox{either}\ x_1 \neq 0 \ \ \mbox{or} \ x_2\in (-d, d), d>0 \}\footnote{The definition of the aperture domain $\Omega$ here differs from the definition in \cite{GaldiPadulaSolonnikov96} by exchanging $x_1$ and $x_2$.}
\end{equation}
supplemented no-slip boundary condition
\begin{equation}\label{noslipaperture}
    \Bu=0\quad \text{on}\,\,\partial \Omega.
\end{equation}
The flows decay at large distances like $|\Bx|^{-1}$, whose leading order terms are certain self-similar solutions of the steady Navier-Stokes system; {see Theorem \ref{TheromeGPS} for the precise statement}.  However, the exact types of those leading order self-similar solutions are not identified in \cite{GaldiPadulaSolonnikov96}. For the further progress, one may refer to \cite{BorchersPileckas92, BorchersGaldiPileckas93, NazarovSequeiraVideman01, NazarovSequeiraVideman02, Galdi11}.

With the aid of the classifications of the self-similar solutions obtained in Theorems \ref{2Dresult_2} and \ref{main2_2}, we can give the precise description for the leading order terms of the solutions at far fields in the aperture domain. 
\begin{theorem}\label{main3}
Let the aperture domain $\Omega$ be given by \eqref{defapture}.
There exists  $\Phi_0>0$ such that if $ \Phi\in (-\Phi_0,\Phi_0)$, then the unique solution obtained in \cite[Theorem 5.1]{GaldiPadulaSolonnikov96} of the problem \eqref{SNS} and \eqref{noslipaperture} with given flux $\Phi$ defined by
\[
\Phi:=\int_{-d}^d u_1(0, x_2)\, dx_2,
\]
exhibits the following behavior: its leading order term at the upstream is a type $(0,1)$ self-similar solution, while at the downstream, the leading order term is a type $(1,2)$ self-similar solution.
\end{theorem}

 % \begin{remark}
 %     If $|\Phi|$ is sufficiently small, then there exists a unique solution with flux $\Phi$ for the Navier-Stokes system \eqref{SNS} under the no-slip boundary condition \eqref{noslipaperture}.
 % \end{remark}

\begin{remark}
For $\Phi > 0$, the upstream flows in $\Omega$ correspond to the region where $x_1 \to -\infty$, while the downstream flows correspond to the region where $x_1 \to +\infty$. For $\Phi < 0$, these two regions are interchanged.
  \end{remark}

\begin{remark}
    The reason why the leading order term at the downstream is of type $(1,2)$, not $(1,0)$, should be related to the boundary layer separation. See \cite[p. 4-5]{Abernathy70} and also \cite[Introduction]{Fraenkel62}.
\end{remark}

Our main idea is to study the Navier-Stokes equations on the arc $(\sin \theta, \cos \theta), \theta\in (-\alpha,\alpha)$, under the assumption of self-similarity. The key ingredient is to rewrite the flux condition and angle condition in terms of the complete and incomplete elliptic functions, whose behaviors are analyzed in detail.

The rest of the paper is organized as follows. 
In Section \ref{sec:2DC}, we reduce the problem for the Navier-Stokes system into an ODE problem by using the self-similar assumption. The outflows, periodic flows, and inflows are studied in Sections \ref{Sectionpureoutflow}-\ref{Sectionpureinflow}, respectively.  The flows of other types are investigated in Section \ref{sec:Flows of other types}. Section \ref{sum} is devoted to the special case $\Phi=\Phi^{(m_+,m_-)}_{max}(\alpha)$ and the additional properties of $\Phi^{(m_+,m_-)}_{max}(\alpha)$.
Finally, we study the asymptotic behavior of solutions at infinity of an aperture domain in Section \ref{sectionaperturedomain}. 

\section{Reduction of the problem to an ODE system} \label{sec:2DC}
    
In this section, we reduce the study of self-similar solutions $\Bu$ to the Navier-Stokes equations \eqref{SNS} with no-slip boundary condition \eqref{NoslipBC} and flux condition \eqref{eq:flux} in the sector $K$ to a boundary-value problem for ODE with an integral constraint. 
% For the half space $\mathbb{R}^2_+$, the domain has boundary as opposed to $\mathbb{R}^2\setminus \{0\}$. So we need to impose a boundary condition. In particular, we will impose the vanishing Dirichlet boundary condition except at the origin:
%     \begin{align*} 
%         \Bu
%         |_{(\partial \mathbb{R}^2_+) \setminus \{0\}}=0.
%     \end{align*}
For self-similar solutions, one can  write
    \begin{align} \label{SSdec}
        \Bu(x)= \Bu(r,\theta)= u^r \Be_r +u^\theta \Be_\theta = 
        \frac{f(\theta)\Be_r + g(\theta) \Be_\theta}{r}, ~ ~ ~\ \ p(x) =p(r,\theta) = \frac{p(\theta)}{r^2}   \end{align}
    with an abuse of the notation of $p$.  
%Our main goal is to study non-trivial self-similar solutions to \eqref{SNS}-

From the equations in \eqref{SNS}, $f$ and $g$ should satisfy the following  equtions:
    \begin{align}
    \label{SSNS}
    \left\{
    \begin{aligned}
        &(p-2f)'=0,
        \\
        &-f''+gf'-f^2 -|g|^2-2p=0,
        \\
        &g'=0,
    \end{aligned}
    \right.
    \end{align}
where $'$ denotes the differentiation with respect to $\theta$.  
These equations hold in the interval $(-\alpha,\alpha)$, and the no-slip boundary condition in \eqref{NoslipBC} becomes
    \begin{align*}
        f(-\alpha)=f(\alpha)=0
        \quad \text{and}\quad 
        g(-\alpha)=g(\alpha)=0.
        \label{BC3}
    \end{align*}
From  $\eqref{SSNS}_3$ and the boundary condition of $g$ above, it follows that $g\equiv 0$ and hence $u^\theta=0$. {\bf This proves Part (1) in Theorem \ref{2Dresult_2}.} In addition, using $\eqref{SSNS}_1$, one can find
    \begin{align*}
        p=2f-\frac{b}{2}
    \end{align*} 
    for some constant $b$,
and $\eqref{SSNS}_2$ becomes
    \begin{align*}
        f''=-f^2 -4f+b.
    \end{align*}
This is a classical equation defining elliptic functions. 
% subject to the boundary condition${\eqref{eq:ODEsys}}_2$.
 The flux condition \eqref{eq:flux} now becomes 
 \begin{equation*}
      \int_{-\alpha}^{\alpha} f(\theta) d \theta = \Phi.
 \end{equation*}
In summary, the study of the self-similar solutions to the  problem \eqref{SNS}, \eqref{NoslipBC}, and \eqref{eq:flux} is reduced to the study of the following ODE problem
\begin{equation}\label{eq:ODEsys}
\left\{
\begin{aligned}
&f''=-f^2 -4f+b,\\
&f(-\alpha)=f(\alpha)=0,\\
   &  \int_{-\alpha}^{\alpha} f(\theta) d \theta = \Phi.
\end{aligned}
\right.
 \end{equation}
% {\red
% Most likely, we will not need this red part at all, but let's keep it here just in case.

% One can prove the following theorem by using \cite[Theorem 1.1]{GuillodWittwer15SIAM}\cite[Theorem 2]{Sverak11} without difficulties.

% \begin{theorem} 
% Let $\alpha\in(0,\pi)$ and assume $\alpha=\frac{k}{n}\pi$ for some integers $n\geq 2, 1\leq k \leq n-1$. Then there exists a non-trivial $2\alpha$-periodic function $f$ in $\mathbb{R}$ with the minimal period of $\frac{2\pi}{n}\left(=\frac{2\alpha}{k}\right)$ such that
% the restriction $f|_{[-\alpha,\alpha]}$ satisfies the equation ${\eqref{eq:ODEsys}}_1$ and boundary condition${\eqref{eq:ODEsys}}_2$ and
%     \begin{align*}
%         \int_{-\alpha}^{\alpha} f(\theta) \,d\theta=0.
%     \end{align*}
% \end{theorem}
% This way, one can create infinitely many non-trivial solutions easily for a fixed $\alpha$ if $\alpha=k\pi/n$ for such integers $n,k$. For example, $\alpha=\pi/2$, which corresponds to the half-space, satisfies this condition for $k=\frac{n}{2}, n\geq 2$. 
% \begin{proof}
%     Let $\tilde{f}$ be the elliptic function determined in \cite[Theorem 1.1]{GuillodWittwer15SIAM}, \cite[Theorem 2]{Sverak11} with the zero flux with the minimal period of $2\pi/n$. Then $\tilde{f}$ is $2\alpha$-periodic, and there exists $\tilde{\theta}\in \mathbb{R}$ such that $\tilde{f}(\tilde{\theta})=0$ as the flux is zero. Define $f(\theta)=f(\theta+\tilde{\theta}+\alpha)$. Then it satisfies the boundary condition${\eqref{eq:ODEsys}}_2$, completing the proof.
% \end{proof}
% }

To study the ODE problem \eqref{eq:ODEsys}, 
we multiply ${\eqref{eq:ODEsys}}_1$ by $2f^\prime$, then integrate both sides with respect to $\theta$ to have
\begin{align}\label{eq:f'square}
    (f^\prime)^2 = -\frac23 f^3 - 4f^2 + 2bf + 2E := Q(f)=  -\frac23 (f- e_1)(f-e_2) (f-e_3),
\end{align}
where $E$ is a constant.
Here $e_1$, $e_2$, $e_3$ are  roots of $Q(f)$, which may be complex. Moreover, it is easy to note that they satisfy the constraint
\begin{align}\label{roots}
    e_1 + e_2 + e_3 = -6.
\end{align}

\section{Pure outflow-Type $(1, 0)$ flow}\label{Sectionpureoutflow}

In this section, we consider the pure outflow or type (1,0) flow, i.e., the flow that satisfies $f(\theta)>0$ when $\theta \in (-\alpha, \alpha)$. 
The main result in this section is as follows.
\begin{pro}\label{Lemmapureoutflow2}
% Consider the Navier-Stokes equations \eqref{SNS} in $K$ with conditions \eqref{BC}, \eqref{eq:fluxcon}. 
Consider a self-similar (SS) solution of type (1,0) to the Navier-Stokes equations \eqref{SNS} with no-slip boundary condition \eqref{NoslipBC} and the flux condition \eqref{eq:flux}.
\begin{enumerate}
\item 
If the angle $\alpha \geq \frac{\pi}{2}$, there does not exist SS solutions of type (1,0).
\item 
For any $0<\alpha < \frac{\pi}{2}$, there exists a maximum flux $\Phi_{max}^{(1,0)}(\alpha)>0$ defined in \eqref{eq:maxifluxpureout} below, such that the following dichotomy holds: for every $0<\Phi \leq  \Phi_{max}^{(1, 0)} (\alpha)$, there exists a unique SS solution of type (1,0); on the other hand, if $\Phi > \Phi_{max}^{(1, 0)}(\alpha)$, there is no SS solution of type (1,0). 
\item 
Moreover, we have the following asymptotic behavior of
  $\Phi_{max}^{(1,0)}(\alpha)$ as $\alpha \to 0$ and $\frac{\pi}{2}$: $\Phi_{max}^{(1,0)}(\alpha)=  \frac{3\pi}{\alpha} + o(\frac{1}{\alpha})$ as $\alpha\to 0$ and
            $\Phi^{(1,0)}_{max} (\alpha) = 8\left(\frac{\pi}{2}-\alpha\right)  + o \left(\frac{\pi}{2}-\alpha \right) $ as $\alpha \to \frac{\pi}{2}$.
\end{enumerate}
\end{pro}

Clearly, Proposition \ref{Lemmapureoutflow2} gives Part (2) of Theorem \ref{2Dresult_2} as well as Part (1) of Theorem \ref{main2_2}.
 As discussed in Section \ref{sec:2DC}, we only need to solve the ODE problem \eqref{eq:ODEsys}. 
 
 For the ODE problem \eqref{eq:ODEsys},
our starting point is the following simple analysis of the roots of $Q(f)$. There are two possible cases. One is that there is one real root and two complex conjugate roots. In this case, without loss of generality, we let $e_1$ be the real root and $e_2, e_3 $ be the complex conjugate roots. 
The other case is that all three roots are real, and without loss of generality, we can adjust the order of $e_1, e_2$ and $e_3$ so that
$e_3 \leq e_2 \leq e_1$.

\begin{lemma}\label{lem:roots}
For a pure outflow, we have  either $e_1>0$ and $e_2, e_3 $ are complex conjugates or $e_1>0 \geq  e_2 \geq e_3$ and $e_3<0$. 
\end{lemma}
\begin{proof}
We first show that $e_1>0$.    Suppose $e_1 \leq 0$. Note the cubic polynomial $Q(f)$
satisfies  $Q(f) < 0$ when $f > e_1$, and for the pure outflow $f>0 \geq e_1$. Hence the identity \eqref{eq:f'square}, i.e., $Q(f)=(f')^2\geq 0$ leads to a contradiction. 
 % Suppose $e_1 = 0$, the above argument tells that to ensure $Q(f)\geq 0$, $f$ must be zero solution, which is not the non-trivial pure outflow. 
 Due to the same reason, one can easily see that if  $e_3 \leq e_2 \leq e_1$ are all real, then $e_3\leq e_2\leq0$.
 Indeed,  the fact that $e_1 + e_2 + e_3 =-6 <0$ implies $e_3 <0.$ Now we prove that $e_2\leq 0$.  Suppose $e_2 >0$, then $Q(0)<0$. Then due to the boundary condition${\eqref{eq:ODEsys}}_2$, $Q(f(\pm \alpha))=Q(0)<0$. This contradicts with $Q(f)=(f')^2\geq 0$. 
\end{proof}

The possible plots for $Q(f)$ are presented in Figures  \ref{one real root} and \ref{three real roots}.
{Then for a pure outflow, the value of $f(\theta)$ increases from $0$ to $e_1$ and then decreases from $e_1$ to $0$. And the solution $f$ is symmetric with respect to the half line $\theta=0$; see Figure \ref{pureout}.}

\begin{figure}[h]
\begin{subfigure}{0.3\textwidth}
     \centering
   \includegraphics[width=0.8\textwidth]{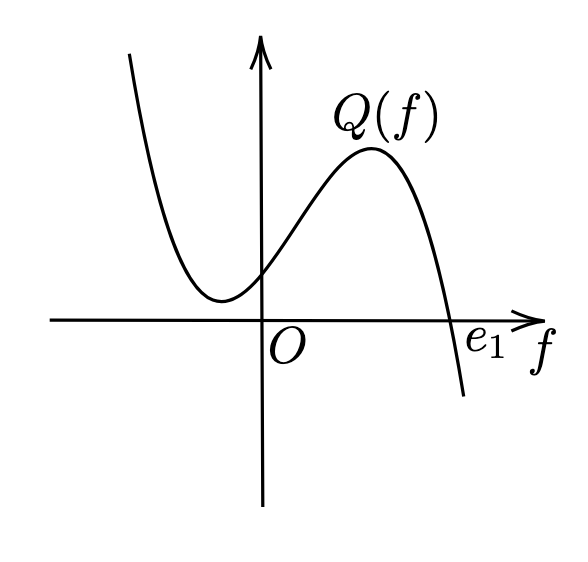}
			\caption{$Q(f)$ has one real root}
   \label{one real root}
\end{subfigure}
\begin{subfigure}{0.3\textwidth}
     \centering
   \centering
    \includegraphics[width=0.8\textwidth]{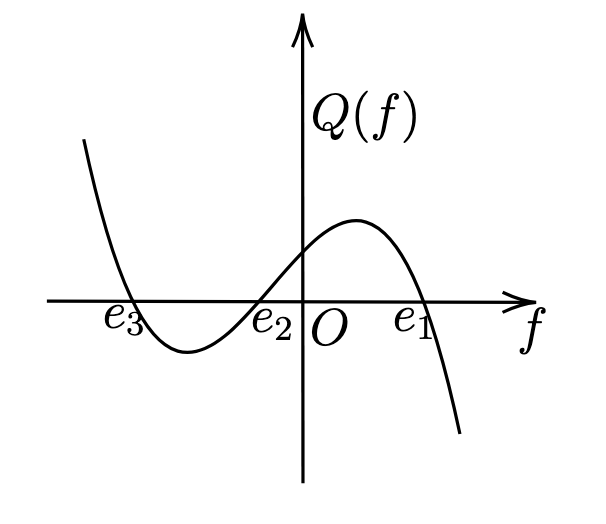}
			\caption{$Q(f)$ has three real roots}
   \label{three real roots}
\end{subfigure}
  \begin{subfigure}{0.3\textwidth}
     \centering
    \centering
   \includegraphics[width=0.8\textwidth]{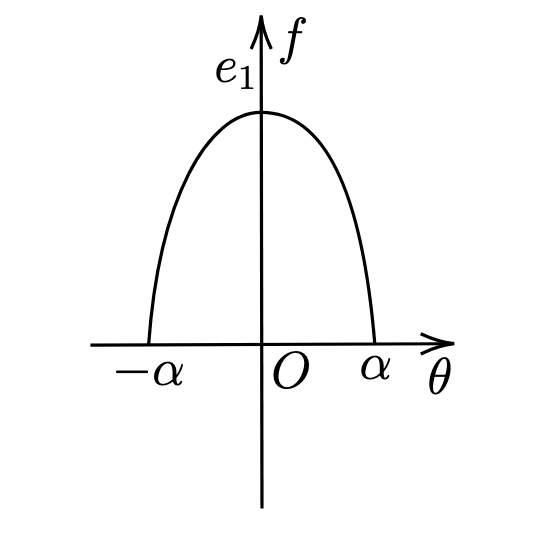}
			\caption{Velocity of pure outflow}
   \label{pureout}
\end{subfigure}
\caption{Plots for $Q(f)$ and a plot for velocity of pure outflow}
\end{figure}

% \begin{figure}[h]
% 			\begin{center}
% 				\includegraphics[width=0.3\textwidth]{4.jpg}
% 			\caption{$Q(f)$ has one real root}
% 		\end{center}
%   \end{figure}
% \begin{figure}[h]
% 			\begin{center}
% 				\includegraphics[width=0.3\textwidth]{5.jpg}
% 			\caption{$Q(f)$ has three real roots}
% 		\end{center}
%   \end{figure}
% \begin{figure}[h]
% 			\begin{center}
% 				\includegraphics[width=0.3\textwidth]{5.jpg}
% 			\caption{Pure outflow}
% 		\end{center}
%   \end{figure}

The above fact together with \eqref{eq:f'square} tells  that given the angle $\alpha$ and the flux $\Phi$, we should have the following two equations,
 \begin{align}\label{angle}
 I_{+}(e_1, e_2): = \int_0^{e_1} \frac{df}{\sqrt{Q(f)}} = \int_0^{e_1} \frac{df}{\sqrt{-\frac23 (f - e_1)(f-e_2) (f-e_3)} }  = \alpha,
  \end{align}
  \begin{align}\label{flux}
J_{+}(e_1, e_2) :=   \int_0^{e_1} \frac{f df}{\sqrt{Q(f)}} = \int_0^{e_1} \frac{f\, df }{\sqrt{-\frac23 (f - e_1)(f-e_2) (f-e_3)} }  = \frac{\Phi}{2}, 
  \end{align}
with $e_3 = -6-e_1 -e_2. $

We first prove the following lemma, which is Part (1) of Proposition \ref{Lemmapureoutflow2}.
  \begin{lemma}\label{nonexistence} When the angle $\alpha \geq \frac{\pi}{2}$, there exists no pure outflow.       
  \end{lemma}

\begin{proof}
The analysis is divided into two cases, according to the type of roots. 

{\bf Case 1}. $e_1$ is real, $e_2$ and $e_3$ are complex conjugates. It follows from  \eqref{roots} that one has
\begin{align}\nonumber
e_2= -3 -\frac12 e_1 + ci, \ \ e_3 =-3-\frac12 e_1 - ci,\ \ \ c>0.
\end{align}
%Let
%\begin{align}
  %  I_{+}(e_1, e_2) = \int_0^{e_1} \frac{1}{\sqrt{-\frac23 (f-e_1) (f-e_2) (f+6 + e_1 + e_2) }} \, df . 
%\end{align}
In this case, 
\begin{align*}
I_{+}(e_1, e_2) = \int_0^{e_1} \frac{df }{\sqrt{-\frac23 (f-e_1) \left[ (f+ 3 + \frac12 e_1)^2 + c^2 \right]}} .
\end{align*}
In this situation, with $e_1$ fixed, $I_+(e_1, e_2)$ attains its maximum when $c=0$, which is 
\begin{align*}
    I_{+}\left(e_1, -3 - \frac12 e_1\right)= \int_0^{e_1} \frac{df }{\sqrt{-\frac23 (f-e_1)} (f+3 +\frac12 e_1)}. 
\end{align*}
On the other hand, with $e_1$ fixed, $I_+(e_1,e_2)$ decreases as $c$ increases. In addition, as $c$ goes to $\infty$, $I_{+}(e_1, e_2)$ goes to 0. 

\vspace{3mm}
 {\bf Case 2}. The three roots are real, and as assumed above, $e_3 \leq e_2 \leq e_1$. 
 It is shown in Lemma \ref{lem:roots} that $e_1>0 \geq  e_2 \geq e_3$ and $e_3<0$. 
 Moreover, $-6-e_1= e_2+e_3 \leq 2e_2$. It follows that 
$ e_2 \geq -3- \frac12 e_1$.
This shows that for every fixed $e_1>0$, the admissible value of $e_2$ belongs to $ [-3- \frac12 e_1, 0]$. And
\begin{align*}
    (f-e_2) (f-e_3) &= (f-e_2) (f+ 6 + e_1 + e_2) = f(f+6+ e_1)-(6+ e_1) e_2 -e_2^2\\
    &= f(f+6+ e_1)-\left(e_2+3+ \frac12 e_1\right)^2 +  \left(3+ \frac12 e_1\right)^2.
\end{align*}
When $f$ and $e_1$ are fixed,
 the above expression is decreasing with respect to $e_2\in [-3- \frac12 e_1,0]$.
 % hence attains its maximum when  $e_2 = -3- \frac12 e_1$ (this also implies that $e_3 = -3- \frac12 e_1$), and when $e_2 =0$, the above expression attains its minimum.
This implies that {\bf{$I_{+}(e_1, e_2)$ is increasing  with respect to $e_2\in [-3- \frac12 e_1,0]$}}.
Then for each $e_1$,
$I_{+}(e_1, e_2)$ attains its minimum when  $e_2 = -3- \frac12 e_1$ (this also implies that $e_3 = -3- \frac12 e_1$), which is 
\begin{align*}
    I_{+}\left(e_1, -3 - \frac12 e_1\right)= \int_0^{e_1} \frac{df }{\sqrt{-\frac23 (f-e_1)} (f+3 +\frac12 e_1)}. 
\end{align*}
It is exactly the maximum of Case 1. 
On the other hand, for a fixed $e_1$, the integral $I_{+}(e_1, e_2)$ attains its maximum when $e_2 =0$, which is 
\begin{align}\label{eq: I_{+}(e_1, 0)}
  I_{+}(e_1, 0)=  \int_0^{e_1} \frac{df}{\sqrt{-\frac23 (f- e_1) f (f+6 + e_1)}} . 
\end{align}
Note that \textbf{$I_{+}(e_1, 0)$ is  strictly decreasing with respect to $e_1$}, which can be seen from the following expression
\begin{align*}
   I_{+}(e_1, 0) =  \int_0^{e_1} \frac{df}{\sqrt{-\frac23 (f-e_1)f (f+ 6 + e_1)}}
    = \int_0^1 \frac{dg}{\sqrt{-\frac23 (g-1) g (e_1 g + 6 + e_1)}}. 
\end{align*}
Hence $I_{+}(e_1, 0)$ attains its maximum when $e_1 = 0$. A direct computation shows that this maximum is exactly 
\begin{equation}\label{eq:maxi0.5pi}
    \int_0^1 \frac{dg}{\sqrt{4(1-g) g}} \xlongequal {g=\sin^2\theta} \int_{0}^{\frac{\pi}{2}} \frac{2\sin \theta \cos \theta}{2\sin \theta \cos \theta} d\theta = \frac{\pi}{2}.
\end{equation}
The above computations show that the maximum value of the half angle (of the sector) $\alpha$ that a pure outflow can achieve is $\frac{\pi}{2}$, and the maximum value is in fact achieved when $(e_1,e_2)=(0,0)$, which corresponds to the trivial solution $f\equiv 0$ under the condition $f\geq 0$. Therefore, { \bf there exists no pure outflow when $\alpha\geq \frac{\pi}{2}$. }  
\end{proof}

During the proof of Lemma \ref{nonexistence}, we also derive some monotone properties of $I_{+}(e_1, e_2)$, which we formulate as a lemma for later use. 

\begin{lemma}\label{monotone}
    For every fixed $e_1>0$, $e_2$ and $e_3$ real, $I_{+}(e_1, e_2)$ is strictly increasing with respect to $e_2\in [-3-\frac12 e_1, 0]$. Moreover, $I_{+}(e_1, 0)$ is strictly decreasing with respect to $e_1\in (0, \infty)$. 
\end{lemma}

\vspace{3mm}
It is then natural to ask when $0<\alpha <\frac{\pi}{2}$, for any given flux $\Phi>0$, whether there is a pure outflow. The answer is no, there must be a restriction on $\Phi$ from above, this is exactly Part (2) in Proposition \ref{Lemmapureoutflow2}.
We now proceed to prove the rest of Proposition \ref{Lemmapureoutflow2} including its Part (3) where the asymptotic behavior of  $\Phi_{max}^{(1,0)}(\alpha)$ is considered.

\begin{proof} [Proof of Proposition \ref{Lemmapureoutflow2}]
We only need to prove  Part (2) and  Part (3) in Proposition \ref{Lemmapureoutflow2}.

(i) Proof of Part (2) in Proposition \ref{Lemmapureoutflow2}.
To find the corresponding pure outflow with the given half angle $\alpha$ and flux $\Phi$, we need to find $e_1$ and $e_2$ such that \eqref{angle}-\eqref{flux} can be satisfied. We consider  \eqref{angle} first.
We note that $I_{+}(e_1, 0)$ in \eqref{eq: I_{+}(e_1, 0)} decreases strictly from $\frac{\pi}{2}$ to $0$, as $e_1$ increases from $0$ to $\infty$. Hence, for the angle $\alpha< \frac{\pi}{2}$, there is a unique $e_1= e_1^*(\alpha) >0 $, such that 
\begin{align}\label{angle-10}
    I_{+}(e_1^*(\alpha), 0) = \alpha. 
\end{align}
For $e_1 \leq e_1^*(\alpha)$, it holds that $I_{+}(e_1, 0) \geq \alpha $. On the other hand, 
\begin{align}\nonumber 
    \lim_{c\to \infty} I_{+} \left(e_1, -3- \frac12 e_1 + ci\right) =0. 
\end{align}
Hence, using the monotone properties of $I_{+}(e_1, e_2)$, for every $0<e_1 \leq e_1^*(\alpha)$, we can find a unique $e_2=e_2(e_1)$, such that 
\begin{align*}
    I_{+}(e_1, e_2(e_1)) = \alpha. 
\end{align*}
Here $e_2(e_1)$ could be real or complex. 
On the other hand, for every $e_1 > e_1^*(\alpha)$, $I_{+}(e_1, 0) < I_{+}(e_1^*(\alpha), 0)=\alpha$. According to Lemma \ref{monotone}, we know that 
 $I_{+}(e_1, e_2) \leq I_{+}(e_1, 0)$. Hence, 
there is no $e_2$ such that $I_{+}(e_1, e_2)=\alpha$.
This shows that the admissible value of $e_1$ is $(0,e_1^*(\alpha)]$.

Due to the same reason as that in the proof of Lemma \ref{nonexistence},  for fixed $e_1$, $J_{+}(e_1, e_2)$ is increasing with respect to $e_2\in [-3- \frac12 e_1,0]$. Hence,
the maximum of $J_{+}(e_1, e_2)$ is also attained when $e_2 = 0$ (among all possible $e_2$'s, real or complex), which is 
\begin{align*}
    J_{+}(e_1, 0) = \int_0^{e_1} \frac{f\, df }{\sqrt{-\frac23 (f - e_1) f (f+ 6 + e_1)}}  = \int_0^1 \frac{\sqrt{e_1} g \, dg }{\sqrt{\frac23(1-g) g (g+ 6/e_1 + 1)}} .
\end{align*}
Then it is easy to see $J_{+}(e_1, 0)$ is increasing with respect to $e_1$. Hence, under the angle constraint \eqref{angle}, the maximum  flux is 
\begin{align}\label{eq:maxifluxpureout} 
    \Phi_{max}^{(1, 0)}(\alpha) : =2 J_{+}(e_1^*(\alpha), 0).
\end{align}

Next, for every $0<\Phi < \Phi_{max}^{(1, 0)}(\alpha)$, we show that there exists an $e_1\in (0, e_1^*(\alpha))$, such that 
\begin{align}\nonumber
    J_{+}(e_1, e_2(e_1)) = \int_0^{e_1} \frac{f \,df }{\sqrt{-\frac23 (f-e_1) (f - e_2) (f+6+e_1 + e_2)}}  = \frac{\Phi}{2}. 
\end{align}
As proved above, for every $0< e_1 \leq e_1^*(\alpha)$, there exists a unique  $e_2(e_1)$, such that $I_{+}(e_1, e_2(e_1)) = \alpha$. Note that  $J_{+}(e_1, 0)$ goes to $0$ as $e_1$ goes to $0$, and
\begin{align}\nonumber
   0\leq  J_{+}(e_1, e_2(e_1)) \leq J_{+}(e_1, 0).
\end{align}
Hence
\begin{align}\nonumber
    \lim_{e_1 \rightarrow 0^+} J_{+}(e_1, e_2(e_1)) = 0. 
\end{align}
Since $J_{+}(e_1, e_2(e_1))$ is continuous as a function of $e_1$, for every $0<\Phi < \Phi_{max}^{(1, 0)}(\alpha)$, by the intermediate value theorem, there must be at least an $e_1$, such that 
$ J_{+}(e_1, e_2(e_1)) = \frac{\Phi}{2}. $
This proves the existence and non-existence results in Part (2) of Proposition \ref{Lemmapureoutflow2}.

 We then show the uniqueness of the pure outflow. 
 We prove the uniqueness by contradiction. Suppose there exists another $(\tilde{e}_1, \tilde{e}_2)$ satisfying \eqref{angle} and \eqref{flux}, with 
$0< \tilde{e}_1 < e_1$. We have from \eqref{angle} that
\begin{align}\nonumber
\alpha & = \int_0^{e_1} \frac{df}{\sqrt{\frac23 (e_1 - f) (f- e_2) (f - e_3)}} \\ \nonumber
&= \int_0^1 \frac{dg}{\sqrt{\frac23 (1-g) (e_1 g - e_2) ( e_1 g - e_3) / e_1 }}\\ \nonumber
&= \int_0^1 \frac{dg}{\sqrt{\frac23 (1- g) \left[  e_1 g^2 - (e_2 + e_3) g + e_2 e_3 / e_1\right]}}\\ \nonumber
&= \int_0^1 \frac{dg}{\sqrt{\frac23 (1- g) \left[ e_1 (g^2 + g) + 6 g + e_2 e_3 / e_1 \right]}} \\ \nonumber
& = \int_0^1 \frac{dg}{\sqrt{\frac23 (1- g) \left[\tilde{e}_1 (g^2 + g) + 6 g + \tilde{e}_2 \tilde{e}_3 / \tilde{e}_1 \right]}},
\end{align}
where $\tilde{e}_3=-6-\tilde{e}_1-\tilde{e}_2$.
Since $0<\tilde{e}_1 < e_1$, and then $\tilde{e}_1 (g^2 + g) < e_1 (g^2 + g)$, and hence it holds that 
\begin{align}\label{eq:quo}
\frac{e_2e_3}{e_1} < \frac{\tilde{e}_2 \tilde{e}_3}{\tilde{e}_1}. 
\end{align}
On the other hand, we have from \eqref{flux}
\begin{equation}\label{halfPhi}
\begin{aligned} 
\frac{\Phi}{2}&  = \int_0^{e_1} \frac{f\, df}{\sqrt{\frac23 (e_1 -f ) (f-e_2) (f-e_3)}}\\ 
&= \int_0^1 \frac{g \, dg}{\sqrt{\frac23 (1- g) \left[ \frac{1}{e_1}(g^2 + g) + \frac{6}{e_1^2} g + \frac{e_2e_3}{e_1^3}\right]}}\\ 
& = \int_0^1 \frac{g \, dg}{\sqrt{\frac23 (1- g) \left[ \frac{1}{\tilde{e}_1}(g^2 + g) + \frac{6}{\tilde{e}_1^2} g + \frac{\tilde{e}_2 \tilde{e}_3}{\tilde{e}_1^3}\right]}}.
\end{aligned}
\end{equation}
However, due to  $0<\tilde{e}_1 < e_1$ and \eqref{eq:quo}, we have
\begin{align}\nonumber
\frac{1}{e_1} (g^2 + g) + \frac{6}{e_1^2} g + \frac{e_2 e_3}{e_1^3}
< \frac{1}{\tilde{e}_1 }(g^2 + g) + \frac{6}{\tilde{e}_1^2} g + \frac{\tilde{e}_2 \tilde{e}_3}{\tilde{e}_1^3}. 
\end{align}
This leads to a contradiction with \eqref{halfPhi}.

(ii) Proof of Part (3) in Proposition \ref{Lemmapureoutflow2}.
Finally, let us prove the asymptotics of  $\Phi_{max}^{(1, 0)} (\alpha)$ when $\alpha\to 0$ or $\frac{\pi}{2}$. 
We consider  $\alpha\to 0$ first. In this case, we have $e_1^*(\alpha) \to +\infty$, where  $e_1^*(\alpha)$ is defined in \eqref{angle-10} explicitly.
As $e_1\to +\infty$, we have
\begin{equation} \label{eq:asy1}
\begin{aligned}
    \lim_{e_1\to +\infty} \sqrt{e_1} I_{+}(e_1, 0)& =   \lim_{e_1\to +\infty}\int_0^{e_1} \frac{\sqrt{e_1} \, df}{\sqrt{-\frac23 (f-e_1)f (f+ 6 + e_1)}}\\
  & =  \lim_{e_1\to +\infty}\int_0^1 \frac{dg}{\sqrt{-\frac23 (g-1) g (g +1+ 6/e_1 )}} \\
   & =  \int_0^1 \frac{dg}{\sqrt{\frac23 (1-g) g (g +1 )}}  = \frac{\sqrt{6\pi}  \Gamma(5/4) }{  \Gamma(3/4) }.
\end{aligned}
\end{equation}
% Hence,  as $\alpha \to 0$,
% \begin{align*}
%    \alpha  = I_{+}(e_1^*(\alpha), 0)  & \simeq \frac{\sqrt{6\pi}  \Gamma(5/4) }{  \Gamma(3/4) } \frac{1}{\sqrt{e_1^*(\alpha)}}, \textrm{ or } \sqrt{e_1^*(\alpha)} \simeq \frac{\sqrt{6\pi}  \Gamma(5/4) }{  \Gamma(3/4) } \frac{1}{\alpha}
% \end{align*}
And as $e_1\to +\infty$,
\begin{equation}\label{eq:asy2}
    \begin{aligned}
      \lim_{e_1\to +\infty} \frac{ J_{+}(e_1, 0)}{\sqrt{e_1}} &= \lim_{e_1\to +\infty} \int_0^1 \frac{ g\, dg}{\sqrt{\frac23(1-g) g (g+ 6/e_1 + 1)}}  \\
   &= \int_0^1 \frac{g \, dg}{\sqrt{\frac23 (1-g) g (g +1 )}}  = \frac{\sqrt{\frac{2}{3}\pi}  \Gamma(7/4) }{  \Gamma(5/4) }.   
    \end{aligned}
\end{equation}
Hence, we  have as $\alpha \to 0^+$ that
\begin{align*}
   \lim_{\alpha \to 0^+} \alpha \Phi_{max}^{(1, 0) }(\alpha) &=  \lim_{\alpha \to 0^+} 2 \alpha   J_{+}(e_1^*(\alpha), 0) \\
   &= \lim_{\alpha \to 0^+} 2\frac{ J_{+}(e_1^*(\alpha), 0)}{\sqrt{e_1^*(\alpha)}}\sqrt{e_1^*(\alpha)} I_{+}(e_1^*(\alpha), 0)\\
   &= \lim_{e_1^*(\alpha) \to +\infty} 2\frac{ J_{+}(e_1^*(\alpha), 0)}{\sqrt{e_1^*(\alpha)}}\sqrt{e_1^*(\alpha)} I_{+}(e_1^*(\alpha), 0)\\
   &=  2\frac{\sqrt{6\pi}  \Gamma(5/4) }{  \Gamma(3/4) }\frac{\sqrt{\frac{2}{3}\pi}  \Gamma(7/4) }{  \Gamma(5/4) }=
    \frac{4\pi \Gamma(7/4) }{  \Gamma(3/4) }   = 3\pi.
\end{align*}
On the other hand, $I_{+}(e_1, 0)$ increases to $\frac{\pi}{2}$, as $e_1$ decreases to $0$. Hence, we  have $e_1^*(\alpha) \to 0^+$, as   $\alpha\to \frac{\pi}{2}$. 
% And  hence 
% \begin{align*}
%     \Phi_{max}^{(1, 0)}(\alpha) =2 J_{+}(e_1^*(\alpha), 0)  \leq 2 e_1^*(\alpha) \cdot I_{+} (e_1^*(\alpha), 0) = 2e_1^*(\alpha) \cdot \alpha \to 0.
% \end{align*}
Now we compute
\begin{align*}
    \frac{d}{de_1} I_{+}(e_1, 0) & =\frac{d}{de_1} \int_0^{e_1} 
    \frac{df}{\sqrt{-\frac23(f-e_1)f(f+6+e_1 )}}  \\
    & = \frac{d}{de_1} \int_0^1 \frac{dg}{\sqrt{\frac23 (1-g)g(e_1 g+6+e_1)}} \\
    &= \int_0^1 \frac{-\frac12 (g+1)\, dg}{\sqrt{\frac23 (1-g)g (e_1 g + 6+ e_1)^3}}. 
\end{align*}
Hence
\begin{equation}\label{eq:limt1}
\begin{aligned}
     \lim_{e_1\to 0}  \frac{I_{+}(e_1, 0)-\frac{\pi}{2}}{e_1 -0}&= \lim_{e_1\to 0}  \frac{I_{+}(e_1, 0)-I_+(0,0)}{e_1 -0}= \lim_{e_1\to 0} \frac{d}{de_1} I_{+}(e_1, 0) \\
     &= -\frac{1}{24} \int_{0}^{1}\frac{1+g}{\sqrt{(1-g)g}} dg
    = - \frac{\pi}{16}.
\end{aligned}
\end{equation}
Moreover,
\begin{equation}\label{eq:limt2}
\begin{aligned}
  \lim_{e_1\to 0}   \frac{J_{+}(e_1, 0)}{e_1}  &=  \lim_{e_1\to 0}\frac{1}{e_1}\int_0^{e_1} \frac{f\, df }{\sqrt{-\frac23 (f - e_1) f (f+ 6 + e_1)}}  \\
  &= \lim_{e_1\to 0}\frac{1}{e_1}\int_0^1 \frac{\sqrt{e_1} g \, dg }{\sqrt{\frac23(1-g) g (g+ 6/e_1 + 1)}} \\
  &= \lim_{e_1\to 0}\int_0^1 \frac{g \, dg }{\sqrt{\frac23(1-g) g (e_1g+ 6 + e_1)}} \\
  &=\int_0^1 \frac{g \, dg }{\sqrt{4(1-g) g }} =\int_0^{\frac{\pi}{2}} \frac{2\sin^2\theta \sin\theta \cos \theta}{2\sin\theta \cos\theta} d \theta =\frac{\pi}{4}.
\end{aligned}
\end{equation}
Finally, combining \eqref{eq:limt1} and \eqref{eq:limt2}, we conclude that 
\begin{align*}
    \lim_{\alpha\to \frac{\pi}{2}} \frac{ \Phi_{max}^{(1, 0)}(\alpha)}{\frac{\pi}{2}-\alpha} &= \lim_{\alpha\to \frac{\pi}{2}}  \frac{2 J_{+}(e_1^*(\alpha), 0)}{\frac{\pi}{2}-\alpha} = \lim_{\alpha\to \frac{\pi}{2}}  \frac{2 J_{+}(e_1^*(\alpha),0)}{e_1^*(\alpha)}\frac{e_1^*(\alpha)}{\frac{\pi}{2}-\alpha}\\
    &=  \lim_{e_1^*(\alpha) \to 0}  \frac{2 J_{+}(e_1^*(\alpha),0)}{e_1^*(\alpha)}\frac{e_1^*(\alpha)-0}{\frac{\pi}{2}- I_{+}(e_1^*(\alpha),0)}=8.
\end{align*}
In summary,  for pure outflow, it holds that
\begin{equation}
\begin{aligned}\label{eq:asy3}
    \Phi_{max}^{(1, 0)}(\alpha) =\frac{3\pi}{\alpha} + o \left(\frac{1}{\alpha}\right)\mbox{ as } \alpha \to 0, \ \mbox{and}\ \    \Phi_{max}^{(1, 0)} (\alpha) = 8\left(\frac{\pi}{2}-\alpha\right)  + o \left(\frac{\pi}{2}-\alpha \right) \mbox{ as }  \alpha\to \frac{\pi}{2}.   
\end{aligned}
\end{equation} 
This completes the proof for Proposition \ref{Lemmapureoutflow2}. 
\end{proof}

\begin{remark}
    In fact, according to \cite{DrazinRiley}, it has been computed numerically that 
    \begin{align}\nonumber
 \Phi_{max}^{(1, 0)}(\alpha) \simeq 2.934 \sqrt{e_1^*(\alpha)}, ~
 \alpha \simeq \frac{3.211}{\sqrt{e_1^*(\alpha)}}, \textrm{ and } 
 \Phi_{max}^{(1, 0)}(\alpha) \simeq  \frac{9.424}{\alpha},  
  \textrm{ as } \alpha \to 0. 
    \end{align}
    This is consistent with our computations in \eqref{eq:asy1}, \eqref{eq:asy2} and ${\eqref{eq:asy3}}_1$ when the angle $\alpha\to0$. 
    Our expansion
    $\Phi_{max}^{(1, 0)} (\alpha) = 8\left(\frac{\pi}{2}-\alpha\right)  + o \left(\frac{\pi}{2}-\alpha \right) \mbox{ as }  \alpha\to \frac{\pi}{2}$ seems to be new and not mentioned in the literature.
\end{remark}

\begin{remark}
    The fact that  $\alpha=\frac{\pi}{2}$ is a borderline case for pure outflows is known since the numerical evidence of \cite{Rosenhead40}. There was also numerical evidence for the existence of the maximum flux $\Phi_{max}^{(1, 0)}(\alpha)$. 
    %Our contribution here is the rigorous proof of existence and uniqueness of pure outflow when $\Phi \leq \Phi_{max}^{(1, 0)}(\alpha)$. 
    Our contribution here is a rigorous justification for these facts and a more detailed description of the asymptotic of $\Phi_{max}^{(1, 0)}(\alpha)$ as $\alpha\to 0$ and $\frac{\pi}{2}$.
\end{remark}

%\vspace{4mm}
\section{Periodic flow-Type $(m, m)$ flow}\label{sectionperiodicflow}
In this section, we consider the periodic flows, i.e., the flow of type $(m, m)$, $m\geq 1$.
The main objective of this section is to prove the following two propositions which are about the existence and non-existence of type $(1,1)$ flows and type $(m,m)$ flows with $m\geq 2$, respectively. 
\begin{pro}\label{Lemmaperiodicflow1}
% Consider the Navier-Stokes equations \eqref{SNS} in $K$ with conditions \eqref{BC}, \eqref{eq:fluxcon}.
Consider a self-similar (SS) solution of type (1,1) to the Navier-Stokes equations \eqref{SNS} with no-slip boundary condition \eqref{NoslipBC} and the flux condition \eqref{eq:flux}.
\begin{enumerate}
  \item 
  Assume that $0< \alpha < \frac{\pi}{2}$, there exists a maximum flux  $\Phi_{max}^{(1,1)}(\alpha)=\Phi_{max}^{(1, 0)}(\alpha)> 0$ which is given in \eqref{eq:maxifluxpureout} above, such that for every $\Phi <  \Phi_{max}^{(1, 1)}(\alpha)$, there is a unique SS solution of type (1,1); on the other hand, if $\Phi > \Phi_{max}^{(1, 1)}(\alpha)$, there is no SS solution of type (1,1).
\item
Assume that $\alpha \geq \frac{\pi}{2}$, there exists a maximum flux $\Phi_{max}^{(1, 1)} (\alpha)\leq 0$ which is defined by \eqref{periodic-10} below, such that for every $\Phi < \Phi_{max}^{(1, 1)}(\alpha)$, there is a unique SS solution of type (1,1); on the other hand, if $\Phi > \Phi_{max}^{(1, 1)} (\alpha)$, there is no SS solution of type (1,1). 
\item 
 The maximum flux $\Phi_{max}^{(1, 1)}(\alpha)$ is decreasing with respect to $\alpha\in(0,\pi)$. Moreover,  $\Phi^{(1,1)}_{max}<\frac{\pi^2}{\alpha}-4\alpha$ for $0<\alpha<\pi$.
\end{enumerate}
\end{pro}

\begin{pro}\label{pro:periodicflow2}
Consider a self-similar (SS) solution of type $(m,m)$ to the Navier-Stokes equations \eqref{SNS} with no-slip boundary condition \eqref{NoslipBC} and the flux condition \eqref{eq:flux}.
Assume that  $\alpha\in (0, \pi)$ and $m\geq 2$. There exists a maximum flux $\Phi_{max}^{(m, m)}(\alpha) = m \Phi_{max}^{(1, 0)}\left(\frac{\alpha}{m}\right)$, such that
for every $\Phi < \Phi_{max}^{(m, m)}(\alpha)$, there is a unique SS solution of type $(m,m)$; on the other hand, if $\Phi >\Phi_{max}^{(m, m)}(\alpha)$, there is no SS solution of type $(m,m)$.
\end{pro}

Propositions \ref{Lemmaperiodicflow1} and \ref{pro:periodicflow2} pertain to the case $\Phi\neq \Phi^{(m, m)}_{max}(\alpha)$. The critical case $\Phi=\Phi^{(m, m)}_{max}(\alpha)$ will be discussed in detail in Section \ref{sum}.

% There are some remarks for Proposition \ref{Lemmaperiodicflow1}.
% \begin{remark}
%    Let $0< \alpha < \frac{\pi}{2}$. Recall Proposition \ref{Lemmapureoutflow2} for the  pure outflow, the maximum flux $\Phi_{max}^{(1, 0)}(\alpha)$ is attained when $e_1 = e_1^*(\alpha)$ and $e_2 = 0$. In fact, this particular pure outflow can also be regarded as a special periodic flow of type $(1, 1)$.   Part (1) of Proposition \ref{Lemmaperiodicflow1} exactly says this particular flow also attains the maximum flux in the class of type $(1, 1)$ flows.
% \end{remark}

% \begin{remark}
%  When the angle $\alpha \geq  \frac{\pi}{2}$, there is no pure outflow. Part (2) of Proposition \ref{Lemmaperiodicflow1} tells that the flow of type $(1, 1)$ attains its maximum flux when $e_1=0$ in this case, which is a pure inflow.
% \end{remark}

% We will only prove Proposition \ref{Lemmaperiodicflow1}, and Proposition \ref{pro:periodicflow2} will be a direct consequence of Proposition \ref{Lemmaperiodicflow1}.
We begin with some preliminary results.
\begin{lemma}\label{lem:(m,m)}
For the type $(m,m)$ flows, we must have $e_1\geq 0$ and $e_2$ and $e_3$ are real. Moreover, $e_2\leq0$ and $e_3<e_2$.
\end{lemma}
\begin{proof}
    First, we show that $e_1 \geq 0$. Suppose $e_1< 0$, for $Q$ defined in \eqref{eq:f'square}, $Q(0)$ must be negative. This cannot happen due to the boundary condition ${\eqref{eq:ODEsys}}_2$  and \eqref{eq:f'square}. 
    
    Second, we show that $e_2$ and $e_3$ must be real. Suppose $e_2$ and $e_3$ are complex conjugates, then using 
$e_1\geq 0$, we have $Q(f)>0$ when $f<0$ which says $f^\prime $ cannot change its sign in the region $\{f<0\}$, hence is monotone on the region $\{f<0\}$. Then  $f$ cannot move from $0$ to a negative value and then move back to $0$. 

Third, we show that $e_2\leq 0$. We note $e_3$ must be negative, since $e_1 + e_2 + e_3 =-6$ and $e_1\geq 0$, $e_3\leq e_2$. If $e_2>0$, then $Q(0)$ must be negative. This cannot happen due to the Dirichlet boundary condition ${\eqref{eq:ODEsys}}_2$  and \eqref{eq:f'square}.  

Finally, we show $e_3<e_2$. If not, then $e_3=e_2$. This cannot happen due to \eqref{periodic-01} below as the integral on the left-hand side of \eqref{periodic-01} is finite.
\end{proof}
Then we can conclude that for type $(1,1)$ flows, the corresponding plot for $Q(f)$ is  Figure \ref{three real roots}. {For $f(\theta)$, either the value of solution $f$ increases from $0$ to $e_1$ and then decreases from $e_1$ to  $e_2\leq 0$, and finally increases back to 0 (see Figure \ref{1，1-a}); or the value of solution $f$ decreases from $0$ to $e_2$ and then increases from $e_2$ to  $e_1> 0$, and finally decreases back to 0 (see Figure \ref{1，1-b}). A general type $(m,m)$ flow is simply a repetition of the above process $m$ times; see Figures \ref{2，2-a} and \ref{2，2-b} for plots of type (2,2) flows.} Hence, Proposition \ref{pro:periodicflow2} will be a direct consequence of Proposition \ref{Lemmaperiodicflow1}.

\begin{figure}[htb]
\begin{subfigure}[h]{0.4\textwidth}
     \centering
   \includegraphics[width=0.92\textwidth]{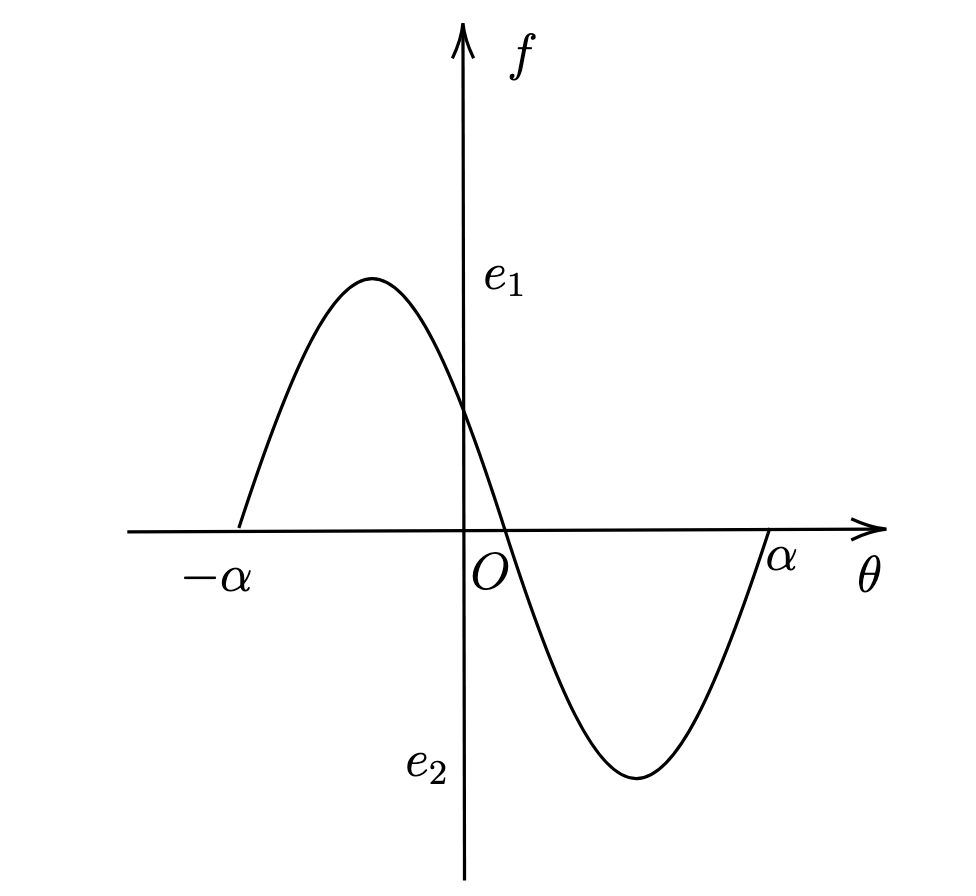}
			\caption{Type (1,1) flow }
   \label{1，1-a}
\end{subfigure}
  \hspace{0.3cm}
\begin{subfigure}[h]{0.4\textwidth}
     \centering
   \centering
    \includegraphics[width=0.8\textwidth]{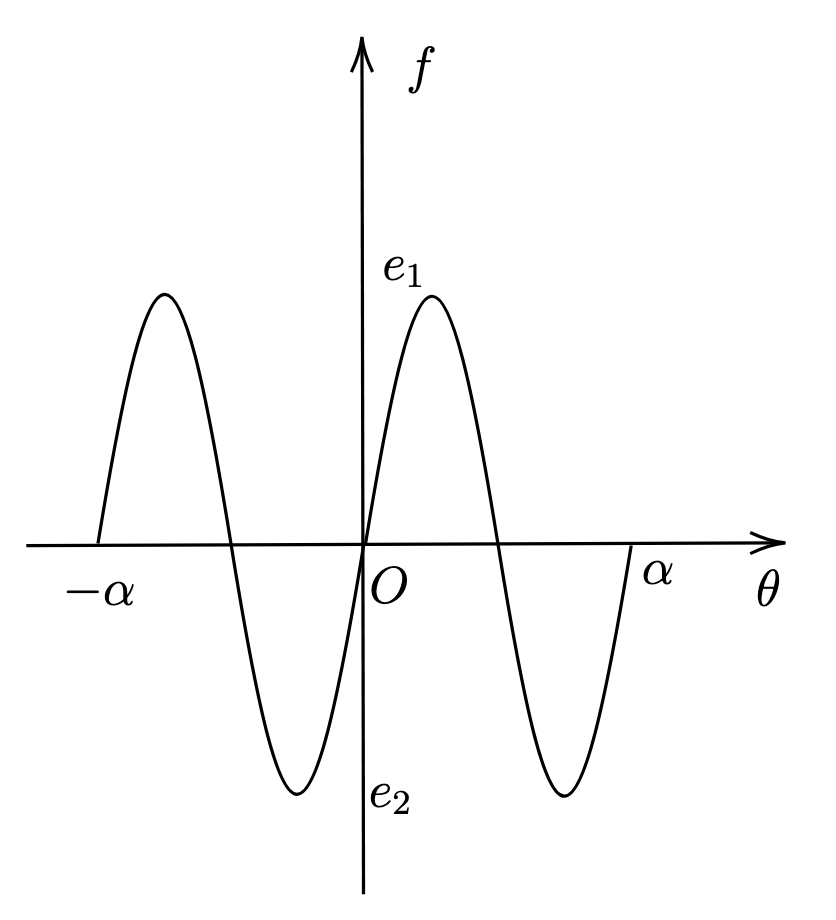}
			\caption{Type (2,2) flow }
   \label{2，2-a}
\end{subfigure}
\caption{Type  (1,1) flow and type (2,2) flow-case I}
\end{figure}

\begin{figure}[h]
\begin{subfigure}{0.4\textwidth}
     \centering
   \includegraphics[width=0.92\textwidth]{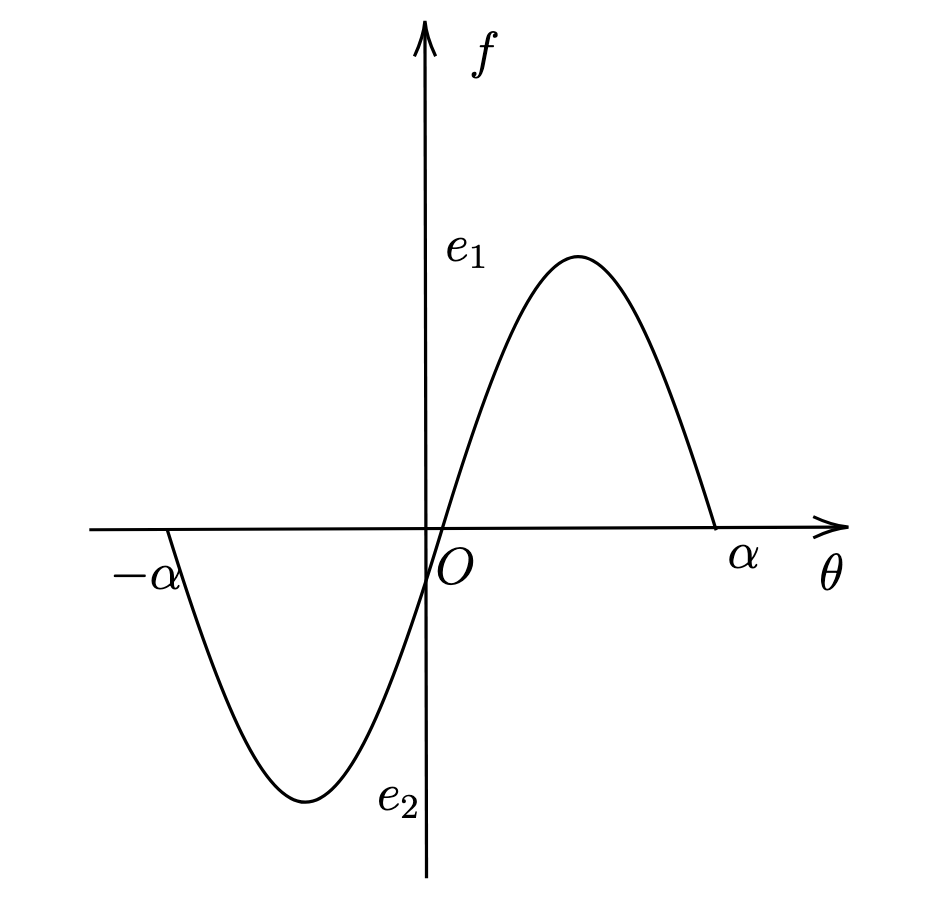}
			\caption{Type (1,1) flow }
   \label{1，1-b}
\end{subfigure}
\begin{subfigure}{0.4\textwidth}
     \centering
   \centering
    \includegraphics[width=0.92\textwidth]{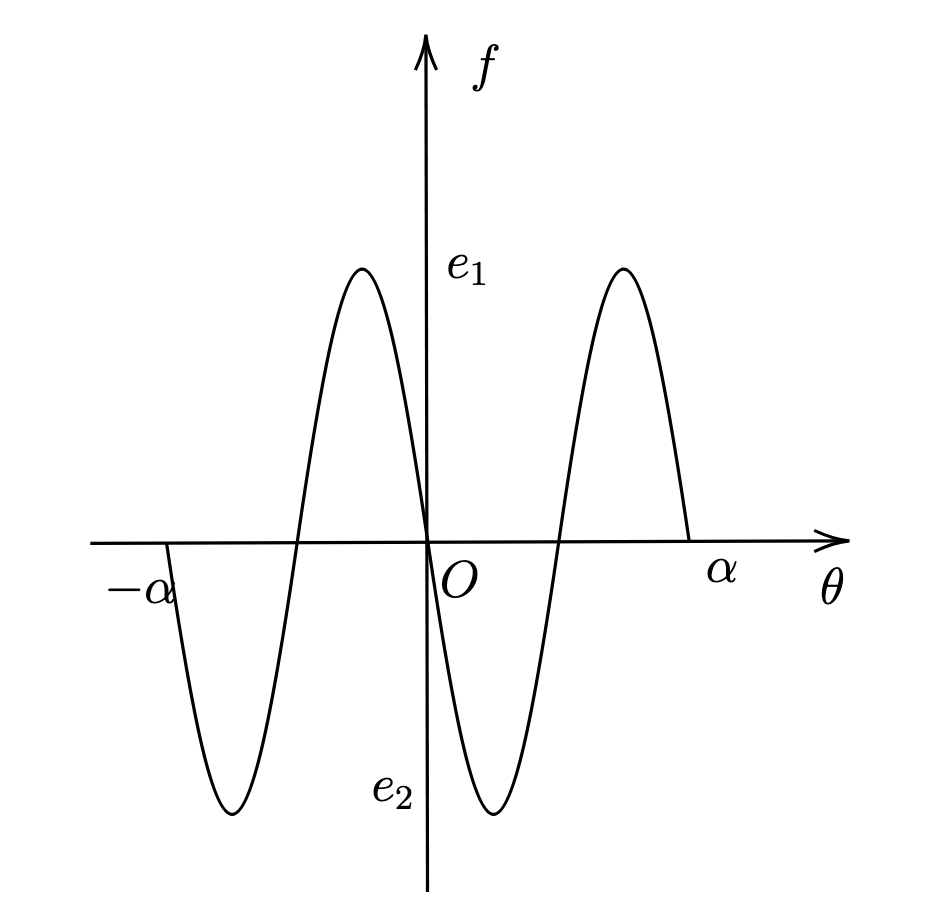}
			\caption{Type (2,2) flow }
   \label{2，2-b}
\end{subfigure}
\caption{Type  (1,1) flow and type (2,2) flow-case II}
\end{figure}

Due to the symmetry of the solution, it holds that for each flow of type $(1, 1)$,  the angle and flux condition can be rewritten as  
\begin{equation}\label{periodic-01}
I(e_1, e_2):=  \int_{e_2}^{e_1} \frac{df}{\sqrt{-\frac23 (f-e_1)(f-e_2) (f- e_3)}} = \alpha , 
\end{equation}
and 
\begin{equation}\label{periodic-02}
  J(e_1, e_2): =  \int_{e_2}^{e_1} \frac{f\, df}{\sqrt{-\frac23 (f-e_1) (f-e_2)(f-e_3)}} = \frac{\Phi}{2}. 
\end{equation}
Let $K( \gamma)$ and $E( \gamma)$ be the complete elliptic functions of the first kind and the second kind, respectively, i.e., 
\begin{equation}\label{defKE}
    K( \gamma) =\int_0^1 \frac{dt}{\sqrt{(1-t^2)(1- \gamma^2t^2})}\quad \text{and}\quad  E( \gamma) = \int_0^1 \frac{\sqrt{1-\gamma^2t^2}}{\sqrt{1-t^2}} dt.
\end{equation}
Denote
\begin{equation}\label{periodic-03}
 { \bar{\gamma}} = \sqrt{\frac{e_1 - e_2}{e_1 - e_3}}\in[0,1).
\end{equation} 
Then \eqref{periodic-01}-\eqref{periodic-02} can be rewritten as follows, 
\begin{equation}\label{periodic-1}
  \frac{ \sqrt{6} K( { \bar{\gamma}})}{\sqrt{e_1 - e_3}} =  \alpha,  
\end{equation}
\begin{equation}\label{periodic-3}
    \alpha^2 + \frac{\alpha \Phi}{4} =  [(\bar{\gamma}^2 -2)K( { \bar{\gamma}}) + 3E( \bar{ {\gamma}})]K( { \bar{\gamma}}) : = H( { \bar{\gamma}}). 
\end{equation}
Indeed, let 
\begin{equation*}t= \sqrt{\frac{e_1 -f}{e_1 - e_2}}.
\end{equation*}
It follows from \eqref{periodic-01}, \eqref{periodic-02} and \eqref{periodic-03} that one has
\begin{align*}
  \alpha 
  &= \int_{0}^{1} \frac{2\sqrt{e_1 - e_2} \, dt }{\sqrt{\frac23 (e_1-e_2) (1-t^2) (e_1- e_3- (e_1 - e_2)t^2)}}\\
  &= \int_0^1 \frac{\sqrt{6} \, dt }{\sqrt{(1-t^2) (e_1 - e_3 - (e_1 - e_2)t^2 )}}  \\
  &= \frac{1}{\sqrt{e_1 - e_3}}\int_0^1 \frac{\sqrt{6}\, dt}{\sqrt{(1-t^2) (1-  \bar{\gamma}^2 t^2)}}
  = \frac{\sqrt{6}}{\sqrt{e_1 -e_3}} K( \bar{\gamma}). 
\end{align*}
    On the other hand, 
\begin{align}\label{periodic-3-1} 
\begin{aligned}
    \Phi & = 2 \int_{e_2}^{e_1}\frac{f\, df}{\sqrt{-\frac23 (f-e_1) (f-e_2) (f-e_3)}}\\ 
    & = 2 \int_{0}^{1} \frac{ 2\sqrt{e_1 - e_2} \cdot \left[e_1 - (e_1-e_2)t^2\right]\, dt}{\sqrt{\frac23 (e_1-e_2)  (1-t^2) (e_1 - e_3 - (e_1 -e_2)t^2 )}}\\
    &= \frac{2\sqrt{6}}{\sqrt{e_1-e_3}}\int_0^1 \frac{[ e_1 -(e_1-e_2)t^2 ] \, dt}{\sqrt{(1-t^2)(1- \bar{\gamma}^2 t^2)}}\\
    &= \frac{2 \sqrt{6}}{\sqrt{e_1-e_3}} \int_0^1 \frac{[e_3 +   (e_1 - e_3) (1 -  \bar{\gamma}^2 t^2)]\, dt}{\sqrt{(1-t^2)(1- \bar{\gamma}^2t^2)}}\\
    &= \frac{2\sqrt{6}}{\sqrt{e_1-e_3}} \left[e_3 K( \bar{\gamma}) + (e_1-e_3) E( \bar{\gamma})\right].
\end{aligned}
\end{align} 
It follows from \eqref{periodic-03} that
\begin{equation*}
     \bar{\gamma}^2 (e_1-e_3) = e_1 - e_2 = 6 +2e_1 +e_3.
\end{equation*}
This implies 
\begin{equation}\label{periodic-3-2}
    e_3 = -2 + \frac13 ( \bar{\gamma}^2 -2) (e_1 - e_3).
\end{equation}
Taking \eqref{periodic-1} and  \eqref{periodic-3-2} into \eqref{periodic-3-1} yields
\begin{align*}
    \Phi &= 2\sqrt{6} \frac{-2 + \frac13 ( \bar{\gamma}^2-2)(e_1-e_3)}{\sqrt{e_1-e_3}} K( \bar{\gamma}) + 2\sqrt{6}\sqrt{e_1-e_3}E( \bar{\gamma})\\
    &=  -4\cdot \alpha  + \frac{2\sqrt{6}}{3}( \bar{\gamma}^2 -2) \frac{\sqrt{6}(K( \bar{\gamma}))^2}{\alpha} + 2 \sqrt{6}\cdot \frac{\sqrt{6}K( \bar{\gamma})}{\alpha}E( \bar{\gamma}).
\end{align*}
Hence one has
\begin{equation*}
    \Phi \alpha + 4\alpha^2 = 4( \bar{\gamma}^2 -2)K^2( \bar{\gamma}) + 12 K( \bar{\gamma}) E( \bar{\gamma})=4H( \bar{\gamma}).
\end{equation*}
This gives \eqref{periodic-3}.

Regarding the function $H({\gamma})$, it was proved by Guillod and Wittwer in \cite[Appendix A]{GuillodWittwer15SIAM} that  $H( {\gamma}) $ is a strictly decreasing function with respect to $ {\gamma}$. We recall their result as the following lemma.
\begin{lemma}\label{functionH}
The function $H$ defined by \eqref{periodic-3} is strictly decreasing, mapping $[0, 1)$ to $(-\infty, \frac{\pi^2}{4}]$. 
\end{lemma}

  % First, let us denote the three roots for this case by $(e_1^*, 0, e_3^*)$ and $ \bar{\gamma}^* = \sqrt{\frac{e_1^*}{e_1^* - e_3^*}}$. 

Now we are ready to prove Proposition \ref{Lemmaperiodicflow1}.
\begin{proof}[Proof of Proposition \ref{Lemmaperiodicflow1}] 
 (i) Proof of Part (1) in Proposition \ref{Lemmaperiodicflow1}. 
First, let us express  $(e_1, e_2, e_3)$ in terms of $\alpha$ and $ \bar{\gamma}$. 
According to \eqref{periodic-1} and  \eqref{periodic-03}, one has
\begin{equation*}
   \frac{6 K^2( \bar{\gamma})}{\alpha^2} = e_1 - e_3\quad \text{and}\quad 
     \bar{\gamma}^2 = \frac{e_1 - e_2}{e_1 - e_3}.
\end{equation*}
Consequently, it holds that
\begin{equation*}
   \bar{\gamma}^2 (e_1 - e_3) =e_1 -e_2 =  6+2e_1 + e_3.  
\end{equation*}
Hence one has
\begin{equation}\label{periodic-4}
e_1 = -2 + \frac13( \bar{\gamma}^2 +1) (e_1 - e_3)= -2 + \frac{2}{\alpha^2} ( \bar{\gamma}^2 +1) K^2( \bar{\gamma}), 
\end{equation}

\begin{equation}\label{periodic-5}
    e_3=  -2 + \frac13( \bar{\gamma}^2 -2)(e_1 - e_3)= -2 + \frac{2}{\alpha^2} ( \bar{\gamma}^2 -2 ) K^2( \bar{\gamma}),
\end{equation}
and 
\begin{equation}\label{periodic-6}
    e_2= -2 - \frac13 (2 \bar{\gamma}^2 -1)(e_1 - e_3)= -2 - \frac{2}{\alpha^2}(2 \bar{\gamma}^2 -1) K^2( \bar{\gamma}). 
\end{equation}

We claim that for any given $\alpha$, $e_2$ is strictly decreasing with respect to $ \bar{\gamma}$. Indeed,  direct computations show that
\begin{align*}
   \alpha^2 \frac{d e_2}{d \bar{\gamma}}& = -8 \bar{\gamma} \cdot K^2( \bar{\gamma}) + 4(1- 2 \bar{\gamma}^2) K( \bar{\gamma})K^\prime( \bar{\gamma}) \\
    & = 4K( \bar{\gamma}) \left[ -2 \bar{\gamma} K( \bar{\gamma}) + (1-2 \bar{\gamma}^2)K^\prime( \bar{\gamma}) \right]\\
    & = 4K( \bar{\gamma}) \left[ -2 \bar{\gamma} K( \bar{\gamma}) +  (1- 2 \bar{\gamma}^2) \frac{E( \bar{\gamma})}{ \bar{\gamma} (1- \bar{\gamma}^2)} - (1-2 \bar{\gamma}^2) \frac{K( \bar{\gamma})}{ \bar{\gamma}}\right]\\
    &= \frac{4K( \bar{\gamma})}{(1- \bar{\gamma}^2)  \bar{\gamma}} \left[ -K( \bar{\gamma})(1- \bar{\gamma}^2) + (1-2 \bar{\gamma}^2) E( \bar{\gamma})\right], 
\end{align*} 
where we used the formula (cf. \cite[p.501]{Whittaker62})
\begin{equation}\nonumber
    \frac{dK( \gamma)}{d \gamma} = \frac{E( \gamma)}{ \gamma (1-  \gamma^2)} - \frac{K( \gamma)}{ \gamma}. 
\end{equation}
Note that (cf. \cite[Eq. (A.2)]{GuillodWittwer15SIAM})
\begin{equation}\nonumber
    1-  \gamma^2 < \frac{E( \gamma)}{K( \gamma)} < 1 - \frac{ \gamma^2}{2}. 
\end{equation}
One can easily prove that 
\begin{equation}\label{periodic-8}
    \frac{de_2}{d \bar{\gamma}} < 0 \quad \text{for}\,\,  \bar{\gamma} \in (0, 1). 
\end{equation}
Let 
\begin{equation}
\label{eq1312}
     \bar{\gamma}^*= \sqrt{\frac{e_1^*(\alpha)}{e_1^*(\alpha) - (-6 - e_1^*(\alpha))}}
\end{equation}
be the ratio associated with the triple of roots $(e_1^*(\alpha), 0, -6-e_1^*(\alpha))$, where $e_1^*(\alpha) $ is implicitly defined through \eqref{angle-10}.  
Note that  $(e_1^*(\alpha),0,-6-e_1^*(\alpha))$ satisfies the equations \eqref{angle},\eqref{flux} for type $(1,0)$ flows as well as the equations \eqref{periodic-01}, \eqref{periodic-02} for type $(1,1)$ flows because of
    $I_+(e_1^*(\alpha),0)=I(e_1^*(\alpha),0)$ and $ J_+(e_1^*(\alpha),0)=J(e_1^*(\alpha),0)$.
Assume that  $(e_1, e_2, e_3)$ is the triple of roots satisfying both \eqref{periodic-01} and \eqref{periodic-02},  which corresponds to a  flow of type $(1, 1)$. It holds that  
$e_2 \leq 0$ from Lemma \ref{lem:(m,m)}. Then due to \eqref{periodic-8}, 
we have $ \bar{\gamma} \geq  \bar{\gamma}^*.$ And then it follows from Lemma \ref{functionH} that $H( \bar{\gamma}) \leq H( \bar{\gamma}^*)$. This is equivalent to say  $\Phi \leq \Phi_{max}^{(1,0)}(\alpha)$ by using \eqref{periodic-3}. Therefore, when $\Phi>\Phi^{(1,0)}_{max}(\alpha)$, there is no periodic flow of type $(1,1)$ such that \eqref{periodic-01},\eqref{periodic-02} hold.

For every $\Phi < \Phi_{max}^{(1,0)}(\alpha)$,  there is a unique $ \bar{\gamma} >  \bar{\gamma}^* $ such that \eqref{periodic-3} holds, due to Lemma \ref{functionH}. Then we have the unique triple of roots $(e_1, e_2, e_3)$ given by the equations \eqref{periodic-4}, \eqref{periodic-5}, abd \eqref{periodic-6}.
Moreover, due to the monotonicity  of $e_1$ and $e_2$ with respect to $ \bar{\gamma}$, it holds that 
\begin{equation*}
e_1 > e_1^*(\alpha)>0, \ \ \ \ e_2 <  0. 
\end{equation*}
 Therefore, for each $\Phi < \Phi_{max}^{(1,0)}(\alpha)$, we can find a unique periodic flow of type $(1, 1)$ such that \eqref{periodic-01}-\eqref{periodic-02} hold.

% \begin{remark}
%     Strictly speaking, according to the definition of flow type, the  roots $(e_1^*(\alpha), 0, -6-e_1^*(\alpha))$ gives a pure outflow, which is of type $(1, 0)$, not $(1, 1)$.  It is a limiting case of type $(1, 1)$ flow. In this paper, we also take it as a type $(1, 1)$ flow. 
% \end{remark}

%What will be the maximum flux? What will happen when the maximum flux is attained? Next, we shall prove that {\color{blue} the maximum flux corresponds to $e_1=0$ when $\alpha > \frac{\pi}{2}$.} 

% \begin{pro}\label{Lemmaperiodic3}
    
%     \end{pro}
    % We only need to prove part (2) and part (3) in Proposition \ref{Lemmaperiodicflow1}.   
     (ii) Proof of Part (2) in Proposition \ref{Lemmaperiodicflow1}.
    When $e_1 = 0$, it follows from \eqref{periodic-4} that 
\begin{equation}\label{periodic-9}
( \bar{\gamma}^2 + 1) K^2( \bar{\gamma}) = \alpha^2. 
\end{equation}
Since $( {\gamma}^2 +1)K^2( {\gamma})$ is strictly increasing and maps $[0, 1)$ to $[\frac{\pi^2}{4}, \infty)$, for fixed $\alpha\geq \frac{\pi}{2}$, there is a unique $ \bar{\gamma}=  \bar{\gamma}^*$ such that \eqref{periodic-9} holds. For this particular $ \bar{\gamma}^*$, we get the corresponding triple of roots $(0, e_2^*(\alpha), -6-e_2^*(\alpha))$, by virtue of \eqref{periodic-4}-\eqref{periodic-6}. It is noted that this particular flow is also a pure inflow. 
Let 
\begin{equation}\label{periodic-10}
    \Phi_{max}^{(1, 1)} (\alpha) := 2\int_{e_2^*(\alpha)}^0 \frac{f\, df}{\sqrt{-\frac23 f (f - e_2^*(\alpha))(f+ 6 + e_2^*(\alpha))}} \leq 0, 
    \quad
    \frac{\pi}{2}\leq \alpha<\pi. 
\end{equation}
$\Phi_{max}^{(1, 1)} (\alpha)$ is the flux of the periodic flow with the roots $(0, e_2^*(\alpha), -6- e^*_2(\alpha))$. Next, we will prove that $\Phi_{max}^{(1, 1)}(\alpha)$ is the maximum flux of type $(1, 1)$ flows for $\frac{\pi}{2}\leq \alpha<\pi$. 
Assume that $(e_1, e_2, e_3)$ is a triple of roots satisfying \eqref{periodic-01}-\eqref{periodic-02}. Recall 
    $\bar{\gamma} = \sqrt{\frac{e_1 - e_2}{e_1 - e_3}}$
and it holds that 
$ e_1 \geq 0$. On the other hand, 
it follows from \eqref{periodic-4} that $e_1$  is strictly increasing with respect to $ \bar{\gamma}$. Hence $ \bar{\gamma} \geq  \bar{\gamma}^*$.   Then Lemma \ref{functionH}  and  \eqref{periodic-3} implies that the corresponding flux $\Phi \leq \Phi_{max}^{(1,1)}(\alpha)$. Therefore, if $\Phi>\Phi^{(1,1)}_{max}(\alpha)$, there is no periodic flow of type $(1,1)$ such that \eqref{periodic-01} and \eqref{periodic-02} hold.

 Similar to the argument in the proof of Part (1) above, for every $\Phi < \Phi_{max}^{(1, 1)}(\alpha)$, there is a unique flow of type $(1, 1)$ such that \eqref{periodic-01}-\eqref{periodic-02} hold and $e_1>0, e_2<0$.

  (iii) Proof of Part (3) in Proposition \ref{Lemmaperiodicflow1}.
  % The identity $\Phi^{(1,1)}_{max} (\alpha)=\Phi_{max}^{(1,0)}(\alpha)$ for $0<\alpha<\frac{\pi}{2}$ has already been proved in Part (i) and  
  The fact $\Phi^{(1,1)}_{max}(\alpha)<\frac{\pi^2}{\alpha}-4\alpha$ for $\frac{\pi}{2}\leq\alpha<\pi$ is easily seen from ${\eqref{periodic-3}}_2$ and Lemma \ref{functionH} that $H(\bar{\gamma})\leq \frac{\pi^2}{4}$.
  % \begin{equation*}
  %     \Phi\alpha + 4\alpha^2=4H(\bar{\gamma})\leq \pi^2
  % \end{equation*}
It remains to show  $\Phi_{max}^{(1, 1)}(\alpha)$ is decreasing as a function of $\alpha$.
% \begin{remark} \label{Remark-maximumflux} . 
% \end{remark}
Assume that $0< \alpha < \beta < \frac{\pi}{2} $. Then $\Phi_{max}^{(1, 1)}(\alpha) = \Phi_{max}^{(1, 0)}(\alpha)$, and the maximum flux is attained when the second root $e_2=0$. We then have
\begin{equation*}
    \Phi_{max}^{(1, 1)}(\alpha) = 2 J_{+}(e_1^*(\alpha), 0), \ \ \ \mbox{where}\ \  I_{+}(e_1^*(\alpha), 0) =\alpha, 
\end{equation*}
\begin{equation*}
    \Phi_{max}^{(1, 1)}(\beta) = 2 J_{+}(e_1^*(\beta), 0), \ \ \ 
    \mbox{where}\ \ I_{+}(e_1^*(\beta), 0) =\beta.
\end{equation*}
According to Lemma \ref{monotone}, $I_{+}(e_1, 0)$ is strictly decreasing with respect to $e_1$, it holds that $e_1^*(\alpha)> e_1^*(\beta).$ On the other hand, as proved in Section \ref{Sectionpureoutflow}, $J_{+}(e_1, 0)$ is strictly increasing with respect to $e_1$. 
Hence $\Phi_{max}^{(1, 1)}(\beta)< \Phi_{max}^{(1,1)}(\alpha). $

Assume that $0< \alpha < \frac{\pi}{2}\leq \beta < \pi$. 
The maximum flux $\Phi_{max}^{(1, 1)} (\alpha) =\Phi_{max}^{(1, 0)} (\alpha)> 0$,  $\Phi_{max}^{(1, 1)}(\beta) \leq 0$. Hence, $\Phi_{max}^{(1,1)}(\beta)< \Phi_{max}^{(1, 1)}(\alpha). $

    Assume that  $\frac{\pi}{2} \leq \alpha < \beta < \pi$. The maximum flux is attained when the largest root $e_1 = 0$. 
  Using $( {\gamma}^2 +1)K^2( {\gamma})$ is strictly increasing again,
    there exist unique $\bar{\gamma}^*(\beta)$ and $\bar{\gamma}^*(\alpha)$ satisfying
\begin{equation}\nonumber
( \bar{\gamma}^*(\beta)^2+1)K( \bar{\gamma}^*(\beta))^2 = \beta^2\quad \text{and}\quad 
( \bar{\gamma}^*(\alpha)^2 + 1) K( \bar{\gamma}^*(\alpha))^2 = \alpha^2,
\end{equation}
respectively. 
% Here $ \bar{\gamma}^*(\alpha)$ and $ \bar{\gamma}^*(\beta)$ are the ratios associated with the root triples $(0, e_2^*(\alpha), -6-e_2^*(\alpha))$ and $(0, e_2^*(\beta), -6-e_2^*(\beta))$ respectively.
And it holds that $ \bar{\gamma}^*(\beta) >  \bar{\gamma}^*(\alpha),$  and then one can find that $\Phi_{max}^{(1,1)}(\beta)< \Phi_{max}^{(1,1)}(\alpha)$ by using the relation \eqref{periodic-3} and Lemma \ref{functionH}.
We finish the proof of Proposition \ref{Lemmaperiodicflow1}. 
\end{proof}

%Suppose there is another pair of roots $(e_1, e_2, e_3)$. To make sure $e_1\geq 0$, it holds that 
%\begin{equation}
 %   ( \bar{\gamma}^2 +1) K^2( \bar{\gamma}) > \alpha^2.
%\end{equation}
%It implies that $ \bar{\gamma} >  \bar{\gamma}^*$. Then $\Phi < \Phi_{max}= \Phi( \bar{\gamma}^*)$. 

%{\color{blue} When $\frac{\pi}{2}< \alpha < \pi$, the maximum flux $\Phi_{max}< 0 $ occurs when $e_1=0$. For each $\Phi< \Phi_{max} $, there exists a unique periodic flow with flux $\Phi$. }

\begin{proof}[Proof of Proposition \ref{pro:periodicflow2}]
 For the periodic flows with $m$ ($m\geq 2$) periods in $[-\alpha, \alpha]$, i.e., the flows of type $(m, m)$, it holds that 
\begin{equation}
\label{eq1410}
I_{m, m}(e_1, e_2) := mI(e_1, e_2) =  m\int_{e_2}^{e_1}\frac{ df }{ \sqrt{-\frac23 (f- e_1)(f-e_2) (f-e_3)}} = \alpha, 
\end{equation}
\begin{equation}
\label{eq1414}
J_{m, m}(e_1, e_2) := mJ(e_1, e_2) = m \int_{e_2}^{e_1} \frac{ f\, df }{ \sqrt{-\frac23 (f- e_1)(f-e_2) (f-e_3)}}= \frac{\Phi}{2}. 
\end{equation}
In other words, we have
\begin{equation*}
    I(e_1, e_2) = \frac{\alpha}{m}<\frac{\pi}{2}, ~ ~~~~J(e_1, e_2) = \frac{\Phi}{2m}.
\end{equation*}
Hence, the problem is reduced to the problem of type (1,1) flow with half angle $\frac{\alpha}{m}<\frac{\pi}{2}$ and flux $\frac{\Phi}{m}$. Applying Part (1) of Proposition \ref{Lemmaperiodicflow1}, we obtain the conclusion in Proposition \ref{pro:periodicflow2}. In particular, we have 
the relation  
\begin{equation*}
   \Phi_{max}^{(m, m)}(\alpha) = m \max_{\{(e_1,e_2):  I(e_1,e_2)=\frac{\alpha}{m}\}}2J(e_1, e_2)= m\Phi_{max}^{(1, 1)}\left(\frac{\alpha}{m}\right)= m \Phi_{max}^{(1, 0)}\left(\frac{\alpha}{m}\right). 
\end{equation*}
This finishes the proof of the proposition.
\end{proof}
% The rest of the analysis will be similar to the proof of Proposition \ref{Lemmaperiodicflow1}, and thus we omit the details.

%\vspace{4mm}{\color{red} How about the complicated case which is  one period plus one pure  outflow? It seems that the maximum flux also happens when $e_2=0$ with two periods, which means the maximum flux is $2\Phi_{max}(\frac{\alpha}{2})$. For this particular case, the negative part vanishes. Hence, the two periods also equal to one period plus one pure outflow. }The argument is as follows.
%Suppose there is pair of roots $(e_1, e_2, e_3)$ such that between $(-\alpha, \alpha)$ the flow is one periodic flow plus one pure outflow. The angle for the periodic part is less than $2\alpha$ and larger than $\alpha$, which means the flux of the periodic flow is less that $\Phi_{max}(\frac{\alpha}{2})$. For the pure outflow part, the angle is less than $\alpha $, with the flux less than $\Phi_{max}$

\vspace{4mm}
\section{Pure Inflow-Type (0,1) flow}\label{Sectionpureinflow}
In this section, we consider the pure inflow or type (0,1) flow, i.e., the flow that satisfies $f(\theta) <0$ when $\theta \in (-\alpha, \alpha)$. 
The main statement of this section is the following.
\begin{pro}\label{Lemmapureinflow2}
% Consider the Navier-Stokes equations \eqref{SNS} in $K$ with conditions \eqref{BC}, \eqref{eq:fluxcon}.
Consider a self-similar (SS) solution of type (0,1) to the Navier-Stokes equations \eqref{SNS} with no-slip boundary condition \eqref{NoslipBC} and the flux condition \eqref{eq:flux}.
\begin{enumerate}
\item  Assume that $0< \alpha \leq \frac{\pi}{2}$. For every $\Phi <\Phi_{max}^{(0,1)}(\alpha):=0$, there exists a {unique} SS solution of type (0,1) and obviously there is no such one if $\Phi \geq 0$.
\item 
Assume that $\frac{\pi}{2} < \alpha < \pi $. Let $\Phi_{max}^{(0,1)}(\alpha)=\Phi_{max}^{(1,1)}(\alpha)<\frac{\pi^2}{\alpha}-4\alpha$, where $\Phi_{max}^{(1,1)}(\alpha)$ is given by \eqref{periodic-10} above.  For every $\Phi \leq \Phi_{max}^{(0,1)}(\alpha)$, there exists a {unique} SS solution of type (0,1).  For every $\Phi > \Phi_{max}^{(0,1)}(\alpha)$, there is no  SS solution of type (0,1). 
\end{enumerate}
\end{pro}

%First, we show that $e_1 \geq 0$. Suppose $e_1< 0$, $Q(0)<0$. It will not happen due to the Dirichlet boundary condition${\eqref{eq:ODEsys}}_2$ . Second, we show that $e_2$ and $e_3$ must be real. Suppose $e_2$ and $e_3$ are complex conjugates, $Q(f)> 0$ when $f<0$ which says $f^\prime$ can not change sign, and thus $f$ can not move from $0$ to a negative value then move back to $0$. Third, we show that $e_2<0$. $e_3$ must be negative, since $e_1 + e_2 + e_3 = -6$ and $e_1 \geq 0, \ e_3\leq e_2$. If $e_2>0$, $Q(0)$ will be negative. This cannot happen. If $e_2 =0$, the pure inflow will be always zero.  Therefore, {\color{blue}\bf the pure inflow $f$ decreases from $0$ to $e_2$ and then increases from $e_2$ to $0$. }

Following similar arguments in Lemma \ref{lem:(m,m)}, one can show that $e_1 \geq 0$ and $e_2 < 0$ for the pure inflow.  Then for the pure inflow, the corresponding plot for $Q(f)$ is given in  Figure \ref{three real roots}.
For the pure inflow, the angle condition and the flux condition can be rewritten as
\begin{align}\label{angle2}
  I_{-}(e_1, e_2) = \int_{e_2}^0 \frac{df}{\sqrt{-\frac23 (f- e_1)(f-e_2)(f-e_3)}}  = \alpha
\end{align}
and
\begin{align}\label{flux2}
   J_{-}(e_1, e_2)= \int_{e_2}^0 \frac{f \, df }{\sqrt{-\frac23 (f- e_1)(f-e_2)(f-e_3)}}  = \frac{\Phi}{2},
\end{align}
respectively. 
We first give a lemma for the possible choices of $e_1$ and $e_2$ under the angle condition \eqref{angle2}.
\begin{lemma}\label{Lemmapureinflow1}
\begin{enumerate}
\item 
Assume that $0< \alpha \leq  \frac{\pi}{2}$.  For every $e_2<0$, there exists a unique $e_1>0$, such that \eqref{angle2} holds.   
\item 
Assume that $\alpha >  \frac{\pi}{2}$. There exists some $-3<e_2^*(\alpha)<0$, such that for every $e_2 \leq e_2^*(\alpha)$, there exists a unique $e_1\geq 0$, such that \eqref{angle2}  holds. For every $e_2 > e_2^*(\alpha)$, there does not exist $e_1\geq 0$, such that \eqref{angle2} holds. 
\end{enumerate}
\end{lemma}

\begin{remark}
  For every $\alpha > \frac{\pi}{2}$,   $e_2^*(\alpha)$ in Lemma \ref{Lemmapureinflow1} is exactly the critical number such that $I(0, e_2^*(\alpha)) = \alpha$ for flows of type $(1, 1)$. More precisely, $e_2^*(\alpha)$ satisfies 
  \begin{equation*}
 e_2^*(\alpha) = -2 - \frac{2}{\alpha^2} \left[ 2(\gamma^*)^2 -1\right] K^2(\gamma^*), 
  \end{equation*}
  with $\gamma^*$ satisfying \eqref{periodic-9} and $K$ defined in \eqref{defKE}. 
 \end{remark}

\begin{proof}[Proof of Lemma \ref{Lemmapureinflow1}] 
We have
\begin{align}\nonumber
    (e_1 -f ) (f - e_3) = (e_1 - f) (f+ 6 + e_1 + e_2) = 
    e_1^2 + (6+ e_2) e_1 - f(f+6 + e_2).
\end{align}
Then for every fixed $e_2 <0$,  when $e_1 = -3 - \frac12 e_2$, the  value for 
$(e_1 -f)(f- e_2)(f-e_3)$ attains the minimum. On the other hand, to make sure $e_2 \geq e_3$ and $e_1 + e_2 + e_3 = -6$, it must hold that 
$e_1 \geq - 6 - 2e_2. $

We consider the following two cases separately:
(1) $-3 \leq e_2 <0$; (2) $e_2<-3$. 

{\bf{Case 1}}.$-3 \leq e_2 <0$.
In this case, the admissible value of $e_1$ belongs to $[0,+\infty)$. For every fixed $e_2$, the integral $I_{-}(e_1, e_2)$ is strictly decreasing with respect to $e_1$. As $e_1$ goes to $\infty$, $I_{-}(e_1, e_2)$ goes to $0$.  For $e_1 = 0$, one has
\begin{align}\nonumber
  I_{-}(0, e_2) &=   \int_{e_2}^0 \frac{df}{\sqrt{\frac23 (-f) (f-e_2) (f+ 6+ e_2) }} 
  = \int_{0}^1 \frac{ dg}{\sqrt{\frac23 g(1-g)(e_2 g + 6 + e_2) }}. \nonumber
\end{align}
Hence
$I_{-}(0, e_2)$ is strictly decreasing with respect to $e_2$. Moreover, $I_{-}(0, e_2)$ goes to $\frac{\pi}{2}$ as $e_2$ goes to $0$ by virtue of \eqref{eq:maxi0.5pi}, and $I_{-}(0, e_2)$ tends to $+\infty$ as $e_2$ goes to $-3$. 

Therefore, if $0<\alpha \leq  \frac{\pi}{2}$, then for every $-3\leq e_2< 0$ we have $I_-(0, e_2) > \alpha. $ By the intermediate value theorem, there exists a unique $e_1 $ such that $I_{-}(e_1, e_2) = \alpha. $  However, if $\alpha > \frac{\pi}{2},$ there exists a unique $-3<e_2^*(\alpha)<0$, such that $I_{-}(0, e^*_2(\alpha) ) = \alpha$. For every $-3\leq e_2 < e_2^*(\alpha)$, it holds that $I_{-}(0, e_2) >\alpha$ and by the intermediate value theorem again, there exists a unique $e_1$ such that 
$I_{-}(e_1, e_2)= \alpha.$ If $e_2 > e_2^*(\alpha)$, then $I_{-}(0, e_2) < \alpha$, and hence there is no $e_1$, such that $I_{-}(e_1, e_2) =\alpha$. 

\vspace{2mm}
{\bf{Case 2}}. $e_2<-3$. For this case,  the admissible value of $e_1$ belongs to $[-6-2e_2,+\infty)$.

For each fixed $e_2$, the integral $I_{-}(e_1, e_2)$ is strictly decreasing with respect to $e_1$.  If $e_1= -6 -2e_2$, then one has $e_3 = e_2$. Furthermore, the integral $I_{-}(e_1, e_2)$ attains its maximum, which is $\infty$. For each fixed $e_2$, $I_{-}(e_1, e_2)$ decreases to $0$ as $e_1$ goes to $\infty$.  Hence, by the intermediate value theorem, there exists a unique $e_1$ such that $I_{-}(e_1, e_2) = \alpha$. 

Combining the above two cases, we finish the proof for Lemma \ref{Lemmapureinflow1}.
%No matter what $\alpha $ is, for every fixed $e_2 < -3$, there exists a unique $e_1(e_2)$, s.t., 
%\begin{align}
 %   I(e_1(e_2), e_2) = \alpha. 
%\end{align}
%{\color{blue} \bf Conclusion: If $0< \alpha \leq \frac{\pi}{2}$, for every $e_2 < 0$, there exists a unique $e_1 $, s. t. \eqref{angle2} holds.  If $ \alpha > \frac{\pi}{2}$, there exists some $e_2^*(\alpha)$, when $e_2 < e_2^*(\alpha)$, there existes a unique $e_1(e_2)$, s. t., \eqref{angle2} holds. }
\end{proof}

Now we are ready to prove Proposition \ref{Lemmapureinflow2}.
\begin{proof}[Proof of Proposition \ref{Lemmapureinflow2}]
We prove the existence part of  Proposition \ref{Lemmapureinflow2} first and leave the proof for the uniqueness as the last part.

{\bf Step 1.} {\it Proof of Part (1) in Proposition \ref{Lemmapureinflow2} without uniqueness}. 
 As we will see, the case for $0< \alpha \leq \frac{\pi}{2}$ is easier, so we start from this case. We only need to prove that for every flux $\Phi<0$, there exists $(e_1, e_2)$ such that  both \eqref{angle2}  and \eqref{flux2} hold. 

As proved in Lemma \ref{Lemmapureinflow1}, for every $e_2<0$, there exists a unique $e_1$, which we denote by $e_1(e_2)$, such that $I_{-}(e_1(e_2), e_2) = \alpha$. 
First, we note
\begin{align}
\label{eq:j-}
  \left|  J_{-}(e_1, e_2) \right| \leq |e_2| I_{-}(e_1, e_2) = |e_2| \alpha. 
\end{align}
As $e_2$ can be chosen arbitrarily close to $0$, $|J_{-}(e_1, e_2)|$ can also be arbitrarily close to $0$. 

Second, if $e_2<-6$, then $e_1(e_2) \geq -6 -2e_2 > -e_2. $ Let $e_3 = -6-e_1(e_2) - e_2$. 
\begin{align} \nonumber
    \alpha & = \int_{e_2}^0 \frac{df}{\sqrt{\frac23 (e_1-f) (f-e_2) (f-e_3)}} \\ \nonumber
    & \leq  \int_{e_2}^0 \frac{df}{\sqrt{\frac23 (-e_2) (f-e_2) (f-e_3)}}  \\ \nonumber
    & = \int_0^1 \frac{dg}{\sqrt{\frac23 (1- g) (e_2 g - e_3)}}\leq \int_0^1 \frac{dg}{\sqrt{\frac23 (1- g) (e_2  - e_3)}} \leq \frac{\sqrt{6}}{\sqrt{e_2 - e_3}}.  \nonumber
\end{align}
It implies that 
$
0<  e_2 - e_3 \leq \frac{6}{\alpha^2},
$
and thus $e_1(e_2) = -6 - e_2 - e_3 = -2e_2 + o(e_2)$, as $e_2 \to -\infty$. Then for large enough $|e_2|$, one has
\begin{align*}
-J_{-}(e_1(e_2), e_2) & \geq \int^{\frac{e_2}{4}}_{\frac{e_2}{2}} \frac{-f \, df }{\sqrt{-\frac23 (f- e_1(e_2))  (f-e_2) (f- e_3)}} \\
& \geq  c \int^{\frac{e_2}{4}}_{\frac{e_2}{2}} \frac{-e_2 \, df }{\sqrt{(-e_2)\cdot   (-e_2) \cdot  (-e_2) }} \geq c \sqrt{-e_2}.
\end{align*}
This implies that $J_{-}(e_1(e_2), e_2) $ goes to $-\infty$ as $e_2$ goes to $-\infty$. 

Combining the above analysis, by the intermediate value theorem, we conclude that {for every $(\alpha, \Phi)$ satisfying $\alpha \in (0, \frac{\pi}{2}]$ and  $\Phi < 0$, there exists a pure inflow. }

\vspace{2mm}

{\bf Step 2.} {\it Proof of Part (2) in Proposition \ref{Lemmapureinflow2} without uniqueness.}
Next, we consider the case that $\alpha > \frac{\pi}{2}$. For this case, as we discussed, there exists an $e_2^*(\alpha)\in (-3,0)$ such that $I_{-} (0, e_2^*(\alpha)) = \alpha$. In fact, $(0, e_2^*(\alpha), -6-e_2^*(\alpha))$ gives a flow with maximum flux $\Phi_{max}^{(1, 1)}(\alpha)$ in the class of flows of type $(1,1)$, according to the proof of Part (2) in Proposition \ref{Lemmaperiodicflow1}.  We shall prove that the particular flow also gives the maximum flux in the class of pure inflows.

According to Lemma \ref{Lemmapureinflow1}, for every $e_2 < e_2^*(\alpha)$, there exists a unique $e_1>0$, which we denote by $e_1(e_2)$, such that $I_{-}(e_1(e_2), e_2) = \alpha$. The pure inflow corresponding to $(e_1(e_2), e_2, -6-e_1(e_2)-e_2)$ can also be considered  as the negative part of a type $(1, 1)$ flow in the interval $(-\beta, \beta)$, $\beta> \alpha$, i.e., 
\begin{equation*}
    I(e_1(e_2), e_2) = \beta >\alpha.
\end{equation*}
Let $\Phi=2 J_-(e_1(e_2),e_2).$
It follows from Part (3) of Proposition  \ref{Lemmaperiodicflow1} that one has, 
\begin{equation}\nonumber
    \Phi_{max}^{(1,1)}(\beta) < \Phi_{max}^{(1, 1)}(\alpha) = 2J (0, e_2^*(\alpha)). 
\end{equation}
Hence it holds that  
\begin{equation}\nonumber
    \Phi = 2 J_{-}(e_1(e_2),e_2) <\Phi_{max}^{(1, 1)}(\beta) < \Phi_{max}^{(1, 1)}(\alpha). 
\end{equation}
Moreover, as proved above, it holds that 
\begin{equation*}
    \lim_{e_2 \to -\infty} J_{-}(e_1(e_2), e_2) = -\infty.
\end{equation*}
Therefore, for $\alpha\in (\frac{\pi}{2}, \pi)$, the maximum flux for the pure inflow is attained when $e_1 =0$. And for every $\Phi \leq  \Phi_{max}^{(1, 1)}(\alpha)$, there is a pure inflow, while for every $\Phi>\Phi^{(1,1)}_{max}(\alpha)$, there is no pure inflow.

{\bf Step 3.} {\it Uniqueness.}  Now we are in position to  prove the uniqueness of type (0,1) flow (pure inflow). As discussed above, if $0< \alpha \leq \frac{\pi}{2}$ then for every $e_2<0$, and if $\frac{\pi}{2} < \alpha< \pi$ then for every $e_2<e_2^*(\alpha)$, there is a unique $e_1(e_2)$ such that $I_{-}(e_1(e_2), e_2) = \alpha$. Hence $J_-(e_1,e_2)=J_{-}(e_1(e_2), e_2)$, which we claim is strictly increasing with respect to $e_2$. 
%Let us consider the function $J_{-}(e_1(e_2), e_2)$. We claim that it is increasing, as a function of $e_2$. 
The straightforward computations yield
\begin{equation*}
    \frac{\partial I_{-}}{\partial e_1} = \int_{e_2}^0 \frac{(-\frac12)\cdot  (6+2e_1 +e_2)\, df }{\sqrt{-\frac23(f-e_2) (f-e_1)^3 (f-e_3)^3}} < 0, 
\end{equation*}
\begin{equation*}
    \frac{\partial I_{-}}{\partial e_2} =
    \int_{e_2}^0 \frac{(-\frac12) \cdot (6e_1 + e_1^2 +e_2f+f^2) \, df }{(-e_2) \sqrt{-\frac23(f-e_2) (f-e_1)^3 (f-e_3)^3}}, 
 \end{equation*}
\begin{equation*}
    \frac{\partial J_{-}}{\partial e_1} = \int_{e_2}^0 \frac{(-\frac12) f \cdot  (6+2e_1 +e_2)\, df }{\sqrt{-\frac23(f-e_2) (f-e_1)^3 (f-e_3)^3}} >0,  
\end{equation*}
and
\begin{equation*}
    \frac{\partial J_{-}}{\partial e_2} = \frac{J_{-}}{e_2} + \int_{e_2}^0  \frac{(-\frac12) f \cdot  (6e_1 + e_1^2 + e_2f + f^2)\, df }{(-e_2)\sqrt{-\frac23(f-e_2) (f-e_1)^3 (f-e_3)^3}}. 
\end{equation*}
Then by chain rules, one has
\begin{equation*}
    \frac{d J_{-}  }{d e_2}(e_1(e_2), e_2) = 
    \frac{\partial J_{-}}{\partial e_1} \frac{de_1}{de_2}  + \frac{\partial J_{-}}{\partial e_2} 
    =   \frac{\partial J_{-}}{\partial e_1} \frac{\partial I_{-} / \partial e_2}{- \partial I_{-}/\partial e_1} + \frac{\partial J_{-}}{\partial e_2} .     
\end{equation*}
Note that $\frac{J_{-}}{e_2}>0$, it holds that
\begin{align*}
        & \frac{d J_{-}  }{d e_2} (e_1(e_2), e_2)\cdot \frac{-\partial I_{-}}{\partial e_1} = \frac{\partial J_{-}}{\partial e_1} \frac{\partial I_{-}}{\partial e_2}  - \frac{\partial J_{-}}{\partial e_2}\frac{ \partial I_{-}}{\partial e_1}\\
    > 
    &  \int_{e_2}^0 \frac{(-\frac12) f \cdot  (6+2e_1 +e_2)\, df }{\sqrt{-\frac23(f-e_2) (f-e_1)^3 (f-e_3)^3}} \cdot \int_{e_2}^0 \frac{(-\frac12) \cdot (6e_1 + e_1^2 +e_2f+f^2) \, df }{(-e_2) \sqrt{-\frac23(f-e_2) (f-e_1)^3 (f-e_3)^3}}\\
   &+  \int_{e_2}^0  \frac{(-\frac12) f \cdot  (6e_1 + e_1^2 + e_2f + f^2)\, df }{(-e_2)\sqrt{-\frac23(f-e_2) (f-e_1)^3 (f-e_3)^3}} \cdot \int_{e_2}^0 \frac{\frac12 \cdot  (6+2e_1 +e_2)\, df }{\sqrt{-\frac23(f-e_2) (f-e_1)^3 (f-e_3)^3}}. \\
 \end{align*}
 Let $ A(f, e_1, e_2) = \sqrt{-\frac23(f-e_2) (f-e_1)^3 (f-e_3)^3} $. By H\"older inequality, 
 \begin{align*}  
& \frac{d J_{-}  }{d e_2} (e_1(e_2), e_2) \cdot \frac{-\partial I_{-}}{\partial e_1} \\
 > & \frac{(6+2e_1 +e_2)}{-4e_2} \left[ 
    \int_{e_2}^0 \frac{(-f)\cdot (e_2 f +f^2) \, df }{A(f, e_1, e_2)} \cdot \int_{e_2}^{0} \frac{df}{A(f, e_1, e_2)} \right. \\
    & \ \ \ \ \ \ \ \ \ \ \ \ \ \ \ \ \ \ \ \ \ \ \ - \left. \int_{e_2}^0 \frac{(-f) \, df}{A(f, e_1, e_2)} \int_{e_2}^0 \frac{(e_2 f + f^2) \, df}{A(f, e_1, e_2)} 
    \right] \\
    =  &  \frac{6+2e_1 +e_2}{4} \left[ \int_{-e_2}^0 \frac{f^2 \, df}{A(f, e_1, e_2)} \int_{-e_2}^0 \frac{df}{A(f, e_1, e_2)} - \int_{e_2}^0 \frac{(-f)\, df}{A(f, e_1, e_2)} \int_{e_2}^0 \frac{(-f) \, df}{A(f, e_1, e_2)}    \right]\\
    &+ \frac{(6+2e_1 +e_2)}{-4e_2}\left[ \int_{e_2}^0 \frac{(-f)^3 \, df}{A(f, e_1, e_2)}\int_{e_2}^0 \frac{df}{A(f, e_1, e_2)}   - \int_{e_2}^0 \frac{(-f)\, df}{A(f, e_1, e_2)} \int_{e_2}^0 \frac{f^2 \, df}{A(f, e_1, e_2)} \right] 
    \\
    \geq& 0
    . 
 \end{align*}
Thus $J_{-}(e_1(e_2), e_2)$ is strictly increasing with respect to $e_2$. This implies the uniqueness of the flow.

\end{proof}

\section{Flows of other types}\label{sec:Flows of other types}
In this section, the existence and uniqueness of the flows of other types, i.e., types of $(m+1, m)$ and $(m, m+1)$, $m\geq 1$, are investigated.  In particular, we prove the non-uniqueness of the type $(m, m+1)$ flows. 
First, following similar arguments in Lemma \ref{lem:(m,m)}, one can show that $e_1 \geq 0$ and $e_2 < 0$ for type $(m+1, m)$ and $(m, m+1)$ flows, $m\geq 1$, and that the corresponding picture for $Q(f)$ is also given in  Figure \ref{three real roots}.
\subsection{Type $(m, m+1)$ flows}\label{sectiontype12}
In this subsection, we consider the flows of type $(m, m+1)$ with $m\geq 1$. The main statement of this subsection is as follows:
\begin{pro}\label{Lemmatype12-5}   
Consider a self-similar (SS) solution of type $(m,m+1)$ to \eqref{SNS} with no-slip boundary condition \eqref{NoslipBC} and the flux condition \eqref{eq:flux}.
Assume that $\alpha \in (0, \pi)$ and  $m \in \mathbb{N}$, $m\geq 1$. 
    \begin{enumerate}
        \item There exists a maximum flux  $\Phi_{max}^{(m, m+1)}(\alpha)$, such that 
        \begin{itemize}
            \item if $\Phi < \Phi_{max}^{(m, m+1)}(\alpha)$, then there exists an SS solution of type $(m, m+1)$; 

            \item if $\Phi >\Phi_{max}^{(m, m+1)}(\alpha)$, there exists no SS solution of type $(m, m+1)$. 
        \end{itemize}

        \item It holds that  $\Phi_{max}^{(m, m)}(\alpha) < \Phi_{max}^{(m, m+1)}(\alpha) \leq  \frac{m}{m+1} \Phi_{max}^{(m+1, m+1)}(\alpha)$.  Moreover, assume $0<\alpha<\frac{\pi}{2}$ for $m=1$ and $0<\alpha<\pi$ for $m\geq2$. Then
        for every $\Phi_{max}^{(m, m)}(\alpha) < \Phi < \Phi_{max}^{(m, m+1)} (\alpha)$, there exist at least two SS solutions of type $(m, m+1)$. 
    \end{enumerate}
\end{pro}

 For type $(m,m+1)$ flows, the angle $\alpha$ and the half flux $\frac{\Phi}{2}$ can be written as
\begin{equation}\label{eq:I1,2}
\begin{aligned}
    I_{m,m+1}(e_1, e_2)&= mI_{+}(e_1, e_2) + (m+1)I_{-}(e_1, e_2), \\
     J_{m,m+1}(e_1, e_2) &= mJ_{+}(e_1, e_2) + (m+1)J_{-}(e_1, e_2),
\end{aligned} 
\end{equation}
where $I_{+}$, $I_{-}$, $J_{+}$ and $J_{-}$ are defined by \eqref{angle},  \eqref{angle2}, \eqref{flux}, and \eqref{flux2}, respectively. 
We mainly focus on the flow of type $(1, 2)$ for the clarity of the presentation, the proof for flows of type $(m, m+1), m\geq 2$ follows the same lines.

\begin{lemma}\label{Lemmatype12-1}
  For $\alpha\in (0,\frac{\pi}{2})$, let $e_1^*(\alpha)$ be the root satisfying $I_{+}(e_1^*(\alpha), 0) = \alpha $, while let $e_1^*(\alpha) = 0$ if $ \alpha\in[\frac{\pi}{2}, \pi)$.  For every $e_1 > e_1^*(\alpha)$, there exists a unique $e_2\in \left(-3 - \frac12 e_1, 0\right)$ such that 
    \begin{equation*}
        I_{1, 2}(e_1, e_2) = \alpha.
    \end{equation*}
\end{lemma}
\begin{proof}
Consider $\alpha\in (0,\frac{\pi}{2})$ first. Since  $I_+(e_1,0)$ is strictly decreasing with respect to $e_1$,  for every $e_1 > e_1^*(\alpha)$, it holds that 
    \begin{equation}\nonumber
       \lim_{e_2 \rightarrow 0^-} I_{1, 2}(e_1, e_2) = I_{+}(e_1, 0) < I_{+}(e_1^*(\alpha), 0)= \alpha .
    \end{equation}
    If $ \alpha\in[\frac{\pi}{2}, \pi)$,
    \begin{equation}\nonumber
       \lim_{e_2 \rightarrow 0^-} I_{1, 2}(e_1, e_2) = I_{+}(e_1, 0) < \frac{\pi}{2}\leq \alpha .
    \end{equation}
On the other hand, 
\begin{equation}\nonumber
    \lim_{e_2 \rightarrow (-3 -\frac12 e_1)+} I_{1,2}(e_1, e_2) \geq \lim_{e_2 \rightarrow (-3 -\frac12 e_1)+} \int_{e_2}^0 \frac{df}{\sqrt{-\frac23 (f- e_1) (f-e_2) (f+6+e_1+e_2)}} =\infty. 
\end{equation}
Hence, by the intermediate value theorem, there exists an $e_2 \in (-3 -\frac12 e_1, 0)$ such that $I_{1,2}(e_1, e_2)=\alpha.$

 Next, we turn to the uniqueness part. We first claim that
   for $  e_2 \in(-3-\frac12 e_1,0]$, $I(e_1, e_2) $ is
  \emph{decreasing} with respect to $e_2$.  
    Let $-3-\frac12 e_1< {{e}}_2' < e_2 \leq 0$.
    According to \eqref{periodic-1}, it holds that 
    \begin{equation}\nonumber
    I(e_1, e_2) = \frac{\sqrt{6} K ( \bar{\gamma})}{\sqrt{e_1 - e_3}}, \ \ \ \ \mbox{with}\ e_3 = -6-e_1-e_2, \ \  \bar{\gamma} =\sqrt{\frac{e_1-e_2}{e_1 -e_3}}, 
    \end{equation}
    \begin{equation}\nonumber
     I (e_1, {e}_2') = \frac{\sqrt{6} K (\tilde{ \bar{\gamma}})}{\sqrt{e_1 - {e}_3'}}, \ \ \ \ \mbox{with}\ {e}_3' = -6-e_1-{e}_2', \ \ { \bar{\gamma}'} =\sqrt{\frac{e_1-{e}_2'}{e_1 -{e}_3'}}.
    \end{equation}
     Since ${e}_2'<e_2$, one has ${e}_3'> e_3$ and hence 
    \begin{equation}\label{nonunique-4}
      \bar{\gamma} < { \bar{\gamma}'}, \ \ \ \ \  \    I(e_1, e_2) <I(e_1, {e}_2').
    \end{equation} 
   Meanwhile, according to Lemma \ref{monotone}, 
   \begin{equation}\label{nonunique-5}
       I_{+}(e_1, e_2)> I_{+}(e_1, {e}_2'). 
   \end{equation}
    Therefore, combining \eqref{nonunique-4} and \eqref{nonunique-5}, we have
    \begin{equation}
    \label{eq1859}
        I_{1, 2}(e_1, e_2) = 2I(e_1, e_2) - I_{+}(e_1, e_2) < I_{1, 2}(e_1, {e}_2').
    \end{equation} 
    Hence the uniqueness of $e_2$ is proved. 
\end{proof}

\begin{lemma}\label{Lemmatype12-2}
Assume that $e_1^*(\alpha)$ is the same as that in Lemma \ref{Lemmatype12-1}, and $e_1>e_1^*(\alpha)$.  Let $e_2(e_1)$ denote the unique root indicated in Lemma \ref{Lemmatype12-1}, such that $I_{1,2}(e_1, e_2(e_1)) = \alpha. $ Then 
\begin{equation} \label{type12-5}\nonumber
    \lim_{e_1\rightarrow \infty} J_{1, 2}(e_1, e_2(e_1)) = -\infty. 
\end{equation}
\end{lemma}
\begin{proof}
    Let
    \begin{equation*}
        \beta =I(e_1, e_2(e_1)) = \int_{e_2(e_1)}^{e_1 } \frac{df}{\sqrt{-\frac23 (f - e_1) (f-e_2(e_1)) (f -e_3)}}, 
    \end{equation*}
    with $e_3 = -6 - e_1 - e_2(e_1)$. It follows from the definition of $I_{1,2}(e_1,e_2)$ in ${\eqref{eq:I1,2}}_1$ (in which $m=1$) that $\frac12 \alpha < \beta <\alpha$. 
According to \eqref{periodic-4}, one has  
\begin{equation*}
    e_1  = -2 + \frac{2}{\beta^2}( \bar{\gamma}^2 +1 ) K ( \bar{\gamma})^2 \ \ \ \ \mbox{with}\  \bar{\gamma} = \sqrt{\frac{e_1 - e_2(e_1)}{e_1 - e_3}}. 
\end{equation*}
As $e_1\to +\infty$, $K( \bar{\gamma}) \to+\infty$.
Hence it holds that
\begin{equation*}
    \lim_{e_1 \rightarrow +\infty}  \bar{\gamma}(e_1, e_2(e_1)) = 1. 
\end{equation*}
Let 
\begin{equation*}
    \tilde{\Phi}= 2J(e_1, e_2(e_1)) = 2\int_{e_2(e_1)}^{e_1} \frac{f\, df}{\sqrt{-\frac23(f- e_1)(f-e_2) (f-e_3)}}.
\end{equation*}
According to \eqref{periodic-3}, one has
\begin{equation*}
    \beta^2 + \frac14 \beta  \tilde{\Phi} = H( \bar{\gamma}) \rightarrow -\infty \ \ \ \ \mbox{as}\  \bar{\gamma} \rightarrow 1^-.
\end{equation*} 
This implies that $\tilde{\Phi} \to -\infty$, as $e_1 \to +\infty$. 
Consequently, 
\begin{equation*}
   \lim_{e_1 \rightarrow +\infty} J_{1,2}(e_1, e_2(e_1)) \leq \lim_{e_1 \rightarrow +\infty} J (e_1, e_2(e_1)) = -\infty.
\end{equation*}
Hence the proof for Lemma \ref{Lemmatype12-2} is completed.
\end{proof}

\begin{proof}[Proof of Part (1) of Proposition \ref{Lemmatype12-5} when $m=1$] Assume that the triple of roots $(e_1, e_2, e_3)$ gives a flow of type $(1, 2)$  with the flux $\Phi$.  Let $\beta = I(e_1, e_2)$. The root  $(e_1, e_2, e_3)$ also gives a flow of type $(1, 1)$ in the sector with angle $2\beta$  with a different flux $\tilde{\Phi}:=2J(e_1,e_2)\geq\Phi$. Obviously, it holds that 
\begin{equation*}
    \frac{\alpha}{2} \leq   \beta \leq  \alpha. 
\end{equation*}
Hence, according to Proposition \ref{Lemmaperiodicflow1}, one has
\begin{equation*}
    \Phi_{max}^{(1, 1)} (\beta) \leq \Phi_{max}^{(1, 1)}\left(\frac{\alpha}{2}\right), 
\end{equation*}
and consequently, 
\begin{equation}\label{new1}
    \Phi
    \leq \tilde{\Phi}\leq
    \Phi_{max}^{(1, 1)}(\beta) \leq \Phi_{max}^{(1, 1)}\left(\frac{\alpha}{2}\right). 
\end{equation}
Therefore, there is an upper bound for the flux. Combining the fact with Lemma \ref{Lemmatype12-2}, one can conclude that there is a critical value $\Phi_{max}^{(1, 2)}(\alpha)$, above which the flow of type $(1, 2)$ is impossible, and below which there is at least one flow of type $(1, 2)$.
\end{proof}

Next we study the \emph{non-uniqueness} of type $(1, 2)$ flows, for given $0< \alpha< \frac{\pi}{2}$ and flux $\Phi$. We first prove that  $\Phi_{max}^{(1, 2)}(\alpha) > \Phi_{max}^{(1, 0)}(\alpha)= \Phi_{max}^{(1, 1)}(\alpha)$. Second, assume that $\Phi_{max}^{(1, 2)}(\alpha)$ is attained at $(e_1^0, e_2^0)$, and $\Phi_{max}^{(1, 0)}(\alpha)$ is attained at $(e_1^*, 0)$, we have $e_1^0> e_1^*$. %due to Lemma \ref{Lemmatype12-1}.
Then for every $\Phi_{max}^{(1, 0)}(\alpha)=\Phi_{max}^{(1, 1)}(\alpha)< \Phi < \Phi_{max}^{(1, 2)}(\alpha)$, there are at least two flows of type $(1, 2)$. For one flow, the corresponding root satisfies $e_1^* < e_1 < e_1^0$, while for another one, $e_1> e_1^0$. 

% For every $0<\alpha<\frac{\pi}{2}$, there exists some $e_1^*(\alpha) \in (0, \infty)$, such that $\alpha =I_{+} (e_1^*(\alpha), 0)$.  
For simplicity, we take the angle $\bar{\alpha} = I_{+}(1, 0)<\frac{\pi}{2}$ and show that $\Phi_{max}^{(1, 2)}(\bar{\alpha}) > \Phi_{max}^{(1, 0)}(\bar{\alpha})$, and the non-uniqueness of the corresponding type $(1, 2)$ flows. One can generalize the proof for a general $\alpha\in(0,\frac{\pi}{2})$ immediately.

\begin{lemma}\label{Lemmanonunique-2}
Let $e_1 > 1$, and $e_2 = e_2(e_1)$ satisfy $I_{1, 2}(e_1, e_2(e_1)) = \bar{\alpha} =I_{+}(1, 0)$. It holds that 
    \begin{equation}\label{nonunique-7}
\lim_{e_1 \to 1^+} e_2(e_1) = 0. 
    \end{equation}
    \end{lemma}

\begin{proof}
Suppose \eqref{nonunique-7} does not hold. It follows from the fact   $ e_2(e_1)\in (-3-\frac12 e_1 ,0)$ that there exists a sequence of $\{e_1^n\}$, and a positive number $\delta >0$, such that 
\begin{equation*}
    \lim_{n\to \infty} e_1^n =1,\ \ \ \ \ \lim_{n\to \infty}e_2(e_1^n) = -\delta.
\end{equation*}
Then using the monotonicity formula \eqref{eq1859} of $I_{1,2}(e_1,e_2)$ with respect to $e_2$, one can get
\begin{equation*}
    \bar{\alpha} = \lim_{n\to \infty} I_{1, 2}(e_1^n, e_2(e_1^n)) = I_{1, 2}(1, -\delta) > I_{1, 2}(1, 0) = \bar{\alpha}.
\end{equation*}
This leads to a contradiction so that \eqref{nonunique-7} must hold. 
\end{proof}

\begin{lemma}\label{Lemmanonunique-3}
Let $e_1 > 1$, and $e_2 = e_2(e_1)$ satisfy $I_{1, 2}(e_1, e_2(e_1)) = \bar{\alpha} =I_{+}(1, 0)$. 
There exists a constant $C_0>0$ such that 
\begin{align*}
    \lim_{e_1 \to 1^+} \frac{\sqrt{-e_2(e_1)}}{e_1-1} = C_0. 
\end{align*}
\end{lemma}

\begin{proof}
    Note that \begin{align}
        I_{1,2}(e_1, e_2) - I_{1, 2}(1, 0) &= I(e_1, e_2) - I(e_1, 0) + I(e_1, 0) - I(1, 0) + I_{-}(e_1, e_2) \nonumber \\  
        &= \underline{ I(e_1, e_2) - I(e_1, 0) }+ \underline{I_{+}(e_1, 0) - I_{+}(1, 0)} + \underline{ I_{-}(e_1, e_2)}. \label{nonunique-9}
    \end{align}
Herein, due to \eqref{periodic-1}, we have
\begin{align*}
    I (e_1, e_2) = \frac{\sqrt{6} K ( \bar{\gamma} (e_1, e_2))}{\sqrt{e_1 - e_3}} = \frac{\sqrt{6} K ( \bar{\gamma} (e_1, e_2))}{\sqrt{6+ 2e_1 +e_2}}\ \ \ \mbox{with}\  \bar{\gamma}(e_1, e_2) = 
    \sqrt{\frac{e_1-e_2}{6+2e_1 +e_2}}.
\end{align*}
Hence, 
    \begin{align*}
        \frac{\partial I(e_1, e_2)}{\partial e_2} =
        \frac{\sqrt{6}K^\prime ( \bar{\gamma}) \frac{\partial  \bar{\gamma}}{\partial e_2}}{\sqrt{6+ 2e_1 + e_2}}- \frac{ \sqrt{6} K ( \bar{\gamma}) }{2(6+2e_1 +e_2)^{\frac32}}.
    \end{align*}
    and 
\begin{align}
    \lim_{e_1 \to 1^+, e_2 \to 0^-} \frac{I (e_1, e_2) - I (e_1, 0)}{e_2} & =  \lim_{e_1 \to 1^+, e_2 \to 0^-} \frac{\partial I (e_1, e_2)}{\partial e_2} \nonumber \\
    &= \frac{-9\sqrt{6}K^\prime (\sqrt{\frac{1}{8}})}{128}-\frac{\sqrt{6}K(\sqrt{\frac18})}{16\sqrt{8}}:=C_1. \label{nonunique-10}
\end{align}

For the second part of \eqref{nonunique-9}, since
\begin{align*}
    \frac{d}{de_1} I_{+}(e_1, 0) & =\frac{d}{de_1} \int_0^{e_1} 
    \frac{df}{\sqrt{-\frac23(f-e_1)f(f+6+e_1 )}}  \\
    & = \frac{d}{de_1} \int_0^1 \frac{dg}{\sqrt{\frac23 (1-g)g(e_1 g+6+e_1)}} \\
    &= \int_0^1 \frac{-\frac12 (g+1)\, dg}{\sqrt{\frac23 (1-g)g (e_1 g + 6+ e_1)^3}}, 
\end{align*}
it holds that
 \begin{align}
    \lim_{e_1\to 1^+} \frac{I_{+} (e_1, 0) - I_{+}(1, 0)}{e_1 -1} = \frac{d}{de_1}I_{+}(1, 0) :=C_2 <0.  
 \end{align} 

For the last part of \eqref{nonunique-9}, it follows from
\begin{align*}
    I_{-}(e_1, e_2)& =\int_{e_2}^0 \frac{df}{\sqrt{-\frac23(f-e_1)(f-e_2)(f+6+e_1 +e_2)}} \\ & = \int_0^1 \frac{\sqrt{-e_2} \, dg}{\sqrt{\frac23 (e_1 -e_2g)(1-g) (e_2 g + 6+e_1 +e_2)}}
\end{align*}
that 
\begin{align}\label{nonunique-14}
    \lim_{e_1\to 1^+, e_2\to 0^-}\frac{I_{-}(e_1, e_2)}{\sqrt{-e_2}}
    =\int_0^1 \frac{dg}{\sqrt{\frac{14}{3} (1-g)}}:=C_3 >0. 
\end{align}

Now using $I_{1, 2}(e_1, e_2(e_1)) = I_{1, 2}(1, 0) =\bar{\alpha}$ and combining  \eqref{nonunique-9}-\eqref{nonunique-14}, we have 
\begin{align*}
    \lim_{e_1\to 1^+} \frac{\sqrt{-e_2(e_1)}}{e_1 -1}= -\frac{C_2}{C_3}=C_0>0,
\end{align*}
This finishes the proof  the lemma.
\end{proof}

\begin{lemma}\label{Lemmanonunique-4}
Let $e_1 > 1$, and $e_2 = e_2(e_1)$ satisfy 
    $I_{1, 2}(e_1, e_2(e_1)) = \bar{\alpha} =I_{+}(1, 0).$ 
It holds that 
\begin{align*}    
\lim_{e_1 \to 1^+} \frac{J_{1, 2}(e_1, e_2(e_1)) - J_{1, 2}(1, 0)}{e_1 -1}>0.
\end{align*}
\end{lemma}

\begin{proof}
Consider that 
    \begin{align}
    &J_{1, 2}(e_1, e_2(e_1)) -J_{1, 2}(1, 0) \nonumber \\
    = & \underline{J_{+}(e_1, e_2(e_1)) - J_{+}(e_1, 0) }+ \underline{ J_{+}(e_1, 0)- J_{+}(1, 0) }+ \underline{ 2J_{-}(e_1, e_2(e_1))}. \label{nonunique-15}
    \end{align}
First, it follows from the straightforward computations that
\begin{align*}
    \frac{\partial}{\partial e_2} J_{+}(e_1, e_2) &=
    \frac{\partial}{\partial e_2}\int_0^{e_1} \frac{f\, df}{\sqrt{\frac23(e_1 -f) (f-e_2)(f+ 6 + e_1 + e_2)}}\\
    &= \int_0^{e_1} \frac{ (3+ e_2 +\frac12 e_1) f \, df  }{\sqrt{\frac23 (e_1 -f)} \left[ (f-e_2) (f+ 6 + e_1 +e_2)\right]^{\frac32}},
\end{align*}
and thus
\begin{align*}
    \lim_{e_1 \to 1^+, e_2\to 0^-} \frac{\partial}{\partial e_2} J_{+} (e_1, e_2) =\int_0^1 \frac{\frac72  \, df }{\sqrt{\frac23 (1-f) f (f+7)^3 }}  : =C_4>0. 
\end{align*}
Hence combining Lemmas \ref{Lemmanonunique-2} and \ref{Lemmanonunique-3} yields
\begin{align}\label{nonunique-16}
    \lim_{e_1\to 1^+} \frac{J_{+}(e_1, e_2(e_1)) - J_{+}(e_1, 0)}{e_1 -1}=  \lim_{e_1\to 1^+}\frac{J_{+}(e_1, e_2(e_1)) - J_{+}(e_1, 0)}{e_2(e_1)} \cdot \frac{e_2(e_1)}{e_1-1}  =0.
\end{align}

Second, one has
\begin{equation}\label{eq:deriofJ+}
\begin{aligned}
    \frac{d}{de_1}J_{+}(e_1, 0) &= \frac{d}{de_1}\int_0^{e_1} 
    \frac{f\, df }{\sqrt{-\frac23 (f-e_1) f (f+ 6+ e_1)}}\\
    &= \frac{d}{de_1}\int_0^1 \frac{e_1 g \, dg}{\sqrt{\frac23 (1-g) g (e_1 g + 6 + e_1)}}\\
    &= \int_0^1 \frac{g(e_1 g + 6+ e_1) -\frac12 e_1 g (1+g)}{\sqrt{\frac23 (1-g) g (e_1 g +6 +e_1)^3}}\, dg \\
    & = \int_0^1 \frac{\frac12 e_1 g^2 + \frac12 e_1 g + 6g}{\sqrt{\frac23 (1-g) g (e_1 g + 6+ e_1)^3}} \, dg.   
\end{aligned}
\end{equation}
Hence it holds that
\begin{align}\label{nonunique-17}
    \lim_{e_1 \to 1^+} \frac{J_{+}(e_1, 0) - J_{+}(1, 0) }{e_1 -1} =  \int_0^1 \frac{ g^2 + 13 g }{2\sqrt{\frac23 (1-g)g (g+7)^3}} \, dg := C_5>0. 
\end{align}

Finally, when $(e_1, e_2)$ is close to $(1, 0)$, one has
\begin{equation}\label{nonunique-18}
\begin{aligned}
    |J_{-}(e_1, e_2)| = & \left|\int_{e_2}^0 \frac{f\, df }{\sqrt{\frac23(e_1 -f) (f-e_2) (f+6+ e_1 +e_2)}} \right| \\
    \leq & C \left| \int_{e_2}^0 \frac{f\, df }{\sqrt{f-e_2}} \right| 
    \le C ( \sqrt{-e_2} )^3. 
\end{aligned}
\end{equation}

Combining \eqref{nonunique-16}, \eqref{nonunique-17}, and \eqref{nonunique-18} gives 
\begin{align*}
\lim_{e_1\to 1^+ }\frac{J_{1,2}(e_1, e_2(e_1)) - J_{1, 2}(1, 0)}{e_1 -1} = C_5 >0.
\end{align*}
This completes the proof of Lemma \ref{Lemmanonunique-4}.
\end{proof}

\begin{proof}[Proof of Part (2) of Proposition \ref{Lemmatype12-5} $(m=1)$] Let $\bar{\alpha}= I_{+}(1, 0).$
Lemma \ref{Lemmanonunique-4} shows that
  \begin{equation}\label{new2}
      \Phi_{max}^{(1, 2)}(\bar{\alpha})>2 J_{1, 2}(1,0) =2 J_{+}(1, 0)= \Phi_{max}^{(1,0)}(\bar{\alpha})=\Phi_{max}^{(1,1)}(\bar{\alpha}). 
  \end{equation}
Meanwhile, as proved in Lemma \ref{Lemmatype12-2}, 
\begin{equation}\label{nonunique-19}
    \lim_{e_1\to +\infty} J_{1, 2}(e_1, e_2(e_1)) = - \infty. 
\end{equation}
Hence, {\bf $\Phi_{max}^{(1, 2)}(\bar{\alpha})$ can be attained. }

    Assume that $\Phi_{max}^{(1, 2)}(\bar{\alpha})$ is attained at $(e_1^0, e_2^0)$. We claim that $e_1^0>1$ and $e_2^0<0$. First, we prove $e_1^0>1$. Suppose $e_1^0<1$. Then $J_{+}(e_1^0, 0) < J_{+}(1, 0).$ Consequently, 
    \begin{equation*}
        \Phi_{max}^{(1, 2)}(\bar{\alpha})=2 J_{1, 2}(e_1^0, e_2^0) \leq 2 J_{+}(e_1^0, 0)< 2 J_{+}(1, 0) =\Phi_{max}^{(1, 0)}(\bar{\alpha}),
    \end{equation*}
which gives a contradiction. Suppose $e_1^0=1$ and $e_2^0<0$. Then according to \eqref{nonunique-4}, one has
\begin{equation*}
    I (e_1^0, e_2^0)> I (1, 0) =I_{+}(1, 0)=\bar{\alpha}.
\end{equation*}
 This also leads to  a contradiction. 
Hence $e_1^0 >1$. Next, we show $e_2^0<0$. Otherwise, $e_2^0=0$, and by Lemma \ref{monotone}, 
\begin{equation*}
    I_{1, 2}(e_1^0, e_2^0) = I_{+}(e_1^0, 0) < I_{+}(1, 0) =\bar{\alpha}. 
\end{equation*}

Given any $\Phi \in (\Phi_{max}^{(1, 1)}(\bar{\alpha}), \Phi_{max}^{(1, 2)}(\bar{\alpha}))$, by the intermediate value theorem, there exists an $e_1 \in (1, e_1^0)$ such that $2J_{1, 2} (e_1, e_2(e_1)) =\Phi$.  On the other hand, due to \eqref{nonunique-19} and  the intermediate value theorem, there exists another ${e}^\flat_1 \in (e_1^0, \infty)$ such that $2J_{1, 2} ({e}^\flat_1, e_2({e}^\flat_1)) =\Phi$. We finish the proof for the special case $\bar{\alpha} = I_{+}(1, 0)<\frac{\pi}{2}$. The general case $\alpha\in(0,\frac{\pi}{2})$ is similar and omitted. 
% The proof for Proposition \ref{Proptype12}, $0<\alpha<\frac{\pi}{2}$, is thus completed. 

Next, we show that for $\frac{\pi}{2}\leq \alpha < \pi$, $\Phi_{max}^{(1, 2)}(\alpha)> \Phi_{max}^{(1, 1)}(\alpha)$ . The case $\alpha = \frac{\pi}{2}$ will be discussed in detail in Section \ref{sectionaperturedomain}. Here we focus on the case $\frac{\pi}{2}< \alpha < \pi$. In fact, we will show that 
\begin{equation}\label{nonunique-21}
    \Phi_{max}^{(1, 2)}(\alpha) \geq 0.
\end{equation}
As proved in Lemma \ref{Lemmatype12-1}, for every $e_1>0$, there exists a unique $e_2(e_1)\in (-3 - \frac12 e_1, 0)$ such that
$I_{1, 2}(e_1, e_2(e_1)) =\alpha.$ Following almost the same proof of Lemma \ref{Lemmanonunique-2}, we have $\displaystyle\lim_{e_1 \to 0^+} e_2(e_1) =0.$ (Following the proof of Lemma 6.5, one can obtain $\alpha=2I_-(0,-\delta)>\pi$, which leads to a contradiction.) Hence $\displaystyle \lim_{e_1\to 0^+} J_{1, 2}(e_1, e_2(e_1)) =0$. This implies \eqref{nonunique-21}.

Finally, we discuss the relation between $\Phi_{max}^{(1, 2)}(\alpha)$ and $\Phi_{max}^{(2, 2)}(\alpha)$. As indicated by \eqref{new1}, 
\begin{equation*}
    \Phi_{max}^{(1, 2)}(\alpha) \leq \Phi_{max}^{(1, 1)}\left(\frac{\alpha}{2}\right) = \frac12 \Phi_{max}^{(2, 2)}(\alpha).  
\end{equation*}
Now, the proof of Proposition \ref{Lemmatype12-5} is completed when $m=1$. 
\end{proof}

The proof of Proposition \ref{Lemmatype12-5} when $m\geq 2$ follows the same lines, so we omit the details.

\subsection{Type $(m+1, m)$ flows}
In this subsection, we consider the flows of type $(m+1, m), m\geq 1$. The main statement of this subsection is the following. 
\begin{pro}\label{Protype21}
   Consider a self-similar (SS) solution of type $(m+1,m)$ to the Navier-Stokes equations \eqref{SNS} with the non-slip boundary condition \eqref{NoslipBC} and the flux condition \eqref{eq:flux}.
  Assume that {$\alpha \in (0, \alpha^*] $ if  $m =1$, where
 $\alpha^*$ is explicitly defined in \eqref{eq:alpha^*}, and $\alpha\in (0,\pi)$ if $m\geq 2$.} The following statements hold.
 \begin{enumerate}
        \item There exists a maximum flux  $\Phi_{max}^{(m+1, m)}(\alpha)$, such that 
        \begin{itemize}
            \item if $\Phi < \Phi_{max}^{(m+1, m)}(\alpha)$,  there exists an SS solution of type $(m+1, m)$; 

            \item if $\Phi >\Phi_{max}^{(m+1, m)}(\alpha)$, there exists no SS solution of type $(m+1, m)$. 
        \end{itemize}

\item It holds that $\Phi_{max}^{(m+1, m) }(\alpha) \geq \Phi_{max}^{(m+1, m+1)}(\alpha).$
        \end{enumerate}
  
\end{pro}

For type $(m+1, m)$ flows, the angle $\alpha$ and the half flux $\frac{\Phi}{2}$ are given by
\begin{equation}\label{type21-1}\nonumber
\alpha = I_{m+1,m} (e_1, e_2) := (m+1)I_{+}(e_1, e_2) + mI_{-}(e_1, e_2) , 
\end{equation}
\begin{equation}\label{type21-2}\nonumber
\frac{\Phi}{2}=  J_{m+1, m}(e_1, e_2): = (m+1)J_{+}(e_1, e_2) + mJ_{-}(e_1, e_2). 
\end{equation} 

To show Part (1) of Proposition \ref{Protype21}, we need to show that  every $\Phi\in (-\infty,\Phi_{max}^{(m+1,m)}(\alpha))$ can be attained by a flow of type $(m+1,m)$.  Since for each fixed $e_1$, there may be more than one $e_2$ such that $I_{m+1,m}(e_1,e_2)=\alpha$. Hence we need a careful analysis of the level sets of $I_{m+1, m}$. We start by
considering the flows of type $(2,1)$.
The next lemma shows the convexity of $I_{2,1}(e_1,e_2)$ with respect to $e_2$.

{
\begin{lemma}\label{convexity}
For every $e_1>0$, $-3-\frac12 e_1 < e_2 <0$, it holds that 
\begin{equation}\label{convexity-1}
    \frac{\partial^2 I_{2, 1}}{\partial e_2^2} (e_1, e_2) >0. 
\end{equation}
\end{lemma}
\begin{proof}  
Recall that $\bar{\gamma}= \sqrt{\frac{e_1-e_2}{e_1-e_3}}\in [0,1)$, $I(e_1,e_2) = \frac{\sqrt{6}K(\bar{\gamma})}{\sqrt{e_1-e_3}}$ and $e_3=-6-e_1-e_2$. Straightforward computations  yield
% \begin{equation*}
% \begin{aligned}
%     \frac{\partial \bar{\gamma}}{\partial e_2} &= -\frac{1}{2}\sqrt{\frac{6+2e_1+e_2}{e_1-e_2}}\frac{3e_1+6}{(6+2e_1+e_2)^2} = \frac{-\frac{3}{2}e_1-3}{(e_1-e_2)^{\frac{1}{2}}(6+2e_1+e_2)^{\frac{3}{2}}},\\
%     \frac{\partial^2 \bar{\gamma}}{\partial e_2^2} &= \left(\frac{3}{2}e_1+3\right)\frac{ -\frac{1}{2}(e_1-e_3)+\frac{3}{2}(e_1-e_2)}{(e_1-e_2)^{\frac{3}{2}}(e_1-e_3)^{\frac{5}{2} } } .
%     \end{aligned}
% \end{equation*}
% With the aid of these computations, 
\begin{equation}\label{eq:deriofIe2}
    \begin{aligned}
    \frac{\partial I }{\partial e_2}(e_1, e_2) 
    &= \sqrt{6}K'(\bar{\gamma})\frac{-\frac{3}{2}e_1-3}{(e_1-e_2)^{\frac{1}{2}}(6+2e_1+e_2)^{2}} - \frac{\sqrt{6}}{2}  K(\bar{\gamma})(6+2e_1+e_2)^{-\frac{3}{2}}<0
     \end{aligned}
    \end{equation}
   and
    \begin{equation*}
    \begin{aligned}
        \frac{\partial^2 I }{\partial e_2^2} (e_1, e_2) 
 &=  \frac{\sqrt{6}(\frac{3}{2}e_1+3)}{(e_1-e_2)(6+2e_1+e_2)^{\frac{7}{2}}} \left[K''(\bar{\gamma})\left(\frac{3}{2}e_1+3\right)  +  \frac{K'(\bar{\gamma})}{\bar{\gamma}} \left(-3+\frac{3}{2}e_1-3e_2\right)\right ]\\
       &\quad + \frac{3\sqrt{6}}{4}K(\bar{\gamma})(6+2e_1+e_2)^{-\frac{5}{2}}.
    \end{aligned}
\end{equation*}
% One note that $-3+\frac{3}{2}e_1\leq -3+\frac{3}{2}e_1-3e_2\leq 3e_1+6$ since $-\frac{1}{2}e_1-3\leq e_2\leq 0$.
Using
\begin{equation*}
    K(\gamma) = \frac{\pi}{2}\sum_{n=0}^{\infty}\left(\frac{(2n)!}{2^{2n}(n!)^2}\right)^2 \gamma^{2n}
\end{equation*}
% \begin{equation*}
% \begin{aligned}
%     % K'(\gamma) &= \frac{\pi}{2}\sum_{n=1}^{\infty} 2n\left(\frac{(2n)!}{2^{2n}(n!)^2}\right)^2 \gamma^{2n-1},\\
%       \frac{ K'(\gamma) }{\gamma} & = \frac{\pi}{2}\sum_{n=1}^{\infty} 2n\left(\frac{(2n)!}{2^{2n}(n!)^2}\right)^2 \gamma^{2n-2}.
% \end{aligned} 
% \end{equation*}
yields
\begin{equation*}
    K''(\gamma) = \frac{\pi}{2}\sum_{n=1}^{\infty} 2n(2n-1)\left(\frac{(2n)!}{2^{2n}(n!)^2}\right)^2 \gamma^{2n-2}\geq \frac{ K'(\gamma) }{\gamma}.
\end{equation*}
This, together with 
$K'(\bar{\gamma})\geq 0$, $e_1>0$, and $e_2<0$, yields $$\left(-3+\frac{3}{2}e_1-3e_2\right) \frac{K'(\bar{\gamma})}{\bar{\gamma}} > -3\frac{K'(\bar{\gamma})}{\bar{\gamma}} \geq -3 K''(\bar{\gamma}).$$ 
Hence one has
\begin{equation}\label{eq:convex1}
    \frac{\partial^2 I }{\partial e_2^2} (e_1,e_2) \geq  \frac{3\sqrt{6}}{2}\frac{e_1(\frac{3}{2}e_1+3)}{(e_1-e_2)(6+2e_1+e_2)^{\frac{7}{2}}} K''(\bar{\gamma})+ \frac{3\sqrt{6}}{4}K(\bar{\gamma})(6+2e_1+e_2)^{-\frac{5}{2}}
    >0.
\end{equation}

 On the other hand, 
 direct computations yield that 
\begin{equation}\label{eq:deriofI+e2}
    \begin{aligned}
    \frac{\partial I_+(e_1,e_2)}{\partial e_2} 
=  \int_0^{e_1}  \frac{(3+ \frac12 e_1 + e_2) \, df}{\sqrt{-\frac23 (f-e_1)(f-e_2)^3 (f+ 6 +e_1 +e_2)^3}},  
    \end{aligned}
\end{equation}
and then
\begin{equation*}\label{eq:convex2}
    \begin{aligned}
    \frac{\partial^2 I_+(e_1,e_2)}{\partial e_2^2} = & \int_0^{e_1} \frac{3(3+ \frac12 e_1 + e_2)^2 \, df }{\sqrt{-\frac23 (f-e_1) (f-e_2)^5 (f+6+e_1 +e_2)^5}} \\
    &~\ \  + \int_0^{e_1}  \frac{df}{\sqrt{-\frac23 (f-e_1) (f-e_2)^3 (f+6+e_1 +e_2)^3}} >0.
    \end{aligned}
\end{equation*}
It follows from $I_{2,1}= I+I_+$ that the proof for \eqref{convexity-1} is completed. 
 \end{proof}

% We consider the case $0<\alpha<\frac{\pi}{2}$ first. The proof is divided into several steps.

 \begin{lemma}\label{existencetype21}
     Assume that $0< \alpha < \frac{\pi}{2}$. 
     \begin{enumerate}
         \item For every $e_1> e_1^*(\frac{\alpha}{2})$, there is a unique $e_2 \in (-3-\frac12 e_1, 0)$, such that 
         $I_{2, 1}(e_1, e_2) = \alpha. $ 

         \item There exists an $\tilde{e}_1 (\alpha) \in (e_1^*(\alpha), e_1^*(\frac{\alpha}{2}))$, such that the following statements hold.
         \begin{itemize}
             \item For every $e_1 \in (\tilde{e}_1, e_1^*(\frac{\alpha}{2})]$, there are two $e_2$ such that $I_{2, 1}(e_1, e_2) = \alpha$. 

             \item There exists a unique $\tilde{e}_2(\alpha)$, satisfying 
             $I_{2, 1}(\tilde{e}_1(\alpha), \tilde{e}_2(\alpha) ) = \alpha. $
         \end{itemize}
     \end{enumerate}
 \end{lemma}

\begin{figure}[htb]
\begin{subfigure}[t]{0.39\textwidth}
 \centering
   \includegraphics[width=1\textwidth]{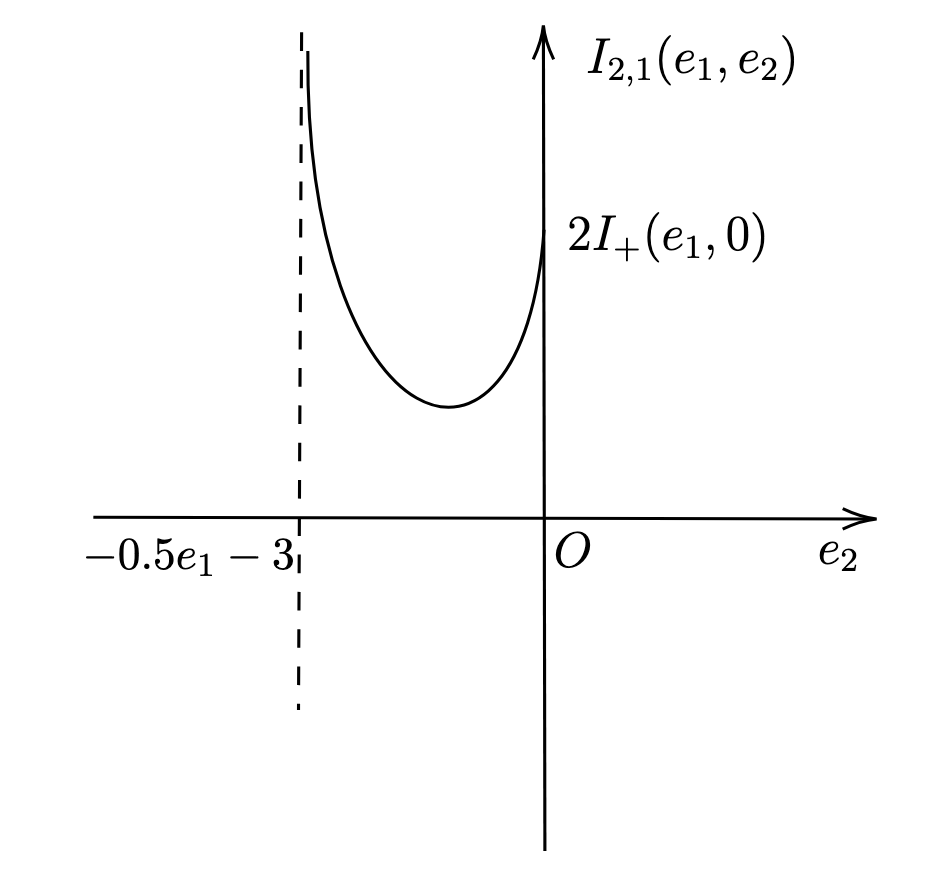}
			\caption{A possible plot of $I_{2,1}(e_1,e_2)$  \qquad with fixed $e_1$.}
   \label{fixed e1}
\end{subfigure}
  \hspace{0.3cm}
\begin{subfigure}[t]{0.46\textwidth}
 \centering
    \includegraphics[width=1\textwidth]{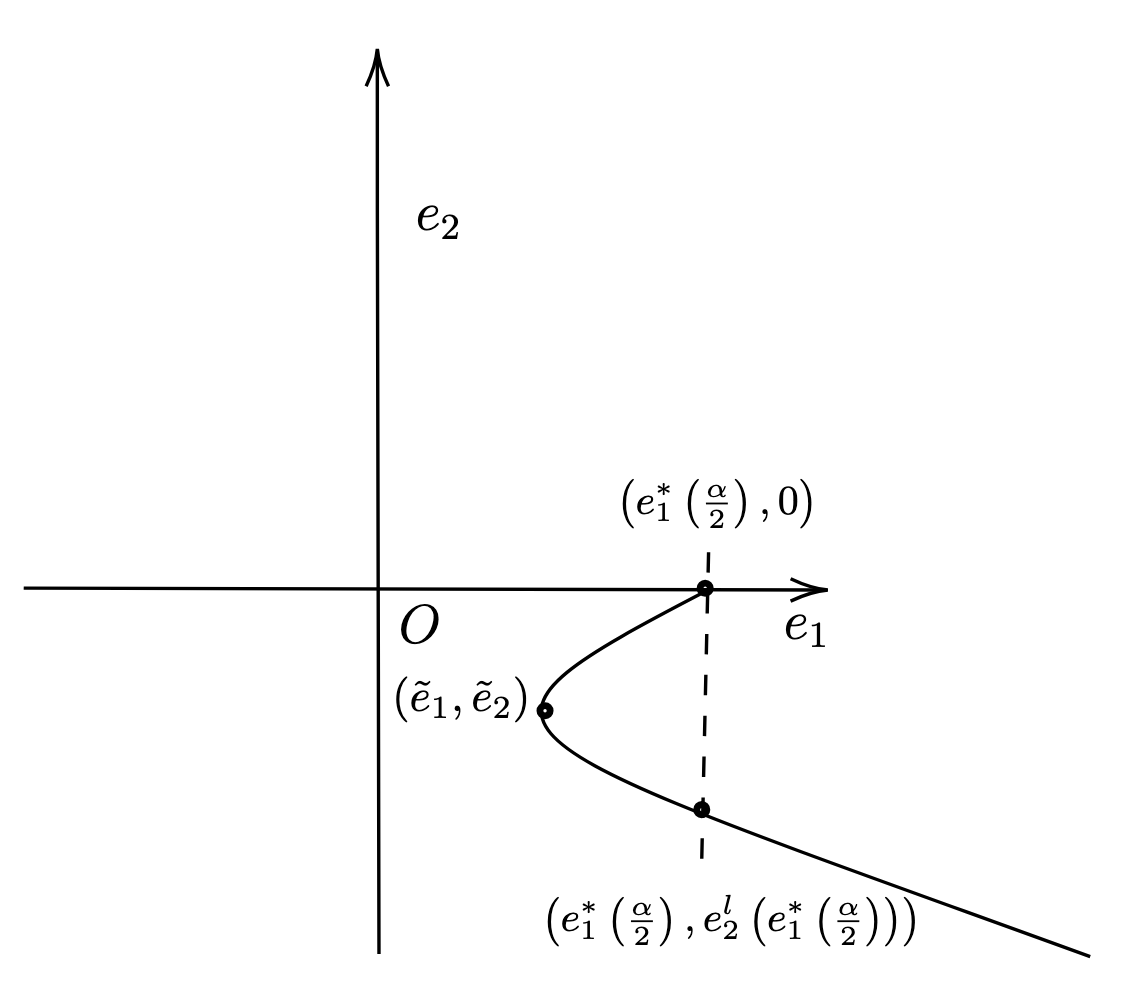}
			\caption{A connected component of the level set \qquad  $I_{2,1}(e_1,e_2)=\alpha$ with $0<\alpha<\frac{\pi}{2}$.}
   \label{connected component}
\end{subfigure}
\caption{The plot of $I_{2,1}(e_1,e_2)$ and level set of  $I_{2,1}(e_1,e_2)=\alpha$}
\end{figure}

 \begin{proof}
 \textbf{Step 1}.
 We begin with the following observations.
 % It holds that 
 % \begin{equation}
 %     \lim_{e_2 \to 0^-}\frac{\partial I_{2, 1}}{\partial e_2}(e_1, e_2) = + \infty, \ \ \ \ \lim_{e_2\to (-3-\frac12 e_1)^+} \frac{\partial I_{2, 1}}{\partial e_2}(e_1, e_2) = - \infty. 
 % \end{equation}
 For every $e_1>0$, it follows from \eqref{eq:deriofIe2} and \eqref{eq:deriofI+e2} that 
 \begin{equation*}
 \begin{aligned}
     \lim_{e_2\to 0^-} \frac{\partial I}{\partial e_2}(e_1,e_2)  &= \sqrt{6}K' \left(\sqrt{\frac{e_1}{6+2e_1}}\right) \frac{-\frac{3}{2}e_1-3}{e_1^{\frac{1}{2}}(6+2e_1)^2} -\frac{\sqrt{6}}{2}K \left(\sqrt{\frac{e_1}{6+2e_1}}\right)(6+2e_1)^{-\frac{3}{2}},\\
      \lim_{e_2\to 0^-}\frac{\partial I_+}{\partial e_2}(e_1,e_2) & = 
      \lim_{e_2 \to 0^-} \int_0^{e_1} \frac{(3+\frac12 e_1 +e_2) \, df}{\sqrt{-\frac23 (f-e_1) (f-e_2)^3 (f+6+e_1+e_2)^3}} =+\infty.
 \end{aligned}     
 \end{equation*}
Hence one has
$$
    \lim_{e_2\to 0^-}\frac{\partial I_{2,1}}{\partial e_2} =\lim_{e_2\to 0^-}\frac{\partial I}{\partial e_2}+\lim_{e_2\to 0^-}\frac{\partial I_+}{\partial e_2}=+\infty.$$
On the other hand,  as $e_2 \to (-3 - \frac12 e_1)^+$, $e_2 - e_3 \to 0$, and thus $\bar{\gamma} \to 1$. It then follows from \eqref{eq:deriofIe2} and \eqref{eq:deriofI+e2} that
\begin{equation*}
    \lim_{e_2\to (-3 -\frac{1}{2}e_1)^+} \frac{\partial I }{\partial e_2}
    = \lim_{\bar{\gamma}\to1}-\sqrt{6}K'(\bar{\gamma})\left(\frac{3}{2}e_1+3\right)^{-\frac32} - \frac{\sqrt{6}}{2}  K(\bar{\gamma})\left(\frac{3}{2}e_1+3\right)^{-\frac{3}{2}}=-\infty
    \end{equation*}
    and
    \begin{equation*}
 \lim_{e_2\to (-3 -\frac{1}{2}e_1)^+} \frac{\partial I_+(e_1,e_2)}{\partial e_2} =0,
      \end{equation*}
      respectively.
Hence one has
$$\lim_{e_2\to (-3 -\frac{1}{2}e_1)^+} \frac{\partial I_{2, 1 } }{\partial e_2}
    =-\infty.$$
For a fixed $e_1$, the plot of $I_{2,1}(e_1,e_2)$ is illustrated by Figure \ref{fixed e1}.

    \textbf{Step 2}. According to Step 1 and Lemma \ref{convexity}, for every $e_1>0$, there exists a unique ${e}^\sharp_2 (e_1) \in (-3 -\frac12 e_1, 0)$, such that 
    \begin{equation*}
     I_{2, 1}(e_1, {e}^\sharp_2(e_1) = \inf \left\{ I_{2, 1}(e_1, e_2): \ -3-\frac12 e_1 <e_2 \leq 0 \right\}.   
    \end{equation*}
   $I_{2, 1}(e_1, e_2)$ is strictly decreasing on $(-3 -\frac12 e_1, {e}^\sharp_2]$ and strictly increasing on $[{e}^\sharp_2, 0]$.  

\textbf{Step 3}. Assume that $e_1> e_1^*(\frac{\alpha}{2})$. According to Lemma \ref{monotone}, 
\begin{equation*}
    \lim_{e_2 \to 0^-} I_{2, 1}(e_1, e_2) =  2I_{+} (e_1, 0) < 2 I_{+} \left(e_1^* (\frac{\alpha}{2}), 0\right) = \alpha. 
\end{equation*}
On the other hand, 
\begin{equation*}
    \lim_{e_2 \to (-3-\frac12 e_1)^+ } I_{2, 1}(e_1, e_2) = \infty. 
\end{equation*}
Due to Step 2, there exists a unique $e_2 = e_2(e_1)\in (-3 -\frac12 e_1, {e}^\sharp_2(e_1))$, such that $I_{2, 1}(e_1, e_2) = \alpha$. Moreover, it holds that 
\begin{equation*}
    \frac{\partial I_{2, 1} }{\partial e_2}(e_1, e_2(e_1)) <0. 
\end{equation*}
 
    \textbf{Step 4}. 
Let 
\begin{equation*}
    S = \left\{ e_1>0: \ \ \inf_{e_2\in(-3-\frac{1}{2}e_1,\, 0)}I_{2,1}(e_1,e_2) \geq \alpha  \right\}. 
\end{equation*}
    First, we claim that $S$ is not empty. If $e_1 <  e_1^*(\alpha)$, then
    \begin{equation*}
      \inf_{e_2\in(-3-\frac{1}{2}e_1,\, 0)}  I_{2, 1}(e_1, e_2) \geq \inf_{e_2\in(-3-\frac{1}{2}e_1,\, 0)}  I(e_1, e_2) \geq I_{+}(e_1, 0) > I_{+}(e_1^*(\alpha), 0) =\alpha, 
    \end{equation*}
    where we have applied Lemma \ref{monotone} and the fact that $ I(e_1, e_2)$ is decreasing with respect to $e_2$ due to \eqref{eq:deriofIe2}. Hence, $(0, e_1^*(\alpha))\subset S$ and $S$ is not empty. 

     Let $\tilde{e}_1= \sup S\in [ e_1^*(\alpha),  e_1^*(\frac{\alpha}{2})) $. 
     The fact $\tilde{e}_1<e_1^*(\frac{\alpha}{2})$ can be seen from
      if $e_1 \geq  e_1^*(\frac{\alpha}{2})$, then
    \begin{equation*}
      \inf_{e_2\in(-3-\frac{1}{2}e_1,\, 0)}  I_{2, 1}(e_1, e_2) <  I_{2, 1}(e_1, 0)  =2I_{+}(e_1, 0) \leq 2 I_{+}(e_1^*\left(\frac{\alpha}{2}\right), 0) =\alpha. 
    \end{equation*}
     It follows from the definition of $\tilde{e}_1$ that if $e_1\in (\tilde{e}_1, e_1^*(\frac{\alpha}{2})]$, then
       $$\inf_{e_2\in(-3-\frac{1}{2}e_1,\, 0)}I_{2,1}(e_1,e_2) < \alpha .$$ Due to Step 2, there are exactly two $e_2$ such that $I_{2,1}(e_1,e_2)=\alpha$, which we denote by $e_2^l(e_1)$ and $e_2^r (e_1) $ respectively, $e_2^l(e_1) < {e}^\sharp_2(e_1) < e_2^r(e_1)$. Moreover, it holds that 
       \begin{equation*}
        \frac{\partial I_{2, 1} }{\partial e_2}(e_1, e_2^l(e_1)) <0, \ \ \ \ \ \ \ \frac{\partial I_{2, 1} }{\partial e_2}(e_1, e_2^r(e_1)) >0. 
       \end{equation*}

       \textbf{Step 5}. If $e_1  =\tilde{e}_1$, then $\inf_{e_2\in(-3-\frac{1}{2}e_1,\, 0)}I_{2,1}(e_1,e_2)=\alpha$. There is a unique $\tilde{e}_2$, such that $I_{2,1}(\tilde{e}_1, \tilde{e}_2)=\alpha$. It holds that 
       \begin{equation*}
            \frac{\partial I_{2, 1} }{\partial e_2}(\tilde{e}_1, \tilde{e}_2(e_1)) =0.
       \end{equation*}
This finishes the proof of the lemma.       
\end{proof}

It follows from the results in Lemma \ref{existencetype21} and the implicit function theorem that  a component of the
  level set  $\{(e_1,e_2): I_{2,1}(e_1,e_2)=\alpha\}$ is a continuous curve starting from $(e_1^*(\frac{\alpha}{2}),0)$, which can be extended to $(\tilde{e}_1,\tilde{e}_2)$.  Furthermore, a component of the
  level set  $\{(e_1,e_2): I_{2,1}(e_1,e_2)=\alpha\}$ can start from $(e_1^*(\frac{\alpha}{2}),e_2^l(e_1^*(\frac{\alpha}{2})))$, which is extended to $(\tilde{e}_1,\tilde{e}_2)$ on the left and approach  $(+\infty,-\infty)$ on the right. In summary,   there is a continuous curve in the level set of $\{(e_1,e_2): I_{2,1}(e_1,e_2)=\alpha\}$ starting from $(e_1^*(\frac{\alpha}{2}),0)$ and approaches to $(+\infty,-\infty)$. The corresponding plot is given in Figure \ref{connected component}.

%In the next lemma, we show the existence of maximum flux $\Phi_{max}^{(2,1)}(\alpha)$, which is the supremum of the flux that can be attained by a type (2,1) flow. 
% The existence of maximum flux $\Phi_{max}^{(m+1,m)}(\alpha)$ follows from the same proof.
\begin{lemma}\label{Lemmatype21-2}
Assume that $e_1 > e_1^* (\frac{\alpha}{2})$ and $e_2=e_2(e_1)$ is the unique root satisfying $I_{2, 1}(e_1, e_2) = \alpha. $ Then 
\begin{equation} \nonumber
    \lim_{e_1\rightarrow \infty} J_{2, 1}(e_1, e_2(e_1)) = -\infty. 
\end{equation}
\end{lemma}

\begin{proof} 
When $e_1 \geq  e_1^*(\frac{\alpha}{4})$, according to Lemma \ref{monotone},  one has
\begin{equation*}
    I_{+}(e_1, e_2)\leq I_{+}(e_1, 0) \leq I_{+}\left(e_1^*\left(\frac{\alpha}{4}\right), 0\right)= \frac{\alpha}{4}, 
\end{equation*}
and consequently, 
\begin{equation*}
   I_{-}(e_1, e_2(e_1)) \geq \frac{\alpha}{2}. 
\end{equation*}
On the other hand, for $e_1 > 6$, 
\begin{equation}\nonumber
    e_1 > 3+ \frac12 e_1 \geq - e_2.
\end{equation}
Therefore, if $e_1 > \max\{ e_1^*(\frac{\alpha}{4}), \ 6\}, $ then 
\begin{align*}
    \frac{\alpha}{2} \leq I_{-}(e_1, e_2) & = \int_{e_2}^0 \frac{df}{\sqrt{\frac23(e_1 - f)(f-e_2) (f-e_3)}} \\
    & \leq \int_{e_2}^0 \frac{df}{\sqrt{\frac23 (-e_2) (f -e_2) (f-e_3)}}\\
    & = \int_0^1 \frac{dg}{\sqrt{\frac23 (1-g) (e_2 g - e_3)}} \leq \frac{\sqrt{6}}{\sqrt{e_2 -e_3}}. 
\end{align*}
This implies
\begin{equation}\nonumber
    0<e_2 - e_3 \leq \frac{24}{\alpha^2}. 
\end{equation}
Hence it follows from $e_3=-6-e_1-e_2$ that one has
\begin{equation}\label{type21-5}
    \lim_{e_1\to \infty}\frac{|e_2(e_1)|}{e_1} =\frac12 \ \ \  \ \mbox{and}\ \ \ \lim_{e_1\to \infty} \frac{e_3(e_1)}{e_2(e_1)}= 1.  
\end{equation}
Let us compare the sizes of $J_{+}(e_1, e_2)$ and $J_{-}(e_1, e_2)$ when $e_1$ is large. Clearly,
\begin{align*}
    J_{+}(e_1, e_2)& = \int_{0}^{e_1} \frac{f\, df}{\sqrt{-\frac23 (f -e_1)(f-e_2)(f-e_3)}}  \\
    & = \int_0^1 \frac{\sqrt{e_1} g\, dg}
    {\sqrt{\frac23 (1-g) ( g- \frac{e_2}{e_1}) ( g - \frac{e_3}{e_1})}} \leq C \sqrt{e_1}. 
\end{align*}
Meanwhile, due to \eqref{type21-5}, one has
\begin{align}\label{type21-6}
   \lim_{e_1 \to \infty} \frac{ |J_{-}(e_1, e_2)| }{\sqrt{-e_2}} & = \lim_{e_1 \to \infty} \int_0^1 \frac{g \, dg }{\sqrt{\frac23 (1-g) (g-\frac{e_1}{e_2}) (\frac{e_3}{e_2}- g)}} = +\infty. 
\end{align}
Combining  \eqref{type21-5}-\eqref{type21-6}, 
it holds that 
\begin{equation}\label{eq:negainf}
J_{2, 1}(e_1, e_2) = 2J_{+}(e_1, e_2) + J_{-}(e_1, e_2)\to -\infty, \ \ \ \ \ \mbox{as}\ e_1\to +\infty. 
 \end{equation}
 This finishes the proof of the lemma.
 \end{proof}

Now we are ready to prove Proposition \ref{Protype21}.
  \begin{proof}[Proof of Proposition \ref{Protype21}] 
  (i) Proof of Part (1). 
  We first consider the case  $m=1$. The flux has a natural upper bound $2e_1 \alpha $. This, together with \eqref{eq:negainf}, yields that there exists a maximum flux $\Phi_{max}^{(2,1)}(\alpha)$, which is the supremum of the flux and can be attained by a type (2,1) flow. 
  
  For $\alpha\in (0, \frac{\pi}{2})$, we will show that any flux $\Phi \in (-\infty, \Phi_{max}^{(2,1)}(\alpha)]$ can be achieved on the connected component in the level set $\{(e_1,e_2): I_{2,1}(e_1,e_2)=\alpha\}$ plotted in Figure \ref{connected component}. It suffices to show that $\Phi_{max}^{(2,1)}(\alpha)$ cannot be attained if $e_1<e_1^*(\frac{\alpha}{2})$. Indeed, as shown in the proof of Proposition \ref{Lemmapureoutflow2}, $J_+(e_1,e_2)$ is strictly increasing with respect to $e_2\in (-\frac{1}{2}e_1 -3,0)$ and $J_+(e_1,0)$ is strictly increasing with respect to $e_1 \in (0,+\infty)$. Hence if $e_1<e_1^*(\frac{\alpha}{2})$, then 
  \begin{equation*}
      J_{2,1}(e_1,e_2)\leq 2J_+(e_1,e_2)\leq 2J_+(e_1,0)<2J_+\left(e_1^*\left(\frac{\alpha}{2}\right),0\right) \leq \frac12 \Phi_{max}^{(2,1)}(\alpha).
  \end{equation*}

For $\alpha \in [\frac{\pi}{2}, \pi)$, define 
    \begin{equation*}
    \begin{aligned}
          S:& =\left\{e_1\in \left(0,e_1^*\left(\frac{\alpha}{2}\right)\right): \inf_{e_2\in(-3-\frac{1}{2}e_1,0)}I_{2,1}(e_1,e_2) \geq \alpha\right\}. \\
    \end{aligned}     
    \end{equation*}
    Let $\tilde{\tilde{e}}_1 (\alpha) = \sup S$ if $S$ is not empty. Otherwise, let $\tilde{\tilde{e}}_1(\alpha)=0$. 
    It is clear that $\tilde{\tilde{e}}_1(\alpha)<e_1^*\left(\frac{\alpha}{2}\right)$.
         If  $ \tilde{\tilde{e}}_1>0$, the above proof  for the case $\alpha \in (0,  \frac{\pi}{2})$ applies. Indeed, the level set  $I_{2,1}(e_1,e_2)=\alpha$ has a connected component, as plotted as in Figure \ref{connected component}. Hence, any flux $\Phi \in (-\infty, \Phi_{max}^{(2,1)}(\alpha)]$ can be achieved on this component.
      
    We claim that
    \begin{equation}\label{eq:relatisize}
        \tilde{\tilde{e}}_1(\beta)\geq \tilde{\tilde{e}}_1(\alpha), \ \ \ \ \mbox{if }\  0<\beta\leq \alpha< \pi.
    \end{equation}
    It follows from the definition of $ \tilde{\tilde{e}}_1$  that for any $e_1\in(\tilde{\tilde{e}}_1(\alpha), e_1^*\left(\frac{\alpha}{2}\right))$, 
    \begin{equation*}
        \inf_{e_2\in(-\frac{1}{2}e_1-3,0)}I_{2,1}(e_1,e_2)<\alpha.
    \end{equation*}
    If $\tilde{\tilde{e}}_1(\beta)< \tilde{\tilde{e}}_1(\alpha)$, then $\tilde{\tilde{e}}_1(\alpha)\in  (\tilde{\tilde{e}}_1(\beta), e_1^*\left(\frac{\beta}{2}\right))$, where we have used $\tilde{\tilde{e}}_1(\alpha)<e_1^*\left(\frac{\alpha}{2}\right)\leq e_1^*\left(\frac{\beta}{2}\right)$. Hence it follows that
    \begin{equation*}
   \alpha=\inf_{e_2\in(-\frac{1}{2}\tilde{\tilde{e}}_1(\alpha)-3,0)}I_{2,1}(\tilde{\tilde{e}}_1(\alpha),e_2)<\beta\leq \alpha.
    \end{equation*}
    This contradiction implies that \eqref{eq:relatisize} must  hold. With the aid of \eqref{eq:relatisize},  there exists a unique $\alpha^*\in [\frac{\pi}{2}, \pi)$, such that
    \begin{equation}\label{eq:alpha^*}
        \tilde{\tilde{e}}_1(\alpha)>0 \textrm{ if $\alpha<\alpha^*$ and }  \tilde{\tilde{e}}_1(\alpha)=0 \textrm{ if $\alpha^*\leq \alpha<\pi$}.  
    \end{equation}
     It follows from \eqref{eq:alpha^*} that the proof of Proposition \ref{Protype21} for the case $m=1$ is completed.

If $m\geq 2$, then $$\alpha=I_{m+1,m} = (m+1)I_+ + m I_- = mI + I_+.$$ 
It follows from \eqref{eq:convex1} and \eqref{eq:convex2}  that
 $\frac{\partial^2 I_{m+1, m}}{\partial e_2^2} (e_1, e_2) >0$. Next,  define 
\begin{equation*}
    S^m = \left\{ e_1>0: \ \ \inf_{e_2\in(-\frac{1}{2}e_1-3,0)}I_{m+1,m}(e_1,e_2) \geq \alpha  \right\}. 
\end{equation*}
Similar to Step 4 in the proof of Lemma \ref{existencetype21}, we can show that $\tilde{e}_1=\sup S^m \in \left[e_1^*\left(\frac{\alpha}{m}\right), e_1^*\left(\frac{\alpha}{m+1}\right)\right) $.
Hence, $\tilde{e}_1>0$, and one can go through the proof above to conclude that
there is a continuous curve in the level set of $\{(e_1,e_2): I_{m+1,m}(e_1,e_2)=\alpha\}$ starting from $(e_1^*(\frac{\alpha}{m+1}),0)$ and can be extended to $(+\infty,-\infty)$ such that
any $\Phi \in (-\infty, \Phi_{max}^{(m+1,m)}(\alpha)]$ can be achieved on this curve.

(ii) Proof of Part (2). 
Recall that for $m\geq 1$, the maximum flux $\Phi_{max}^{(m+1, m+1)}(\alpha)$ is attained if $(e_1,e_2)=( e_1^*(\frac{\alpha}{m+1}), 0)$. It can also be considered a limiting case of the type $(m+1, m)$ flow. Hence it holds that 
\begin{equation*}
    \Phi_{max}^{(m+1, m)}(\alpha) \geq \Phi_{max}^{(m+1, m+1)}(\alpha).
\end{equation*}
Therefore, the proof for Proposition \ref{Protype21} is completed. 
\end{proof}
 }

\section{Flows with Maximum Fluxes and Further Discussions} 
\label{sum}

This section is mainly devoted to the analysis for the case with maximum fluxes. Our proofs of  Theorems \ref{2Dresult_2} and \ref{main2_2} in fact provide more results than the statements of the two theorems. Theorem \ref{2Dresult_2} pertains to the existence and non-existence of  solution when $\Phi\neq \Phi^{(m_+,m_-)}_{max}(\alpha)$. However, our proofs also explain which type of solution exists when $\Phi=\Phi^{(m_+,m_-)}_{max}(\alpha)$.
\begin{theorem}
\label{limitc}
            \begin{enumerate}
            \item If $\Phi=\Phi^{(1,0)}_{max}(\alpha),0<\alpha<\frac{\pi}{2}$, then there exists an SS solution $\Bu$ of type $(1,0)$ such that $f(\theta):=r\,u^r$ satisfies 
                \begin{equation}
                \label{spc_pro1}
            f'(\alpha)=f'(-\alpha)=0.
                \end{equation}
            \item 
                $
                \Phi^{(0,1)}_{max} (\alpha) 
                \begin{cases}
                    =0, & \text{if }0<\alpha \leq \frac{\pi}{2},
                    \\
                    <\frac{\pi^2}{\alpha}-4\alpha, & \text{if } \frac{\pi}{2} <\alpha <\pi.
                \end{cases}
                $
            \item 
            If $\Phi= \Phi^{(0,1)}_{max}(\alpha), \frac{\pi}{2}<\alpha<\pi$, then there exists an SS solution $\Bu$ of type (0,1) such that $f(\theta):=r\,u^r$ satisfies \eqref{spc_pro1}.
            \item
            $
            \Phi^{(1,1)}_{max} (\alpha)
            =
            \begin{cases}
                \Phi^{(1,0)}_{max} (\alpha), & \text{if }  0<\alpha < \frac{\pi}{2},
                \\
                \Phi^{(0,1)}_{max} (\alpha), & \text{if }\frac{\pi}{2}\leq \alpha < \pi.
            \end{cases}
            $
            \item 
            If $\Phi=\Phi^{(m,m)}_{max}(\alpha),m\geq 2, 0<\alpha<\pi$, then there exists an SS solution $\Bu$ of type $(m,0)$ such that
            $f(\theta)=r\,u^r$ satisfies \eqref{spc_pro1}.

            \item If $\Phi = \Phi_{max}^{(1,2)}(\alpha)$, $0< \alpha \leq \frac{\pi}{2}$, then there exists an SS solution $\Bu$ of type $(1, 2)$.

            \item If $\Phi = \Phi_{max}^{(m,m+1)}(\alpha)$, $m\geq 2$, $0< \alpha < \pi $, then there exists an SS solution $\Bu$ of type $(m, m+1)$.
        \end{enumerate}
\end{theorem}

\begin{remark}
According to Parts (1), (3), (5), if $\Phi=\Phi^{(m_+,m_-)}_{max}(\alpha)$, there exists a solution of a given type with the special property \eqref{spc_pro1}. Given that, Parts (2) and (4) have the following implications.
    \begin{enumerate}[(a)]
        \item (Type (0,1), $0<\alpha \leq \frac{\pi}{2}$) In this case, if $\Phi=\Phi^{(0,1)}_{max}(\alpha)(=0)$, this maximum flux is obviously achieved by the trivial solution $\Bu\equiv 0$. %For this reason, the trivial solution can be regarded as a limiting case of type $(0,1)$  when $0<\alpha \leq \frac{\pi}{2}$.
        
        \item  (Type (0,1), $\frac{\pi}{2}<\alpha <\pi$) In this case, the trivial solution is no longer a limiting case of type $(0,1)$ due to Parts (2) and (3), but instead, there exists a (non-trivial) SS solution of type $(0,1)$ with  \eqref{spc_pro1}, that achieves $\Phi^{(0,1)}_{max}(\alpha)$.

        \item (Type (1,1), $0<\alpha<\frac{\pi}{2}$)  In this case, according to Parts (1) and  (4), if $\Phi=\Phi^{(1,1)}_{max}(\alpha)$, then there exists an SS solution of type $(1,0)$ with  \eqref{spc_pro1}, which achieves $\Phi^{(1,1)}_{max}(\alpha)$. Considering \eqref{spc_pro1}, this solution of type $(1,0)$ in fact can be considered as a limiting case of type $(1,1)$  when $0<\alpha<\frac{\pi}{2}$.

        \item (Type (1,1), $\frac{\pi}{2}\leq \alpha <\pi$) Similarly, according to Parts (3) and (4), 
        if $\Phi=\Phi^{(1,1)}_{max}(\alpha)$, then there exists an SS solution of type $(0,1)$ with  \eqref{spc_pro1}. In fact, this SS solution of type $(0,1)$ can be considered a limiting case of flows of type $(1,1)$  in the case $\alpha \in [\frac{\pi}{2}, \pi)$.

        \item (Type (m,m), $m\geq 2$, $0<\alpha <\pi$) Even in this case, according to Part (5), the SS solution of type $(m,0)$ with \eqref{spc_pro1}, that achieves $\Phi^{(m,m)}_{max}(\alpha)$, can be considered a limiting case of type $(m,m)$.
    \end{enumerate}
\end{remark}

{
On the other hand, for $\alpha \in (0, \frac{pi}{2})$,  according to Part (6) of Theorem \ref{limitc}, the maximal flux of flows of type $(1,2)$ is achieved at some points $(e_1,e_2)$ satisfying $e_1 \not=0$ and $e_2\not =0$. Hence this flow does not satisfy \eqref{spc_pro1}. One can get a similar conclusion for type $(m,m+1)$ according to Part (7).
}

\begin{proof}[Proof of Theorem \ref{limitc}]
As Theorem \ref{limitc} is in fact already proved in the proofs of  Theorems \ref{2Dresult_2} and \ref{main2_2}, we just indicate where each part of Theorem \ref{limitc} is proved. Part (1) was proved in \eqref{eq:maxifluxpureout} where $e_2=0$.
Part (2) for the case $0<\alpha \leq \frac{\pi}{2}$ was proved in \eqref{eq:j-}, and Part (2) for the case $\frac{\pi}{2}<\alpha<\pi$ as well as Part (3) were proved in Part (2) of the proof of Proposition \ref{Lemmapureinflow2}. Part (4) for the case $0<\alpha< \frac{\pi}{2}$ was proved near \eqref{eq1312} in  the proof of Proposition \ref{Lemmaperiodicflow1}, while Part (4) for the case $\frac{\pi}{2}\leq \alpha < \pi$ was also proved in the proof of Proposition \ref{Lemmaperiodicflow1}. Lastly, using \eqref{eq1410} and \eqref{eq1414}, one can see that the proof for Part (5) is similar to that for Part (4). Hence the proof for Part (5) is in fact omitted. Part (6) was proved in the proof  of Proposition \ref{Lemmatype12-5} for the case $m=1$. The proof for  Part (7) is similar to that for Part (6).  
\end{proof}

\section{Asymptotic behavior of solutions to  2D NS in an aperture domain}\label{sectionaperturedomain}
As we mentioned in Section \ref{secintroduction}, 
Galdi, Padula, and Solonnikov (cf. \cite{GaldiPadulaSolonnikov96}) investigated the Navier-Stokes system \eqref{SNS} in an aperture domain defined in \eqref{defapture}.
%which is 
%\begin{equation}\label{aperture}
 %   \Omega =\{ x=(x_1,x_2) \in \mathbb{R}^2:\ \mbox{either}\ x_2 \neq 0 \ \textrm{ or } \ x_1\in (-d, d) \}. 
%\end{equation}
%They considered a symmetric solution $\Bv$, that is,
It was proved in \cite{GaldiPadulaSolonnikov96} that there exists at least a symmetric solution $\Bu$, i.e, it satisfies
 \begin{align}
    \label{sym_22}
        u_1(x_1, -x_2)=u_1(x_1, x_2), \ \ u_2(x_1, -x_2)=- u_2 (x_1, x_2), 
    \end{align}
    which 
decays at large distances with rate $|\Bx|^{-1}$ when $|\Phi|$ is small. Moreover, it is interesting that the leading term of the solution in the far field is a Jeffery-Hamel solution. 
More precisely, the following theorem was established in \cite{GaldiPadulaSolonnikov96}. 
\begin{theorem}\label{TheromeGPS}\cite[Lemma 5.1, Theorem 5.1]{GaldiPadulaSolonnikov96}
\begin{enumerate}[(1)]
\item (Jeffery-Hamel solutions)
     If $|\Phi| < \frac{1}{36}$, there exists a solution of the Navier-Stokes system \eqref{SNS} in the half-plane $H =H_{\pm}$ defined by $H_\pm =\{ x\in \mathbb{R}^2: \pm x_1>0\} $ such that 
    \begin{equation}\label{H1}
    \begin{aligned}
        \Bu =\frac{f(\theta)}{r} \Be_r,
        \quad
         & \int_{-\frac{\pi}{2}}^{\frac{\pi}{2}} f(\theta)\, d \theta = \Phi, \ \ \ f\left(\pm \frac{\pi}{2}\right) = 0, \ \ \ f(\theta)= f(-\theta) \textrm{ in } H_{\pm},\\
         % & \int_{-\frac{\pi}{2}}^{\frac{\pi}{2}} f(\theta)\, d \theta = \Phi, \ \ \ f\left(\pm \frac{\pi}{2}\right) = 0, \ \ \ f(\theta)= f(-\theta) \textrm{ in } H_-.
          \end{aligned}
    \end{equation}
    Here, when considering \eqref{H1} on $H_-$, we make the abuse of the notation that the negative $x_1$-axis corresponds to $\theta=0$.
    Moreover, it holds that
    \begin{equation}\label{H3}
         \max_{\theta\in [-\frac{\pi}{2}, \frac{\pi}{2}]} |f(\theta)| \leq 6 |\Phi|,
    \end{equation}
    and $\displaystyle \max_{\theta\in [-\frac{\pi}{2}, \frac{\pi}{2} ]}|f^\prime (\theta)| \leq 28 |\Phi|. $ The function $f$ is unique in the class of solutions satisfying \eqref{H1} and \eqref{H3}. 
\item (Asymptotic behavior)
    Let the aperture
domain $\Omega$ be as in \eqref{defapture}. Then there exists a $\Phi_1>0$ such that if $|\Phi|< \Phi_1$, the stationary Navier-Stokes system \eqref{SNS} with the following conditions
  \begin{equation*}
  % \left\{
  \begin{array}{ll}
       % -\Delta \Bv + (\Bv \cdot \nabla)\Bv + \nabla p =0,\ \ \ \ \mbox{in}\ \Omega,  \\[2mm]
       %     {\rm div}~\Bv=0,\ \ \ \ \mbox{in}\ \Omega, \\[2mm]
        \displaystyle \Bu=0 \ \ \ \mbox{on}\ \partial \Omega, \ \ \ \int_{-d}^d u_1(0,x_2) \, dx_2 =\Phi, \ \ \ \lim_{|x|\to \infty} \Bu(x)= 0,
    \end{array} 
    % \right.
    \end{equation*}
    admits at least one symmetric solution  $\Bu \in W_{loc}^{1, 2}(\overline{\Omega})$ satisfying \eqref{sym_22},  $\nabla \Bu\in L^2(\Omega)$, 
    and 
    \begin{equation*}
        \sup_{x\in \Omega^R} |x|^{1+\sigma} |\Bu(x) - \bar{\Bu}(x)| \leq  c(\sigma,R)\Phi <\infty. 
    \end{equation*}
    Here $\sigma \in (0,1)$, $\Omega^R:=\Omega \setminus B_R$, $R>2d$ and
    \begin{equation*}        
   \bar{\Bu}(x) = \left\{ \begin{array}{l} \bar{\Bu}_{+}(x), \ \ \ x\in H_+, \\
     \bar{\Bu}_{-}(x), \ \ \ x\in H_-, \end{array}
    \right. 
     \end{equation*}     
     and $\bar{\Bu}_{\pm}$ is the Jeffery-Hamel flow in $ H_{\pm}$ corresponding to  fluxes $\pm \Phi$, given  in Part (1).
\end{enumerate}
\end{theorem}

In this section, we give a clear description of the type of Jeffery-Hamel solution that appeared in Theorem \ref{TheromeGPS} as follows. 
\begin{pro}\label{Prophalfplane}
  For $\Phi \in (0, \frac{1}{36})$, the  unique Jeffery-Hamel solution $\bar{\Bu}_{+}$ satisfying \eqref{H1}-\eqref{H3} indicated in Part (2) of Theorem \ref{TheromeGPS} is an SS solution of type $(1, 2)$. Meanwhile, there exists a constant $\Phi_0 \in (0, \frac{1}{36})$, for any $\Phi \in (-\Phi_0, 0)$, the unique Jeffery-Hamel solution   $\bar{\Bu}_{-}$ in Part (2) of Theorem \ref{TheromeGPS} satisfying \eqref{H1}-\eqref{H3} is an SS solution of type $(0, 1)$. 
\end{pro}

\begin{proof}
According to Part (2) of Theorem \ref{2Dresult_2}, the Jeffery-Hamel solution cannot be of type $(1,0)$ because there exists no self-similar solution of type $(1,0)$ when $\alpha=\frac{\pi}{2}$. In addition, it cannot be of type $(1,1)$ when $\Phi\neq 0$ because a symmetric self-similar solution of type $(1,1)$ has zero flux (see \eqref{sym_22} for the definition of symmetry). 

Next, we rule out the possibility of types $(m, m)$, $(m, m+1)$, $(m+1, m)$, $(m,0)$, $m\geq 2$ for ${\bar{\Bu}}_{\pm}$. 
Suppose ${\bar{\Bu}}_{\pm}$ corresponds to the triple of roots $(e_1, e_2, e_3)$ and is of the type $(m, m)$, $(m, m+1)$, $(m+1, m)$, or $(m,0),$ $m\geq 2$. 
First of all, it follows  from \eqref{H3} that
\begin{equation}\label{eq:maxe1e2}
    \max\{e_1, -e_2\} =  \max_{\theta\in [-\frac{\pi}{2}, \frac{\pi}{2}]} |f(\theta)| \leq 6 |\Phi|<\frac{1}{6}.
\end{equation}
Then using \eqref{periodic-1}, we have 
\begin{equation*}
  I(e_1, e_2) = \frac{\sqrt{6}K( \bar{\gamma})}{\sqrt{e_1 -e_3}} \leq \frac{1}{m} \frac{\pi}{2} \leq \frac{\pi}{4}.
\end{equation*}
Assume that $e_1^*(\frac{\pi}{4})$ is the root satisfying 
\begin{equation*}
    I_{+}\left(e_1^*\left(\frac{\pi}{4}\right), 0\right) = \frac{\pi}{4}. 
\end{equation*}
We have $e_1 \geq e_1^*(\frac{\pi}{4})$. Suppose it was not correct, that is, $0< e_1 < e_1^*(\frac{\pi}{4})$. Since $e_2\leq 0$, by virtue of \eqref{nonunique-4} and  Lemma \ref{monotone}, 
\begin{equation*}
    I(e_1, e_2) \geq I(e_1, 0) =I_{+}(e_1, 0) > I_{+}\left(e_1^*\left(\frac{\pi}{4}\right), 0\right)  = \frac{\pi}{4}.
\end{equation*}
This leads to a contradiction, and hence $e_1 \geq e_1^*(\frac{\pi}{4})$. On the other hand, 
\begin{equation}\label{est2269}
    I_{+}\left(\frac16, 0\right) = \int_0^{\frac16} \frac{df}{\sqrt{-\frac23 (f-\frac16) f (f+ \frac{37}{6})}} > \int_0^1 \frac{3\, dg}{\sqrt{38 (1-g)g }} =\frac{3}{\sqrt{38}}\pi > \frac{\pi}{4}.  
\end{equation}
It follows from Lemma \ref{monotone} and \eqref{est2269} that $e_1 \geq e_1^*(\frac{\pi}{4}) > \frac16$. 
% On the one hand, $\max|f(\theta)|\geq e_1>\frac{1}{6}$.
% On the other hand, $\max|f(\theta)|$ satisfies \eqref{H3}, and hence $\max|f(\theta)|\leq \frac{1}{6}$ if $|\Phi|< \frac{1}{36}$. 
This contradicts with \eqref{eq:maxe1e2}, and then $\bar{\Bu}_{\pm}$ cannot
be of type $(m, m)$, $(m, m+1)$, $(m+1, m)$, $(m,0)$ with $m\geq 2$
if
$|\Phi|< \frac{1}{36}$.  

Next, we rule out the possibility of type $(2, 1)$. Suppose that the solution corresponds to the triple of roots $(e_1, e_2, e_3)$ and is of the type $(2, 1)$. 
Note that we have
   \begin{equation}
   \label{est2270}
  I(e_1, e_2) = \frac{\sqrt{6}K( \bar{\gamma})}{\sqrt{e_1 -e_3}} < \frac{\pi}{2}.  
\end{equation}
Meanwhile,  $K( \bar{\gamma})\geq \frac{\pi}{2}$ for $ \bar{\gamma} \geq 0$. Therefore, using \eqref{est2270}, one can obtain
\begin{equation*}
    6<e_1 - e_3 =  6+ 2e_1 + e_2 , 
\end{equation*}
which says 
\begin{equation}\label{new4}
   0<-e_2< 2 e_1. 
\end{equation}

Note that 
\begin{equation*}
\begin{array}{ll}
     I_{2, 1}(e_1, e_2) & = I(e_1, e_2) + I_{+}(e_1, e_2)\\
     & \displaystyle = \frac{\sqrt{6}K( \bar{\gamma})}{\sqrt{e_1 - e_3}} + \int_0^{e_1}\frac{df}{\sqrt{-\frac23 (f-e_1) (f-e_2) (f-e_3)}}\\
     & \displaystyle = \frac{\sqrt{6} K( \bar{\gamma})}{\sqrt{6+ 2e_1 + e_2}} + \int_0^1 \frac{dg}{\sqrt{\frac23 (1-g) ( g- \frac{e_2}{e_1}) ( e_1 g + 6 + e_1 + e_2)}}.
    \end{array}
\end{equation*}
Since $e_1\in (0,\frac{1}{6})$ and $e_2\in (-\frac{1}{6},0)$ due to \eqref{eq:maxe1e2},  one has
\begin{align*}
 \frac{\sqrt{6} K( \bar{\gamma})}{\sqrt{6+ 2e_1 + e_2}} > \frac{\sqrt{6}K(0)}{\sqrt{6+ \frac13}} > 0.97\cdot \frac{\pi}{2}.
\end{align*}
Meanwhile, using the property $0<-e_2 < 2e_1$ yields
\begin{align*}
    \int_0^1 \frac{dg}{\sqrt{\frac23 (1-g) ( g- \frac{e_2}{e_1}) ( e_1 g + 6 + e_1 + e_2)}} & \geq \int_0^1 \frac{dg}{\sqrt{\frac23 (1-g) (g + 2) \cdot \frac{19}{3} }} \\
    & = \frac{3}{\sqrt{38}} \cdot 2 \arctan\frac{\sqrt{2}}{2} >0.32 \cdot \frac{\pi}{2}. 
\end{align*}
This contradicts the fact that $I_{2, 1}(e_1, e_2) =\frac{\pi}{2}.$

There are only two possibilities left. One is of type $(0, 1)$ and the other is of type $(1, 2)$. Of course, the type $(0, 1)$ has a negative flux. Finally, we will prove that the type $(1, 2)$ flow has the positive flux when $(e_1, e_2)$ is close to $(0, 0)$. 
Assume that $(e_1, e_2, e_3)$ provides a flow of type $(1, 2)$, satisfying that 
\begin{equation*}
    I_{1, 2}(e_1, e_2) = \frac{\pi}{2}. 
\end{equation*}
 As proved in Lemma \ref{Lemmatype12-1}, for each $e_1>0$, $e_2$ is uniquely determined by $e_1$, and we denote it by $e_2(e_1)$. We will prove that 
 \begin{equation*}
     J_{1, 2}(e_1, e_2(e_1)) > 0,
 \end{equation*}
 when $e_1>0$ and is close to $0$. In fact, we have the following lemma.

\begin{lemma}\label{Lemmahalfspace-2}
    It holds that 
    \begin{equation*}
     \lim_{e_1\to 0^+} \frac{e_2(e_1) }{e_1} = 0, \ \ \ \ \mbox{and}\ \ \ \     \lim_{e_1\to 0^+} \frac{J_{1,2}(e_1, e_2(e_1))}{e_1} =\frac{\pi}{2}>0.  
    \end{equation*}
\end{lemma}

We first use Lemma \ref{Lemmahalfspace-2} to finish the proof of Proposition  and then prove Lemma \ref{Lemmahalfspace-2}.

Lemma \ref{Lemmahalfspace-2} shows that the flow of type $(1, 2)$ has positive flux when $ \max\{|e_1|, |e_2|\} $ is small. Now we can conclude that the unique Jeffery-Hamel symmetric flow obtained in Theorem \ref{TheromeGPS} is of type $(0, 1)$ when the flux $\Phi<0$ and $|\Phi|$ is small, while the flow is of type $(1, 2)$ when $0< \Phi < \frac{1}{36}$. Hence the proof for Proposition \ref{Prophalfplane} is completed. 
\end{proof}

\begin{proof}[Proof of Lemma \ref{Lemmahalfspace-2}]
First, we prove that
    \begin{equation}\label{e1e2-new}
        \lim_{e_1\to 0^+} \frac{e_2(e_1) }{e_1} = 0. 
    \end{equation}
Following the same proof for \eqref{new4}, one has 
    \begin{equation}\label{e1e2}
        0< -e_2(e_1) < 2e_1.
    \end{equation}

Since 
\begin{equation*}
    I_{1,2}(e_1, e_2(e_1)) = I(e_1, e_2(e_1)) + I_{-}(e_1, e_2(e_1))
\end{equation*}
and 
\begin{equation*}
    \lim_{e_1\to 0^+, e_2 \to 0^-}I(e_1, e_2) = I(0, 0) = K(0)=\frac{\pi}{2}, 
\end{equation*}
it holds that
\begin{equation}\label{H5}
    \lim_{e_1 \to 0^+}I_{-}(e_1, e_2(e_1)) = 0. 
\end{equation} 
Meanwhile, 
\begin{align}  
\label{eq2356}
    I_{-}(e_1, e_2) & = \int_{e_2}^0 \frac{df}{\sqrt{-\frac23 (f-e_1)(f-e_2) (f-e_3)}} \nonumber \\
& = \frac{\sqrt{-e_2}}{\sqrt{e_1}} \int_0^1 \frac{  \, dg }{   \sqrt{\frac23 \left(1 - \frac{e_2}{e_1 } g \right) (1-g) (e_2g + 6 + e_1 +e_2)}}.
\end{align}
It follows from \eqref{e1e2} that for $\delta>0$ small enough and $e_1\in (0,\delta)$, $e_2\in(-2\delta,0)$. Hence there exists a $C(\delta)>0$ such that
\begin{align*}
 \frac{1}{C(\delta)} \leq   \int_0^1 \frac{  \, dg }{   \sqrt{\frac23 \left(1 - \frac{e_2}{e_1 } g \right) (1-g) (e_2g + 6 + e_1 +e_2)}}\leq C(\delta).
\end{align*}
This,  together with \eqref{e1e2}, \eqref{H5}, and \eqref{eq2356}, implies that 
\begin{equation*}
    \lim_{e_1 \to 0^+ } \frac{\sqrt{-e_2(e_1)}}{\sqrt{e_1}} = 0
\end{equation*}
and proves \eqref{e1e2-new}. 

Consider that 
\begin{equation}\label{H7}
    J_{1, 2}(e_1, e_2) =  \underline{J_{+}(e_1, e_2) - J_{+}(e_1, 0) } + \underline{J_{+}(e_1, 0)- J_{+}(0, 0) }+ \underline{ 2J_{-}(e_1, e_2)}. 
\end{equation}
Herein, as shown in the proof of Lemma \ref{Lemmanonunique-4}, 
\begin{align*}
    \frac{\partial}{\partial e_2}J_{+}(e_1, e_2) & = \int_0^{e_1} \frac{(3+ e_2 + \frac12 e_1) f\, df }{\sqrt{\frac23 (e_1 -f)} [(f-e_2) (f+ 6+ e_1 + e_2)]^{\frac32}}  \\
    & = \int_0^1 \frac{ (3+ e_2 + \frac12 e_1) g\, dg }{ \sqrt{\frac23(1-g)} \left[ ( g -\frac{e_2}{e_1}) (e_2 g + 6 + e_1 +e_2) \right]^{\frac32}} . 
\end{align*}
Hence it holds that
\begin{equation*}
\begin{aligned}
     \lim_{e_1\to 0^+, \, \frac{e_2}{e_1}\to 0} \frac{\partial}{\partial e_2} J_{+}(e_1, e_2) 
     = \int_0^1 \frac{3 \, dg}{\sqrt{\frac23(1-g) g} \cdot 6^{\frac32}} = \int_0^1 \frac{ dg}{4\sqrt{(1-g)g}}=\frac{\pi}{4}>0.
\end{aligned}
 \end{equation*}
Similarly, for the second part of \eqref{H7}, same as \eqref{eq:deriofJ+}, we have
\begin{align*}
    \frac{d}{de_1}J_{+}(e_1, 0) = \int_0^1 \frac{\frac12 e_1 g^2 +\frac12 e_1 g + 6g}{\sqrt{\frac23 (1-g) g (e_1 g + 6 +e_1)^3}} \, dg.
\end{align*}
This yields
\begin{equation*}
    \lim_{e_1 \to 0^+ }\frac{J_{+}(e_1, 0) - J_{+}(0, 0)}{e_1} = \int_0^1 \frac{ \sqrt{g}\, dg}{\sqrt{\frac23 (1-g) \cdot 6}} =\frac{\pi}{4}>0. 
\end{equation*}
For the last part of \eqref{H7}, when $(e_1, e_2)$ is close to $(0, 0)$, one has
\begin{equation*}
\begin{aligned}
      \left| J_{-}(e_1, e_2) \right| &= \left| \int_{e_2}^0 \frac{f\, df}{\sqrt{\frac23 (e_1-f) (f-e_2)(f+6+e_1 +e_2)}}  \right|
    \leq C   \left| \int_{e_2}^0 \frac{f \, df }{\sqrt{(e_1-f)(f-e_2)}}   \right|\\
  & \leq C \frac{\sqrt{-e_2}}{\sqrt{e_1}}\left| \int_{e_2}^0 \frac{\sqrt{-f} \, df }{\sqrt{f-e_2}}   \right| \leq C \cdot (-e_2). 
\end{aligned}
\end{equation*}
Combining the above estimates, we finish the proof for Lemma \ref{Lemmahalfspace-2}. 
\end{proof}

Theorem \ref{main3} follows from Theorem \ref{TheromeGPS} and Proposition \ref{Prophalfplane}. 

\medskip

{\bf Acknowledgement.}
The research of Gui is supported by University of Macau research grants  CPG2024-00016-FST, CPG2025-00032-FST, SRG2023-00011-FST, MYRG-GRG2023-00139-FST-UMDF, UMDF Professorial Fellowship of Mathematics, Macao SAR FDCT 0003/2023/RIA1 and Macao SAR FDCT 0024/2023/RIB1. The research of Wang is partially supported by NSFC grants 12171349 and 12271389, the Natural Science Foundation of Jiangsu Province Grants No. BK20240147. The research of Xie is partially supported by NSFC grants 12250710674 and  12161141004, the Fundamental Research Funds for the Central Universities, and Program of Shanghai Academic Research Leader 22XD1421400. The authors thank Professor Vladim\'ir \v{S}ver\'ak for many helpful discussions.

\medskip
	{\bf Statement and declaration.} The authors have no relevant financial or non-financial interests to disclose. No datasets were generated or analyzed during the current study.

\bibliographystyle{abbrv}

\end{document}